%% file: main.tex
\documentclass{article}
\usepackage[utf8]{inputenc}
\usepackage[colorlinks = true,
            linkcolor = blue,
            urlcolor  = blue,
            citecolor = blue,
            anchorcolor = blue]{hyperref}
\usepackage{amsmath,amssymb,amsfonts,amsthm,bm}
\usepackage{color,graphicx}
\usepackage[a4paper, total={6.5in,8.5in}]{geometry}

\usepackage{amssymb}
\usepackage{empheq}
\usepackage{mathtools}
\usepackage{graphicx}
\usepackage{bbm}
\usepackage{hyperref}
\usepackage{dirtytalk}
\usepackage{stmaryrd}
\usepackage{fancyhdr}
\usepackage{macros}

\usepackage{microtype}
\usepackage{graphicx}
\usepackage{subcaption}
\usepackage{booktabs}
\usepackage{makecell}
\usepackage{caption}
\usepackage{amsmath}
\usepackage{booktabs} 
\usepackage{tablefootnote}
\usepackage{footnote}
\usepackage{hyperref}
\usepackage{tabularx}
\usepackage[ruled,vlined]{algorithm2e}
\usepackage{epigraph}

\usepackage{verbatim}
\usepackage{listings}
\usepackage{algpseudocode}
\usepackage[affil-it]{authblk} 
\usepackage{textcomp}
\usepackage{multirow}
\usepackage[toc,page]{appendix}
\usepackage[numbers, sort, comma, square]{natbib}
\usepackage{float}
\usepackage{todonotes}

\newif\ifshownotes
\shownotestrue

\newif\ifarxiv
\arxivfalse

\usepackage{longtable}
\date{}
\newtheoremstyle{named}{}{}{\itshape}{}{\bfseries}{.}{.5em}{\thmnote{#1: #3}}
\theoremstyle{named}

\author[]{Konstantin Donhauser}
\author[]{Mingqi Wu}
\author[]{Fanny Yang}
\affil[]{Department of Computer Science, ETH Zürich}

\title{How rotational invariance of common kernels prevents generalization in high dimensions}

\begin{document}

\maketitle
\setcounter{page}{1}

\input{sections/new_section1}
\input{sections/new_section2}

\input{sections/new_section3}
\input{sections/new_section4}

\input{sections/new_section5}

\bibliographystyle{icml2021}
\bibliography{bibliography}

\appendix

\input{sections/Appendix/Appendix_Experiments}
\input{sections/Appendix/Appendix_RKHS_Norm}

\input{sections/Appendix/Appendix_Theorem1}

\input{sections/Appendix/Appendix_Theorem2}

\input{sections/Appendix/Appendix_Kernels}
\input{sections/Appendix/Appendix_Technical_Preliminaries}

\end{document}

%% file: sections/new_section1.tex
\begin{abstract}

  Kernel ridge regression is well-known
to achieve minimax optimal rates in low-dimensional settings. However,
its behavior in high dimensions is much less understood.  Recent work
establishes consistency for kernel regression under certain
assumptions on the ground truth function and the distribution of the
input data. In this paper, we show that the rotational invariance
property of commonly studied kernels (such as RBF, inner product
kernels and fully-connected NTK of any depth) induces a bias towards
low-degree polynomials in high dimensions. Our result implies a lower
bound on the generalization error for a wide range of distributions
and various choices of the scaling for kernels with
different eigenvalue decays.  This lower bound suggests that general
consistency results for kernel ridge regression in high dimensions
require a more refined analysis that depends on the structure of the kernel beyond its eigenvalue decay.

\end{abstract}

\section{Introduction}
Traditional analysis establishes good generalization properties of
kernel ridge regression when the dimension $d$
is relatively small compared to the number of samples $n$.
These minimax optimal and consistency results however become less
powerful for modern data sets with large $d$ close to $n$.
High-dimensional asymptotic theory
\cite{Buhlmann11,Wainwright19} aims to fill this gap by providing
bounds that assume $d,n \to \infty$ and are often much more predictive
of practical observations even for finite $d$. 


While recent work \cite{Dobriban15,Hastie19,Bartlett20}
establishes explicit asymptotic upper bounds for the bias
and variance for high-dimensional linear regression, the results for
kernel regression are less conclusive in the regime $d/n^{\beta}  \to c$ with $\beta \in (0,1)$. In particular, even though several papers~\cite{Ghorbani19, Ghorbani20,Liang20b} show that the variance decreases with the dimensionality of the data, the bounds on the bias are somewhat inconclusive.
On the one hand,  \citet{Liang20b} prove asymptotic consistency for ground truth functions 
with asymptotically bounded Hilbert norms for neural tangent kernels (NTK) and inner product (IP) kernels. In contrast, \citet{Ghorbani19, Ghorbani20} show that for uniform distributions on the product of two spheres, consistency cannot be achieved \emph{unless} the ground truth is a low-degree polynomial. 
This \emph{polynomial approximation barrier} can also be observed for random feature and neural tangent regression \cite{Mei19,Mei21,Ghorbani19}.

Notably, the two seemingly contradictory consistency results hold for
different distributional settings and are based on vastly different proof
techniques. While \cite{Liang20b} proves consistency for general input distributions
including isotropic Gaussians,
the lower bounds in the papers \cite{Ghorbani19, Ghorbani20} are limited to
data that is uniformly sampled from the product of two spheres.
Hence, it is a natural question to ask whether the
polynomial approximation barrier is a more general phenomenon or
restricted to the explicit settings studied in \cite{Ghorbani19, Ghorbani20}.
Concretely, this paper addresses the following question
\begin{center}
  \emph{ Can we overcome the polynomial approximation barrier
    when considering different high-dimensional input distributions,
     eigenvalue decay rates or scalings of the kernel function?}
\end{center}
We unify previous
distributional assumptions in one proof framework
and thereby characterize how the rotational
invariance property of common kernels induces a bias towards low-degree polynomials.
Specifically, 
we show that the \emph{polynomial approximation barrier} persists for
\begin{itemize}
\item  a broad range of common rotationally invariant kernels
  such as radial basis functions
(RBF) with vastly different eigenvalue decay rates, inner product kernels and 
NTK of any depth~\cite{Jacot20}.
\item general input distributions including anisotropic Gaussians
  where the degree of the polynomial depends only on the growth of $\effdim := \trace(\inputcov)/\opnorm{\inputcov}$
  and not on the specific structure of $\inputcov$. In particular, we cover the distributions studied in previous related works \cite{Ghorbani19,Ghorbani20,Liang20b,Liang20a,Liu20}.
  \item different scalings $\bw$ with kernel function $\kernf_{\bw}(\xk,\xkd) = \kernf(\frac{\xk}{\sqrt{\bw}},\frac{\xkd}{\sqrt{\bw}})$  beyond the classical choice of $\bw \asymp \effdim$.
\end{itemize}

As a result, this paper demonstrates that the polynomial approximation barrier is a general high-dimensional phenomenon for rotationally invariant kernels, restricting the set of functions for which consistency can at most be reached to low degree polynomials. 

Rotationally invariant kernels are a natural choice if no prior
information on the structure of the ground truth is available, as they
treat all dimensions equally. Since our analysis covers a broad range
of distributions, eigenvalue decays and different scalings,
our results motivate future work to focus on the symmetries
respectively asymmetries of the kernel incorporating prior knowledge
on the structure of the high-dimensional problem,
e.g. \cite{Arora19,Shankar20,Mei21b}.


This paper is organized as follows.
First of all, we show in Section~\ref{sec:priorkernelwork} 
that the bounded norm assumption that previous consistency results \cite{Liang20b,Liang20a} rely on is violated as $d\to \infty$ even for simple functions
such as $\ftrue(x) = e_1^\top x$.  We then introduce our generalized setting in Section~\ref{sec:2} and present our main results in Section~\ref{sec:3} where we show a lower bound on the bias that increases with the dimensionality of the data.  Finally, in Section~\ref{sec:exp} we empirically illustrate how the bias dominates the risk in high dimensions and therefore limits the performance of kernel regression. As a result, we argue that it is crucial to incorporate prior knowledge of the ground truth function (such as sparsity) in high-dimensional kernel learning even in the noiseless setting and empirically verify this on real-world data.

%% file: sections/new_section2.tex
\section{Problem setting}
\label{sec:2}

In this section, we briefly introduce kernel regression estimators in
reproducing kernel Hilbert spaces and subsequently introduce our assumptions on 
the kernel, data distribution and high-dimensional regime.


\subsection{Kernel ridge regression}
We consider nonparametric regression in a \emph{reproducing kernel
  Hilbert space} (RKHS, see e.g. \cite{Wahba90,Smola98})
with functions on the domain $\XX \subset \R^d$ induced by a positive
semi-definite kernel $\kernf: \XX \times \XX \to \R$. That is, for any
set of input vectors $\{\xind{1}, \cdots,\xind{m}\}$ in $\XX$, the
empirical kernel matrix $\kermat$ with entries $\kermat_{i,j} =
\kerfunc{\xind{i},\xind{j}}$ is positive semi-definite.  We denote by
$\langle.,.\rangle_{k}$ the corresponding inner product of the Hilbert
space and by $\| . \|_{\Hilbertspace} := \sqrt{\langle.,.\rangle_{k}}$
the corresponding norm.



We observe tuples of input vectors and response variables $(x,y) $
with $x \in \XX$ and $y \in \R$.  Given $\numobs$ samples,
we consier the ridge regression estimator 
\begin{equation}
\begin{split}
\label{eq:krr}
    \fhatridge = \underset{f \in \Hilbertspace}{\arg\min} ~~\sum_{i=1}^n \left(\y_i - f(\xind{i})\right)^2 + \lambda \| f \|^2_{\Hilbertspace},
    \end{split}
\end{equation} 
with $\lambda>0$ and the minimum norm interpolator (also called the kernel \say{ridgeless} estimate) 
 \begin{equation}
 \label{eq:constraint_min_p}
     \begin{split}
         \hat{f}_0 = \underset{f \in \Hilbertspace}{\arg\min} ~ \| f \|_{\Hilbertspace} ~~ \textrm{such that}~~ \forall i:~f(\xind{i}) = \y_i
     \end{split}
 \end{equation}
 that can be obtained as the limit of the ridge estimate $\hat{f}_0
 =\lim_{\lambda\to 0} \fhatridge$ for fixed $n$. It is well-known that the ridge
 estimator can attain consistency as $n\to\infty$ for some sequence of
 $\lambda$ such that $\frac{\lambda}{n} \to 0$. Recently,
 some works 
 \cite{Liang20b,Liang20a,Ghorbani19,Ghorbani20} have also analyzed the
 consistency behavior of ridgeless estimates motivated by the curiously good
 generalization properties of neural networks with zero training
 error.  

For evaluation, we assume that the observations are i.i.d. samples
from a joint distribution ${(\xind{i}, y_i)_{i=1}^n \sim
\jointp{X}{Y}}$ and refer to $\ftrue(x) := \EE [y \mid X= x]$ as
the \emph{ground truth} function that minimizes the population square
loss $\EE (Y-f(X))^2$.  We evaluate the estimator using the population
risk conditioned on the input data $\datamatrix$
\begin{align}
    \label{eq:biasvariance}
    \begin{split}
    &\Risk(\fhatridge) \define \EEobs \|\fhatridge - \ftrue\|_{\Ell_2(\PP_X)}^2 
    = \underbrace{\|  \EEobs \fhatridge - \ftrue  \|^2_{\Ell_2(\PP_X)} }_{=: \text{ Bias }\Bias}+ \underbrace{ \EEobs \| \EEobs \fhatridge - \fhatridge \|^2_{\Ell_2(\PP_X)}}_{=: \text{ Variance } \Var}. \nonumber
        \end{split}
\end{align}
where $\EEobs$ is the conditional expectation over the observations $\y_i \sim
\PP(Y|X=\x_i)$. In particular, when $y = \ftrue(x) + \eps$, $\EEobs
\fhatridge$ is equivalent to the noiseless estimator with
$\eps=0$. 

Note that consistency in terms of
$\Risk(\fhatridge)\to 0$ as $n \rightarrow \infty$ can only be reached
if the bias vanishes.  In this paper, we lower bound the bias $\Bias$
which, in turn, implies a lower bound on the risk and the inconsistency of the estimator. The theoretical results in
Section~\ref{sec:3} hold for both ridge regression and minimum
norm interpolation. However, it is well known that the ridge penalty controls the bias-variance trade-off, and hence we are primarily interested in lower bounds of the bias for the minimum
norm interpolant.

\subsection{Prior work on the consistency of kernel regression}
\label{sec:priorkernelwork}

For ridge regression estimates in RKHS, a rich body of work
shows consistency and rate optimality when
appropriately choosing the ridge parameter $\lambda$
both in the non-asymptotic setting, e.g.~\cite{Caponnetto07}, and the classical asymptotic
setting, e.g.~\cite{Christmann07}, as $n \to \infty$.  

Similar results have also been shown for high-dimensional asymptotics, where recent papers on minimum norm
interpolation in kernel regression \cite{Liang20b,Liang20a} explicitly show how the bias vanishes when 
the ground truth function has bounded Hilbert norm as $d,n\to\infty$.
Even though this
assumption is perfectly reasonable for a fixed ground truth and
Hilbert space, its plausibility is less clear for a sequence of functions as $d\to \infty$.\footnote{Although \citet{Liu20} replace the bounded Hilbert norm
  assumption with a weaker bounded source condition, we
  expect this condition not to hold with increasing dimension either.
  We defer the detailed discussion to future work.}
After all, the Hilbert space and thus also the norm change with
$\dim$.
In fact, we now show that even innocuous functions have diverging
Hilbert norm as the dimension increases. 
In particular, for tensor product kernels including exponential inner product kernels (also studied in \cite{Liang20b,Liang20a}) defined on $x,x'\in\XX^{\otimes d} \subset \RR^d$, 
we can show the following lemma. 

\begin{lemma}[Informal]
  \label{lm:rkhs}
For any $f$ that is a non-constant sparsely parameterized product
function \\${f(\xk) = \prod_{j=1}^m f_j(\xel{j})}$ for some fixed $m \in
\N_+$, 
\begin{equation*}
  \| f \|_{\Hilbertspace_d} \overset{d \to \infty}{\to} \infty.
\end{equation*}
\end{lemma}

\begin{figure}
    \centering
         \minipage{0.35\textwidth}
         \includegraphics[width=\linewidth]{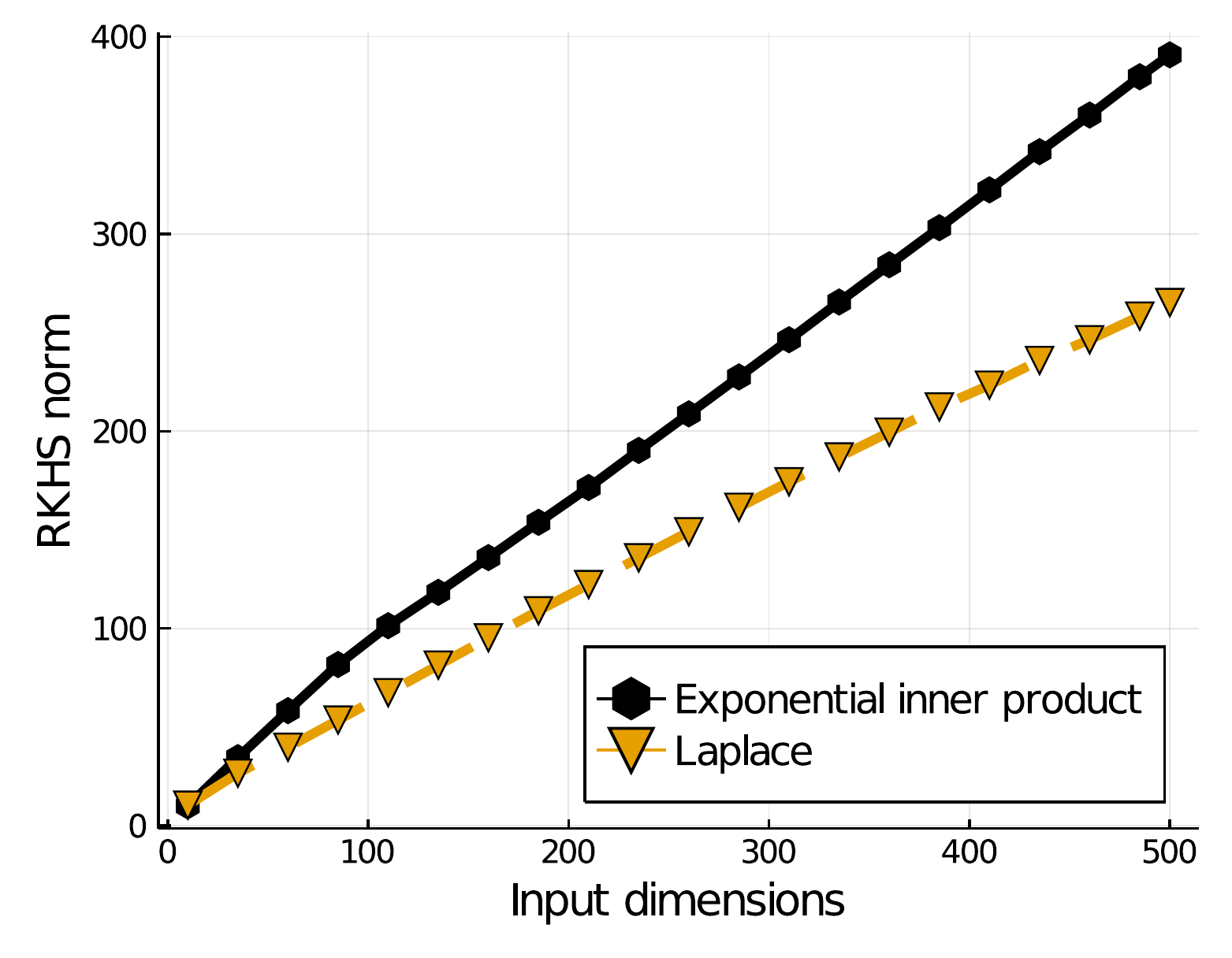}
    \endminipage
    \caption{\small{The approximation of the Hilbert norm induced by the Laplace and exponential inner product kernels of the function $\ftrue(\x) = \x_1$ plotted with respect to the dimension $d$ on the space $\XX = [0,1]^d$. See Section ~\ref{sec:expnorm} for experimental details.}}
    \label{fig:rkhsnorm}
\end{figure}

In words, for simple sequences of sparse product functions, the Hilbert norm
diverges as the dimension $d\to\infty$. The precise conditions on the
kernel and sequence $\Hilbertspace_d$ of induced Hilbert spaces can be found
in~\suppmat~\ref{sec:norm}. Figure~\ref{fig:rkhsnorm} illustrates this
phenomenon for $\fexp (x) = e_1^\top x$ for the Laplace and
exponential inner product kernel.
%

The discussion so far implies that generalization upper bounds that rely on the bounded Hilbert norm assumption become void even for simple ground truth functions. A natural follow-up question is hence: Do we actually fail to learn sparsely parameterized functions consistently or is it simply a loose upper bound?
A recent line of work by \citet{Ghorbani19, Ghorbani20}
shows that kernel
regression estimates can indeed only consistently learn polynomials of
degree at most $\frac{\log \numobs}{\log \effdim}$ as $d,n \to \infty$ (which we refer to
as the \emph{polynomial approximation barrier}).
While the results provide some intuition for the behavior of kernel
regression, the proofs heavily rely on significant simplifications
that hold for the specific distributional assumptions on the
sphere.  It is a priori unclear whether they apply to more
general settings including the ones considered in~\cite{Liang20b}.
In the next sections, relying on a different proof technique we show that the polynomial approximation barrier indeed holds for a broad spectrum of data distributions that also capture the distributions studied in the papers~\cite{Liang20b,Liang20a} and for an entire range of eigenvalue decay rates of the kernel functions (e.g. polynomial and exponential decay rates) and choices of the scaling $\bw$. As a consequence, our results suggest that the polynomial approximation barrier is strongly tied to the rotational invariance of the kernel function and not specific to the settings studied so far. 


\subsection{Our problem setting}
\label{sec:oursetting}


The framework we study in this paper covers random vectors that are
generated from a covariance matrix model, i.e. $\Xrv = \inputcov^{1/2}
\standardrv$ with vector $\standardrv$ consisting of i.i.d entries and their projections onto the $d-1$-dimensional unit sphere.
We now specify all relevant assumptions on
the kernel and data distribution.

\paragraph{Kernels}

Throughout this paper, we focus on continuous and rotationally invariant
kernels. They include a vast majority of the commonly used kernels such as fully connected NTK, RBF and inner product kernels as they only depend on the squared
Euclidean norm of $\xk$ and $\xkd$ and the inner product
$\xk^\top\xkd$.
We first present one set of assumptions on the kernel functions that are sufficient for
our main results involving local regularity of $\kernf$ around the
sphere where the data concentrates in high dimensions.

\begin{itemize}
\item[\Akonebracket] \textit{Rotational invariance and local power series expansion}: 
  The kernel function $\kernf$ is rotationally invariant and there is a function $\rotkerfunc$ such that $k(\xk.\xkd) = \rotkerfunc(\|\xk\|_2^2, \|\xkd\|_2^2, \xk^\top \xkd)$. Furthermore, $\rotkerfunc$ can be expanded as a power series of the form
  \begin{equation}
    \label{eq:kerneldefg}
    \begin{split}
    \rotkerfunc(\|\xk\|_2^2, \|\xkd\|_2^2, \xk^\top \xkd) = \sum_{j=0}^{\infty} \rotkerfunccoeffj{j}(\|\xk\|_2^2,\|\xkd\|_2^2) (\xk^\top \xkd)^j 
    \end{split}
   \end{equation}
  that converges for $\xk, \xkd$ in a neighborhood of the sphere
  $\{x\in \R^d \mid \|x\|^2 \in [1-\delta,1+\delta]\}$
  for some
   $\delta > 0$. Furthermore, all $g_i$ are positive semi-definite
   kernels.

\item[\Aktwobracket] \textit{Restricted Lipschitz
  continuity:} The restriction of $k$ on $\{(x,x)| x \in \R^d,
 \| x \|_2^2 \in [1-\delta_L, 1+\delta_L]\}$ is
  a Lipschitz continuous function for some constant $\delta_L>0$.
\end{itemize}

We show in Corollary~\ref{cor:kernels} that the kernels for which our
main results hold cover a broad range of commonly studied kernels in
practice. In particular, Theorem~\ref{thm:main_thm} also holds for $\alpha$-exponential
kernels defined as $k(\xk,\xkd) = \exp(- \|\xk -\xkd\|_2^\alpha ) $
for $\alpha \in (0,2)$ even though we could not yet show that they satisfy Assumptions~\Akone-\Aktwo. In \suppmat~\ref{sec:proofmainthm}, we show how the
proof of Theorem~\ref{thm:main_thm} crucially relies on the rotational invariance
Assumption \Ckone~and the fact that the eigenvalues of the kernel
matrix $\kermat$ are asymptotically lower bounded by a positive
constant. Both conditions are also satisfied by the $\alpha$-exponential
kernels and the separate treatment is purely due to different
proof technique used to lower bound the kernel matrix eigenvalues.

\paragraph{Data distribution and scaling} We impose the following assumptions on the data distribution.

\begin{itemize}
\item[\Adonebracket] \textit{Covariance model}: We assume that the input data distribution is from one of the following sets
 \begin{equation}
 \begin{split}
    \probsetcov &= \{\prob_X \mid X = \inputcov^{\frac{1}{2}} \standardrv \text{ with } \standardrv_{(i)} \overset{\textrm{i.i.d.}}{\sim} \prob ~\forall i=1, ... ,  d, \prob \in \probsetw \} \\
  \probsetsphere &= \{\prob_X \mid X = \sqrt{\effdim} \frac{Z}{\|Z\|} \text{ with } Z \sim \prob \in \probsetcov\} \nonumber
 \end{split}
 \end{equation}
 where $\inputcov \in \R^{d\times d}$ is a positive semi-definite covariance matrix and the effective dimensions $\effdim$ is defined as $\effdim \define
  \trace(\inputcov)/\opnorm{\inputcov}$.
 The entries of the random vector $\standardrv$ are sampled i.i.d. from a distribution in the set $\probsetw$, containing the standard normal distribution and any zero mean and unit variance distributions with bounded support.

\item[\Adtwobracket] \textit{High dimensional regime}: We assume that
  the effective dimension grows with the sample size $\numobs$ s.t.
  $\effdim/ n^{\beta} \to c$ for some
  $\beta,c >0$.
\end{itemize}
In words, when $\prob_X \in \probsetcov$, the data has covariance
$\inputcov$, and when $\prob_X \in \probsetsphere$, the data can be generated by
projecting $Z\sim \prob_Z \in \probsetcov$ onto the sphere of radius
$\sqrt{\effdim}$. Unlike \citet{Ghorbani19,Ghorbani20}, we do not require the random vectors $x_i$ to be \textit{uniformly} distributed on the sphere. Our distributional assumptions include and generalize previous
works. A comparison can be found in Table~\ref{tab:settings}.
In the sequel, we assume without loss of generality that for simplicity $\opnorm{\inputcov}=1$ and hence $\effdim =
  \trace(\inputcov)$. 

In our analysis, the kernel function $\rotkerfunc$ does not change for any $d$. However as $\effdim, n\to\infty$ we
need to adjust the scaling of the input as the norm concentrates
around $\EE \| \x \|_2^2 = \effdim$.  
Hence, as $\numobs
\to\infty$, we consider the sequence of scale dependent kernels
\begin{equation}
  \label{eq:scalekernel}
  \sckerfunc{\bw}{x}{x'} = \rotkerfunc\left(\frac{\|\xk\|_2^2}{\bw},
  \frac{\|\xkd\|_2^2}{\bw}, \frac{ \xk^\top \xkd}{\bw}\right)
\end{equation}
 and parameterize the scaling by sequence of parameters $\bw$ dependent on $\numobs$. In Section~\ref{sec:tauconst} we study
the standard scaling $\frac{\bw}{\effdim}\to c>0$, 
before discussing $ \frac{\bw}{\effdim} \to 0$ and $
\frac{\bw}{\effdim}\to \infty$ respectively in
Section~\ref{sec:otherbw}, where we show that the polynomial approximation barrier is not a consequence of the standard scaling.

\begin{figure*}[t]
  \begin{center}
     \begin{subfigure}{0.32\textwidth}
       \includegraphics[width=\linewidth]{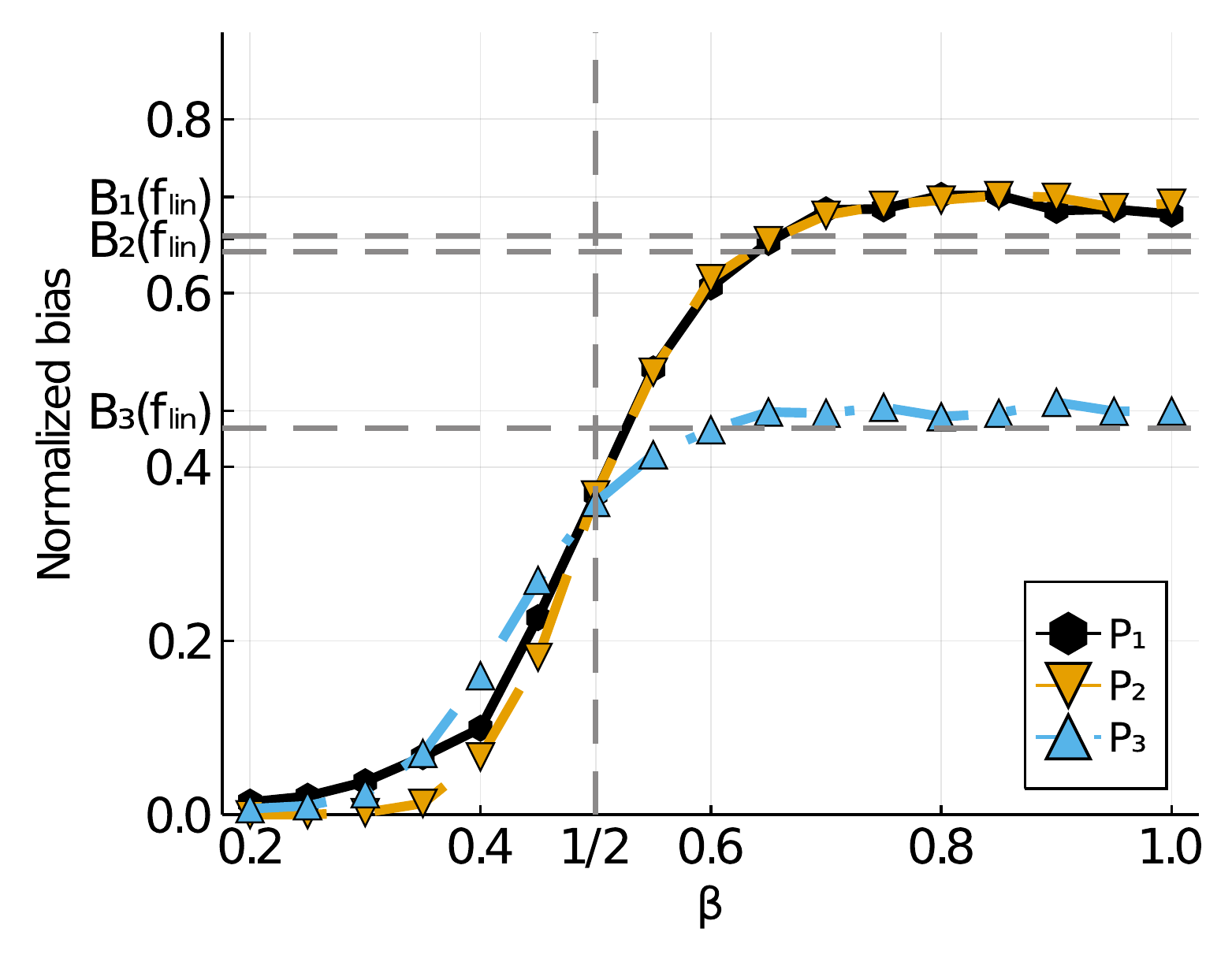}
       \caption{Ground truth $\ftrue = 2 \xel{1}^2$}
       \label{fig:fig2a}
     \end{subfigure}
     \begin{subfigure}{0.32\textwidth}
       \includegraphics[width=\linewidth]{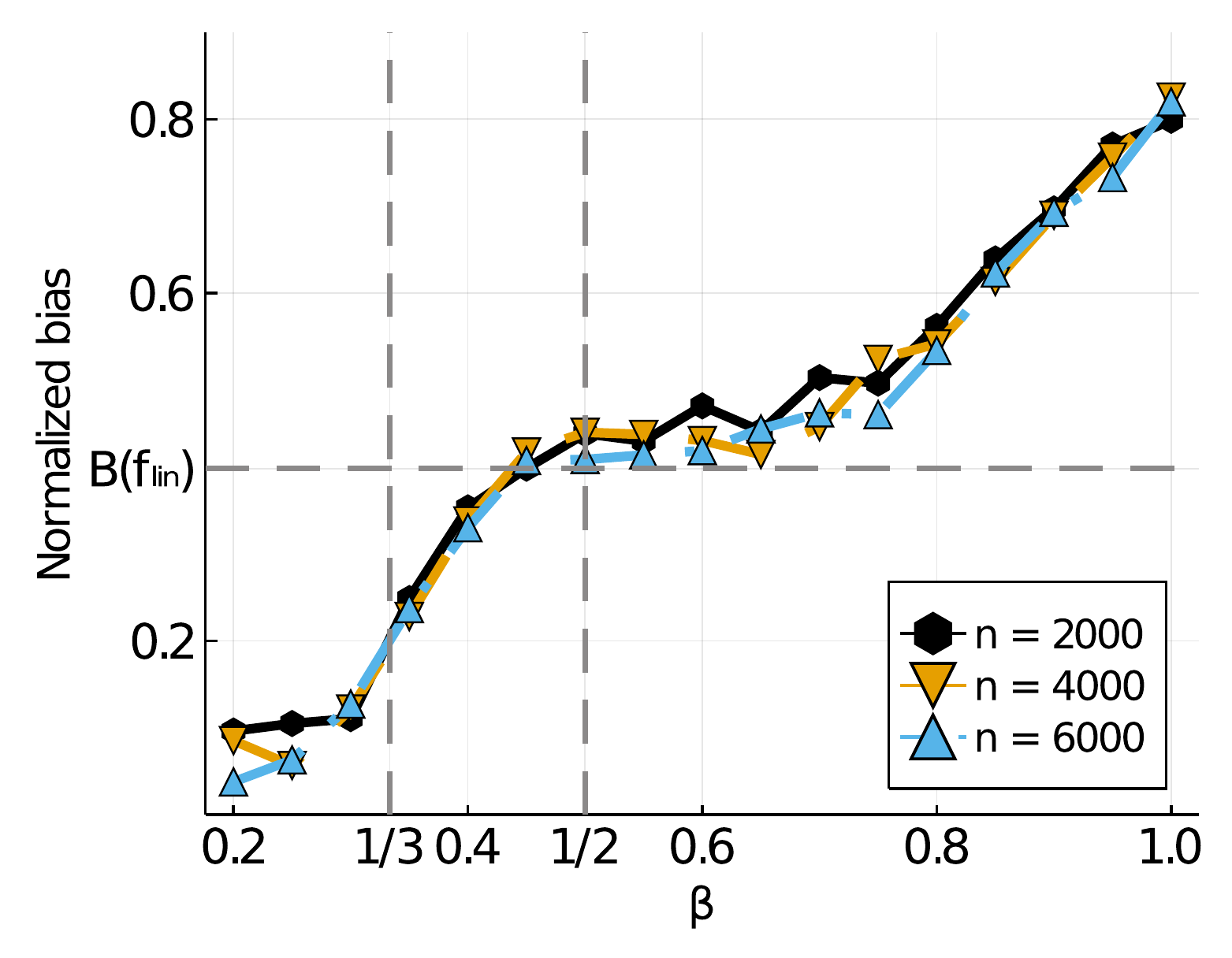}
       \caption{Ground truth $\ftrue = 2 \xel{1}^3$}
       \label{fig:fig2b}
     \end{subfigure}
          \begin{subfigure}{0.32\textwidth}
       \includegraphics[width=\linewidth]{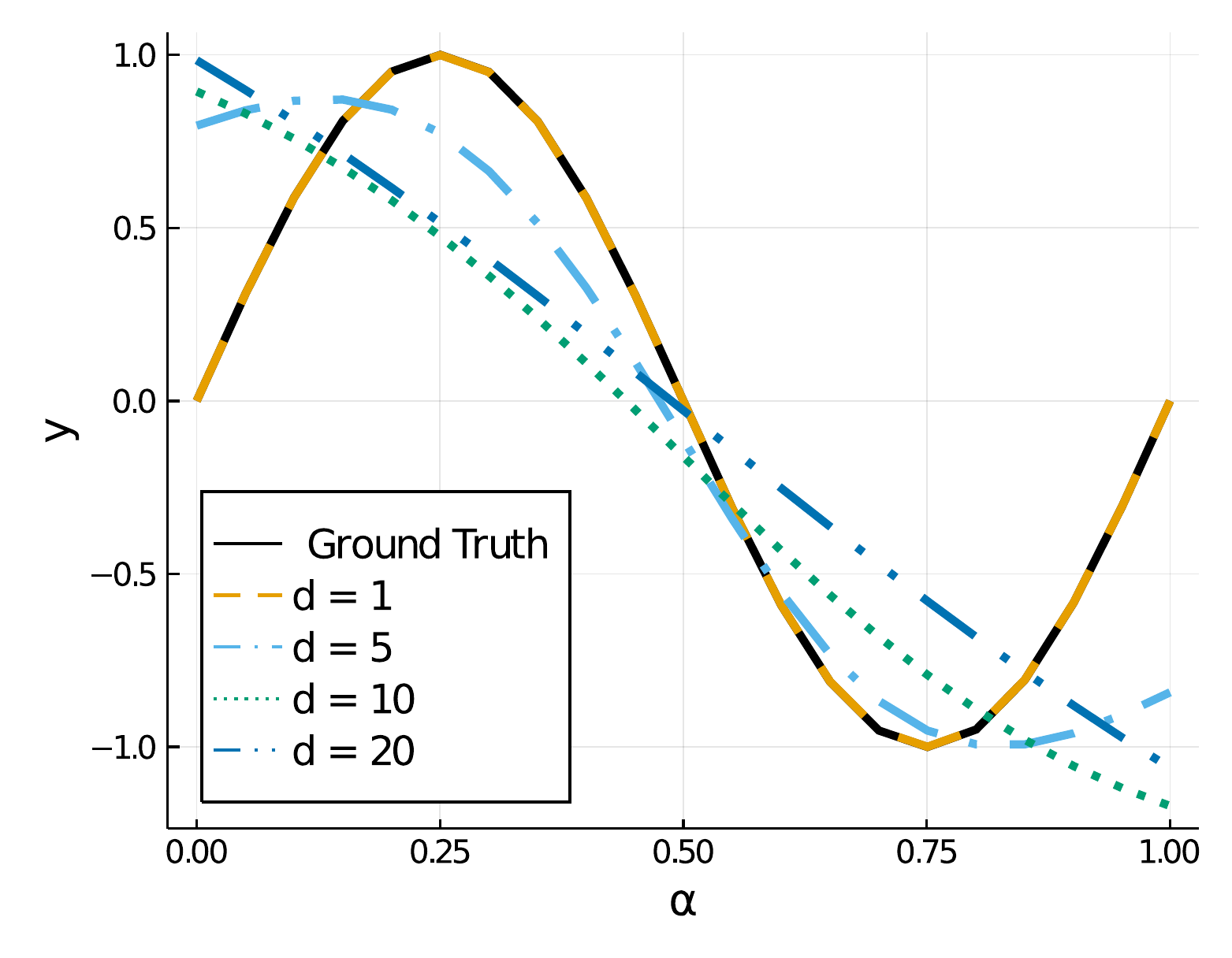}
       \caption{Degeneration of estimator}
       \label{fig:sine}
     \end{subfigure}
     \end{center}

\caption{\small{(a) and (b): The bias of the minimum norm interpolant
    $\Bias(\fhatinterpol)$ normalized by $\Bias(0)$ as a function of
    $\beta$ for (a) different
    covariance models $\prob_1$- $\prob_3$  (see Section \ref{sec:expsynthetic}) with
    $n=4000$ and (b) different choices of $n$ and samples generated from an isotropic Gaussian with $d = \floor{n^{\beta}}$ ($\prob_1$). The horizontal lines $\textbf{B}(f_{\textrm{lin}})$ correspond to the risk of the optimal
    linear model for the different input data distributions. (c): The minimum norm interpolator estimate $\fhatinterpol$ of  $\ftrue(\x) = \sin(2\pi
    \xel{1})$  plotted in the
    direction $(0,1/2,\cdots,1/2)+ \alpha \ev_1$ when fitting noiseless observations with covariates drawn uniformly from $[0,1]^d$ with $n=100$ and varying $d$.
}}

\label{fig:polymodels}
\end{figure*}

%% file: sections/new_section3.tex
\section{Main Results}
\label{sec:3}

We now present our main results that hold for a wide range of
distributions and kernels and show that
kernel methods can at most consistently learn low-degree polynomials.
Section~\ref{sec:tauconst} considers the case $\bw \asymp \effdim$
while Section~\ref{sec:otherbw} provides lower bounds for
the regimes $\frac{\bw}{\effdim} \to 0$ and $\frac{\bw}{\effdim} \to\infty$.


\subsection{Inconsistency of kernel regression for $\bw \asymp \effdim$}
\label{sec:tauconst}

For simplicity, we present a result for the case $\bw=\effdim$ based on the assumptions \Akone-\Aktwo. The more general case $\bw\asymp\effdim$ follows from the exact same arguments.
In the sequel we denote by $\polyspace{m}$ the space of polynomials of
degree at most $m \in \N$.



\begin{theorem}[Polynomial approximation barrier]
  \label{thm:main_thm}
  Assume that the kernel $\kernf$ respectively its restriction onto the sphere satisfies \Akone-\Aktwo~or is an $\alpha$-exponential kernel.
  Furthermore assume that the input distribution $\prob_X$ satisfies Assumptions \Adone-\Adtwo~and that the ground truth function $f^*$ is bounded.
  Then, for some $m \in \NN$ specified below, the following results hold for both the ridge~\eqref{eq:krr} and ridgeless estimator~\eqref{eq:constraint_min_p} $\fhatridge$ with $\lambda \geq 0$.
  \begin{enumerate}

  \item The bias of the kernel estimators $\fhatridge$ is
    asymptotically almost surely lower bounded for any $\eps > 0$,
  \begin{equation}
    \label{eq:biaslb}
    \Bias(\fhatridge) \geq \underset{\polyf \in
      \polyspace{m}}{\inf} \| \ftrue -
    \polyf \|_{\Ell_2(\prob_X)} - \eps ~~~a.s. ~\mathrm{ as }~ n \to \infty.
  \end{equation}
\item For bounded kernel functions on the support of $\prob_X$ the averaged estimator $\EEobs \fhatridge$ converges almost surely in $\Ell_2(\prob_X)$ to a polynomial $\polyf  \in \polyspace{m}$,
  \begin{equation}
  \label{eq:polyapproxlp}
    \left\| \EEobs \fhatridge - \polyf \right\|_{\Ell_2(\prob_X)} \to 0
    ~~ a.s. ~\mathrm{ as }~ n \to \infty.
  \end{equation}
  \end{enumerate}
  More precisely, if $g_i$ is $(\floor{2/\beta}+1-i)$-times
  continuously differentiable in a neighborhood of $(1,1)$ and there
  exists $\jthresh > \floor{2/\beta}$ such that
  $\rotkerfunc_{\jthresh}(1,1) >0$, then the
  bounds \eqref{eq:biaslb},\eqref{eq:polyapproxlp} hold with $m = 2
  \floor{2/\beta}$ for $\prob_X
  \in \probsetcov$ and $m = \floor{2/\beta}$ for $\prob_X
  \in \probsetsphere$.

\end{theorem}

Corollary~\ref{cor:kernels} gives examples for kernels that satisfy the
assumptions of the theorem.
The almost sure statements refers to the sequence of matricies $\Xs$ of random vectors $\x_i$ as $n \to \infty$, but also hold true with probability $\geq 1- n^2\exp(-\Cgenprob{\epsd}
   \log(n)^{(1+\epsilon')})$ over the draws of $\X$ (see
Lemma ~\ref{lm:innerpr} for further details). 

The first statement in Theorem \ref{thm:main_thm} shows that even with
noiseless observations, the estimator
$\fhatridge$ can at most consistently learn ground truth functions $\ftrue$
that are a polynomial of degree less than $m$. 
We refer to $m$ as the $\beta$-dependent
\emph{polynomial approximation barrier}. Figure \ref{fig:fig2a} and \ref{fig:fig2b} illustrate this barrier on synthetic datasets drawn from different input data distributions. 
While the first statement only implies that the
risk does not vanish if $\fstar$ is not a polynomial,
the second part of Theorem~\ref{thm:main_thm} explicitly states that the
averaged estimator $\EEobs \fhatridge$ converges in
$\Ell_2(\prob_X)$ to a polynomial of degree at most $m$ when the
kernel is bounded on the support of
$\prob_X$.\footnote{In~\suppmat~\ref{sec:prop1_proof} we provide a
  closed form expression for the polynomial in the proof of the
  Theorem \ref{thm:main_thm}. Furthermore, it is there straight forward to verify that the second statement in Theorem \ref{thm:main_thm} also applies for the estimate $\fhatridge$ instead of the averaged estimate $\EEobs \fhatridge$ given that the observations $y_i$ are bounded.}  We refer
to Theorem \ref{prop:main_thm_ext} in ~\suppmat~\ref{sec:proofmainthm} for slightly weaker statements
which also apply to unbounded kernels. Figure \ref{fig:sine}
illustrates how the estimator degenerates to a linear function as dimension grows.

The attentive reader might notice that \citet{Ghorbani20} achieve a lower
barrier $m=\floor{1/\beta}$ for their specific setting which
implies that our results are not tight. However, this is
not the main focus in this work as we primarily intend
to demonstrate that the polynomial approximation barrier persists for general covariance model data distributions (Ass. \Bdone-\Bdtwo) and hence asymptotic consistency in high-dimensional regimes is at most reachable for commonly used rotationally invariant kernels if the ground truth function is a low degree-polynomial.  We leave tighter bounds as interesting future work. Finally, we remark that as $\beta \to 0$ we enter again a classical asymptotic regime where $n$ is much larger compared to $d$ and hence $m \to \infty$. This underlines the difference between classical and high-dimensional asymptotics and shows that our results are only meaningful in the latter. 


We now present a short proof sketch to provide intuition for why
the polynomial approximation barrier holds. The full proof can be found in~\suppmat~\ref{sec:proofmainthm}.
  

\paragraph{Proof sketch}
The proof of the main theorem is primarily based
 on the concentration of Lipschitz continuous functions of vectors with i.i.d~entries. In particular, we show in Lemma \ref{lm:innerpr} that
\begin{equation}
\label{eq:concentration}
    \underset{i}{\max}\frac{\left| \xind{i}^\top \Xrv  \right|}{\bw} \leq n^{-\beta/2}(\log n)^{(1+\epsilon)/2}  ~~~a.s. ~\mathrm{ as }~ n \to \infty,
\end{equation}
where we use $\trace(\inputcov) \asymp \numobs^\beta$. Furthermore, 
Assumption \Akone and hence the rotational invariance of the kernel function $\kernf$, implies that for inner product kernels where $g_j$ are constants,
\begin{equation}
    \kernf_{\bw}(\xind{i},\Xrv) = \sum_{j=0}^{m} g_j ~
    \left(\frac{\xind{i}^\top\Xrv}{\bw}\right)^j 
    +
    O\left(n^{-\theta} \right) ~~ a.s. ~\mathrm{ as }~ n \to \infty
 \end{equation}
with $\theta$ some constant such that $1<\theta <(m+1)\frac{\beta}{2}$
that exists because $m\geq \floor{2/\beta}$. Hence, as $n\to\infty$,
$ \kernf_{\bw}(\x_i, X)$ converge to low-degree polynomials.
Using the closed form solution of $\fhatridge$ based on the
representer theorem we can hence conclude the first statement in Theorem
\ref{thm:main_thm} if $\kermat +\lambda I \succ cI$ for some
constant $c>0$.  The result follows naturally for ridge regression with non vanishing $\lambda>0$. However for the
minimum norm interpolator, we need to show that the eigenvalues
of the kernel matrix $\kermat$ themselves are asymptotically lower bounded
by a positive non-zero constant. This follows from the additional assumption in Theorem \ref{thm:main_thm} and the observation that $(\X^\top \X)^{\circ \jthresh} \to \idmat_n$ in operator norm with $\circ$ being the Hadamard product. Finally, the case where $g_j$ depend on $\|x_i\|_2^2$ and $\|x_j\|_2^2$ requires a more careful analysis and constitute the major bulk of the proof in ~\suppmat~\ref{sec:proofmainthm}. \qed
\\

The assumptions in Theorem \ref{thm:main_thm} cover a broad range of
commonly used kernels, including the ones in previous works. The
following corollary summarizes some relevant special cases
\begin{cor}
\label{cor:kernels}
Theorem \ref{thm:main_thm} applies to
\begin{enumerate}
 \item The exponential inner product kernel
 \item The $\alpha$-exponential kernel for $\alpha \in (0,2]$, including Laplace ($\alpha =1$) and the Gaussian ($\alpha =2$) kernels
 \item The fully-connected NTK of any depth with regular activation functions including the ReLU activation $\sigma(x) = \max(0,x)$
\end{enumerate}
\end{cor}

The precise regularity conditions of the activation functions for the
NTK and the proof of the corollary can be found
in~\suppmat~\ref{sec:proofcrkernels}. 

\subsection{Inconsistency of kernel interpolation for $\bw \not\asymp \effdim$}
\label{sec:otherbw}

As Section~\ref{sec:tauconst} establishes the polynomial approximation
barrier for the classical scaling $\frac{\bw}{\effdim} \asymp 1$, an
important question remains unaddressed: can we avoid it with a different
scaling?  When $\frac{\bw}{\effdim} \to 0$, intuitively the estimates
converge to the zero function almost everywhere and hence the bias is
lower bounded by the $\Ell_2(\prob_x)$-norm of $\ftrue$. This implies that
no function can be learned consistently with this scaling (see~\suppmat~
\ref{sec:simplertaucase} for a rigorous statement).  When $\bw$
increases faster than $\effdim$, however, the behavior is unclear a
priori. Simulations in Figure~\ref{fig:taulaplace}, \ref{fig:tauexpo}
suggest that the bias could in fact decrease for $\bw \gg \effdim$ and
attain its minimum at the so called \emph{flat limit}, when $\bw \to
\infty$.  To the best of our knowledge, the next theorem is the first
to show that the polynomial approximation barrier persists in the flat
limit for RBF kernels whose eigenvalues do not decay too fast.

\begin{figure*}[t]
\begin{subfigure}{0.32\linewidth}
    \centering
    \includegraphics[width=\textwidth]{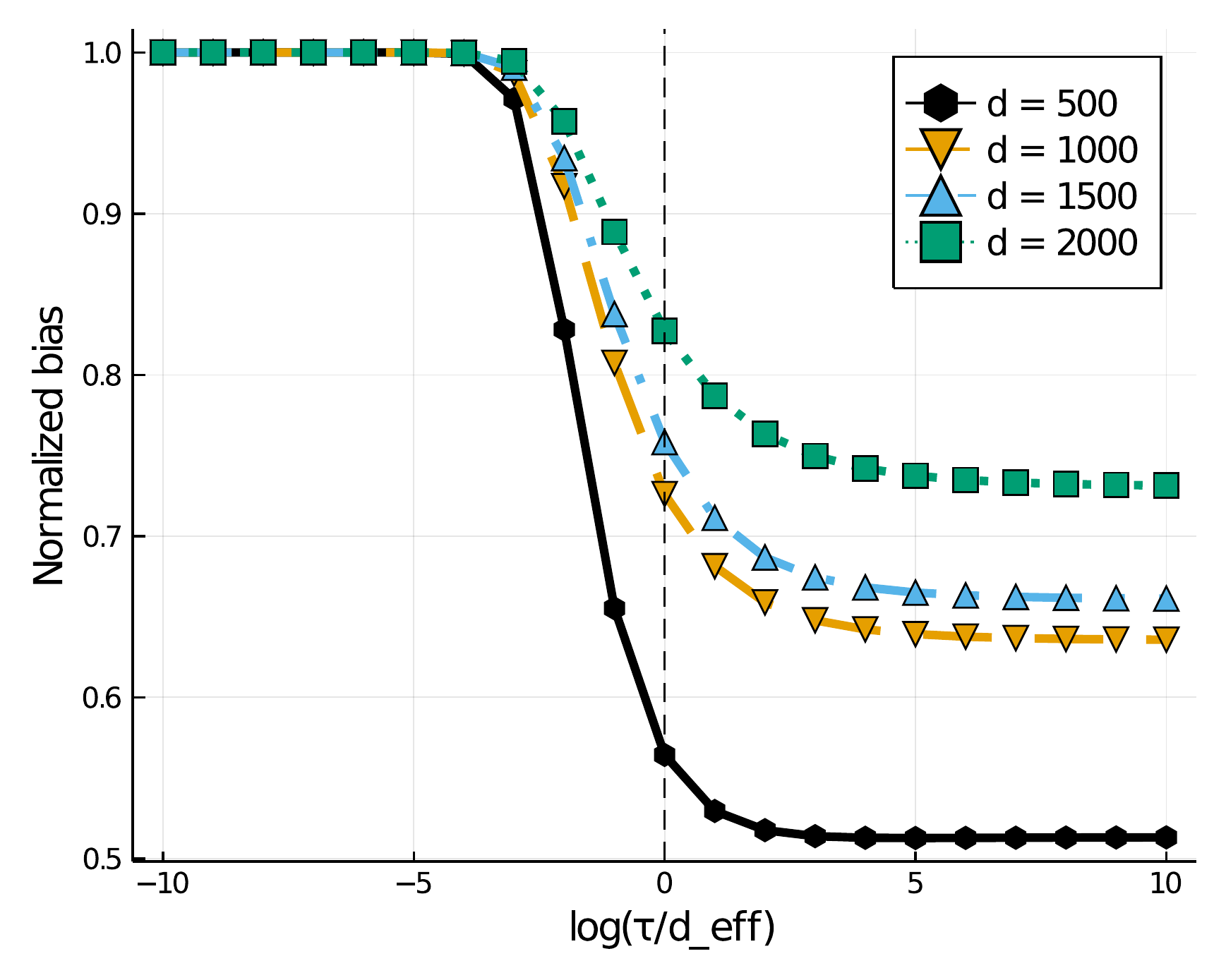}
    \caption{Laplacian kernel}
    \label{fig:taulaplace}
  \end{subfigure}
  \hfill
  \begin{subfigure}{0.32\linewidth}
    \centering
    \includegraphics[width=\textwidth]{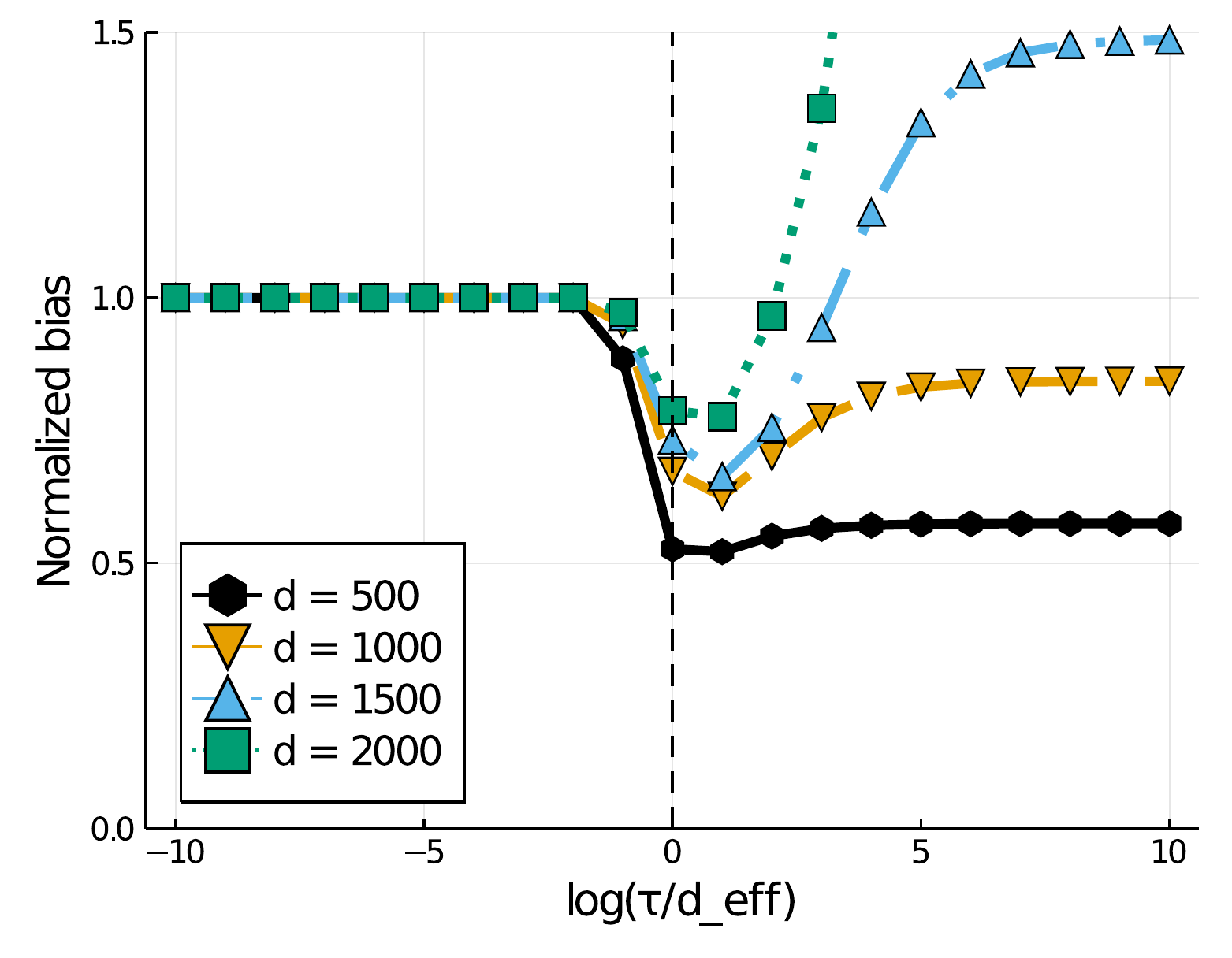}
    \caption{Gaussian kernel}
    \label{fig:taugauss}
  \end{subfigure}
  \hfill
    \begin{subfigure}{0.32\linewidth}
    \centering
    \includegraphics[width=\textwidth]{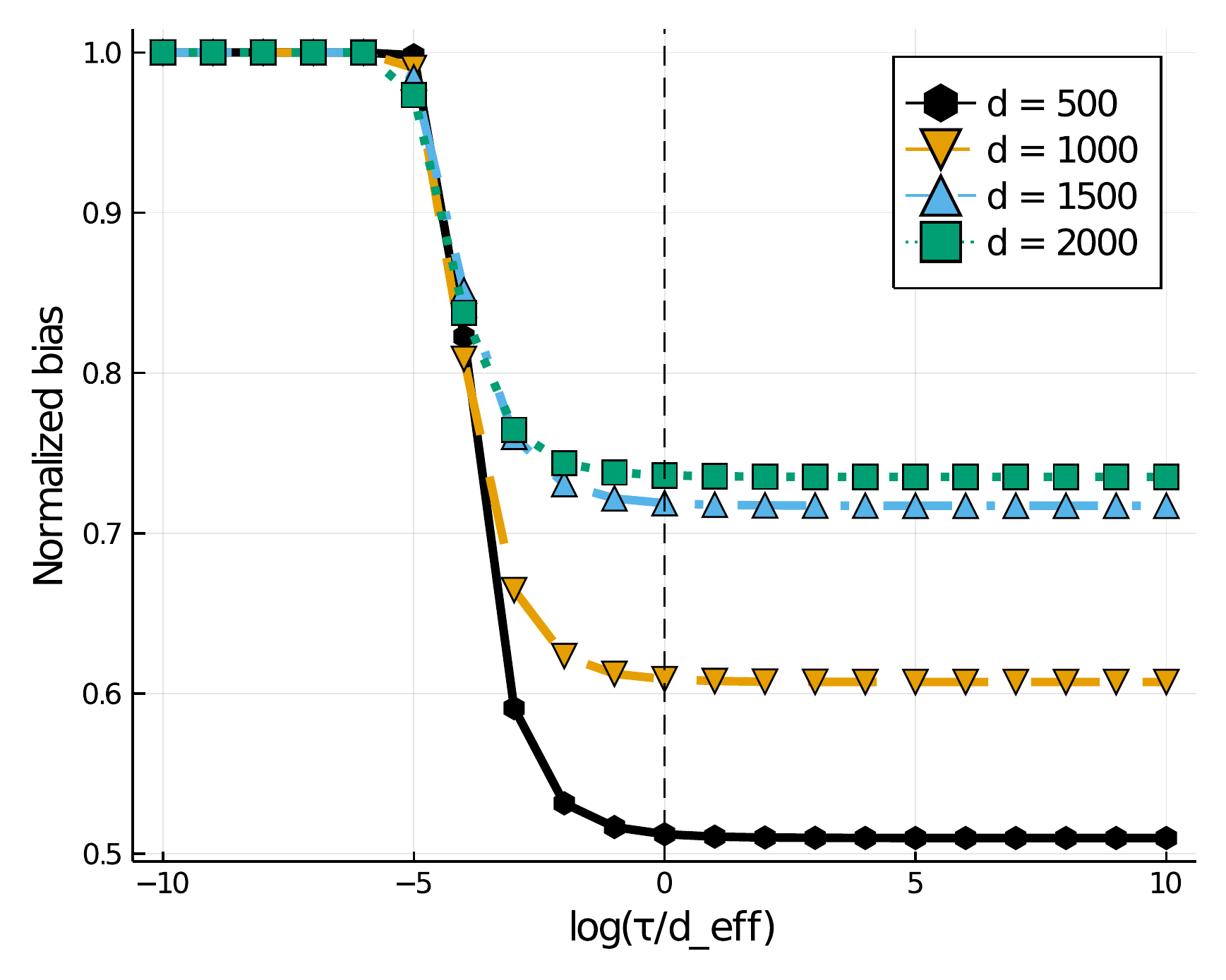}
    \caption{Exponential inner product kernel}
    \label{fig:tauexpo}
  \end{subfigure}  
     
\caption{\small{The bias of the minimum interpolant
    $\Bias(\fhatinterpol)$ normalized by $\Bias(0)$ as a function of the normalization constant
    $\bw$ for different choices of $d = \effdim$.
  The ground truth function is $f(x) = 2\xel{1}^3$ and $n= 2000$ noiseless
  observations are fit where the input vectors are sampled from an isotropic Gaussian with $d = \floor{n^{1/2}}$.}}
\label{fig:taudependence}
\end{figure*}


\begin{theorem}
  \label{thm:otherchoicesoftau}
   Let $\kernf$ be an RBF kernel with Fourier transform
 $\ftkernf$ such that for any $d$, ${\underset{\|\theta\| \to
   \infty}{\lim} \|\theta\|^{d+\alpha} \ftkernf(\theta) = \cphi > 0}$
 for some $\alpha \in(0,2)$.
Under the assumptions \Adone-\Adtwo~on the data distribution, the bias lower bound~\eqref{eq:biaslb} and polynomial approximation~\eqref{eq:polyapproxlp} hold for the flat limit interpolator $\underset{\bw \to \infty}{\lim}\fhatinterpol$ with the same $\beta$-dependence for $m$ as in Theorem~\ref{thm:main_thm}, given that $\ftrue$ is bounded on the support of $\prob_X$.
%

%
%
\end{theorem}

In particular, the assumptions hold for instance for the
$\alpha$-exponential kernels with $\alpha \in (0,2)$ (see
\cite{Blumenthal60}) and the popular Matern RBF kernels with $\nu <2$.
The proof  of the theorem can be found in~\suppmat~\ref{sec:flatlimit} and is based on the flat limit literature on RBFs
\cite{Lee2014,Driscoll02,Schaback05,Larsson05}.
Finally, we remark that Theorem \ref{thm:otherchoicesoftau} applies only for the
interpolating estimator $\fhatinterpol$. However, because the ridge penalty is well known to regulate the bias-variance tradeoff, we argue that the bias increases with the ridge penalty and hence attains
its minimum at $\lambda = 0$.

%% file: sections/new_section4.tex
\subsection{Discussion of theoretical results}
\label{sec:contributions}

As a consequence of our unifying treatment, we cover most of the
settings studied in the existing literature and show
that the polynomial approximation barrier~\eqref{eq:biaslb} is neither
restricted to a few specific distributions nor to a particular choice
of the scaling or the eigenvalue decay of the kernel function.
We summarize our setting
alongside those of previous works in Table~\ref{tab:settings}.
\paragraph{Distribution}
The Assumptions \Bdone-\Bdtwo~ allow very general distributions and include the ones in the current literature. In particular, we also cover the settings studied in the papers \cite{Liang20b,Liang20a,Liu20} and hence put their optimistic conclusions into perspective. 
Besides, our results hold true for \emph{arbitrary} covariance matrices and only depend on the growth rate of the effective dimension, but are independent of the explicit structure of the covariance matrix. This stands in contrast to linear regression
where consistency can (only) be guaranteed for spiky
covariance matrices \cite{Bartlett20,Muthukumar20}. 

\paragraph{Scaling} Our results do not only apply for the standard choice of the scaling ${\bw/\effdim \to c>0}$, but also apply to general RBF kernels in the flat limit scaling, i.e. where $\bw \to \infty$. This case is particularly important since this is where we empirically find the bias to attain its minimum in Figure \ref{fig:tauexpo}.
We therefore conjecture that the polynomial approximation barrier cannot be overcome with different choices of the scaling.


\begin{table*}[t]
  {\small
    \begin{center}
\begin{tabular}{c|c|c|c|c|c|c }
  Functions $f^\star$ & Kernels & Domain & Choice of $\bw$ &
  $\inputcov$ & Regime & Paper \\ \hline & & & & & &
  \\ $ \polyspace{m}$& IP, $\alpha$-exp, NTK &
  $\RR^d$, $\Sphere{d}(\sqrt{\effdim})$ & $\bw =
      \effdim $ & arbitrary  &
  $\effdim\asymp n^{\beta}$ & Ours \\
  $\polyspace{m}$& RBF &$\RR^d$,
  $\Sphere{d}(\sqrt{\effdim})$ & $\bw \to \infty$ & arbitrary
   & $\effdim \asymp n^{\beta}$ & Ours
  \\ 
  $\|\ftrue\|_{\Hilbertspace} = O(1)$& IP, NTK & $\RR^d$ & $\bw =
  \effdim=d$ & $\idmat_d$ & $d \asymp n^{\beta}$ &
  ~\cite{Liang20b} \\
  $\|\ftrue\|_{\Hilbertspace} = O(1)$\footnotemark &IP,
  RBF& $\RR^d$ & $\bw = \effdim=d$ & $\tr(\inputcov)/d \to c$ or $\to 0$
  & $d \asymp n$ &
      \cite{Liang20a},\cite{Liu20}
  \\
  $\polyspace{m'}$ & IP, NTK & $\Sphere{d}(\sqrt{\effdim})$ & $\bw = d $
  & $ \idmat_d$ & $d \asymp n^{\beta}$ &
  ~\cite{Ghorbani19}
  \\
  $\polyspace{m'}$ & IP, NTK &
  $\Sphere{d}(\sqrt{\effdim})$ & $\bw \approx
      \effdim $ & $\strut
      UU^\top + d^{-\kappa}\idmat $ & $\effdim
  \asymp n^{\beta}$ & ~\cite{Ghorbani20} \\
\end{tabular} 
  \end{center}}
 \caption{\small{
   Compilation of the different settings studied in the literature and our paper.
   The left-most column denotes the necessary conditions on the
   function space of the ground truth $\ftrue$ for the corresponding
   consistency results. 
   Here, $m = 2\floor{2/\beta}$ and $m' = \floor{1/\beta}$. 
 }}
 \label{tab:settings}
\end{table*}

\footnotetext{ \cite{Liu20} actually requires the weaker assumption that the source condition parameter $r>0$.}

\paragraph{Eigenvalue decay}
Furthermore, by explicitly showing that the polynomial barrier
persists for all $\alpha$-exponential kernels with $\alpha \in (0,2]$
  (that have vastly different eigenvalue decay rates), we provide a
  counterpoint to previous work that suggests consistency for $\alpha \to 0$.
  In particular, the paper \cite{Belkin19} proves minimax optimal rates for the
  Nadaraya-Watson estimators with singular kernels for fixed
  dimensions and empirical work by \cite{Belkin18c,Wyner17} suggests that
  spikier kernels have more favorable performances.  Our results
  suggest that in high dimensions, the effect of
  eigenvalue decay (and hence ``spikiness'') may be dominated by
  asymptotic effects of rotationally invariant kernels.
  We discuss possible follow-up questions in Section~\ref{sec:discussion}.\\\\
 As a result, we can conclude that the polynomial approximation
 barrier is a rather general phenomenon that occurs for commonly used
 rotationally invariant kernels in high dimensional regimes.  For
 ground truths that are inherently higher-degree polynomials that
 depend on all dimensions, our theory predicts that consistency of
 kernel learning with fully-connected NTK, standard RBF or inner
 product kernels is out of reach if the data is high-dimensional.  In
 practice however, it is possible that not all dimensions carry
 equally relevant information. In Section~\ref{sec:expreal} we show
 how feature selection can be used in such settings to circumvent the bias
 lower bound. 
 On the other hand, for image datasets like CIFAR-10
 where the ground truth is a complex function of all input dimensions,
 kernels that incorporate convolutional structures (such as CNTK
 \cite{Arora19,Novak19} or compositional kernels
 \cite{Daniely16,Shankar20,Mei21b}) and hence break the rotational symmetry
 can perform quite well.

\section{Experiments}
\label{sec:exp}

%
%

In this section we describe our synthetic and real-world experiments to further illustrate our theoretical results and underline the importance of feature selection in high dimensional
kernel learning.

\subsection{Hilbert space norm increases with dimension $d$}
\label{sec:expnorm}
In Figure \ref{fig:rkhsnorm}, we demonstrate how the Hilbert norm of
the simple sparse linear function $\fexp(\x) = \xel{1}$ grows with
dimension $d$.  We choose the scaling $\bw = d$ and consider the Hilbert space induced by the
scaled 
Gaussian $\kernf_{\bw}(\xk,\xkd) = \exp(-\frac{\| \xk-\xkd
  \|^2}{\bw})$, Laplace $\kernf_{\bw}(\xk,\xkd) = \exp(-\frac{ \| \xk-\xkd
  \|_2}{\sqrt{\bw}})$ and exponential inner product $\kernf_{\bw}(\xk, \xkd)
= \exp(-\frac{\xk^T\xkd}{\bw})$ kernels.
To estimate the norm, we draw $7500$ i.i.d.~random samples with noiseless observations from the uniform distribution on $\XX=[0,1]^d$.

\subsection{Illustration of the polynomial approximation barrier}

\label{sec:expsynthetic}
We now provide details for the numerical experiments in
Figure~\ref{fig:polymodels},\ref{fig:taudependence} that illustrate
the lower bounds on the bias in Theorem~\ref{thm:main_thm} and
Theorem~\ref{thm:otherchoicesoftau}.
For this purpose, we consider the
following three different distributions that satisfy the assumptions
of the theorems and are covered in previous works
\begin{itemize}
    \item \textbf{$\prob_1$, $\prob_2$:} $\standardrv_{(i)} \sim \Normal(0,1)$, $d = \floor{n^{\beta}}$ and $\inputcov = \textrm{I}_d$. $\prob_1$: $X = \inputcov^{1/2} \standardrv$, $\prob_2$: $X = \frac{\sqrt{d} \inputcov^{1/2} \standardrv}{ \|\inputcov^{1/2} \standardrv \|_2} $ .
    \item \textbf{$\prob_3$:} $X = \inputcov^{1/2} \standardrv$, $\standardrv_{(i)} \sim \Uniform([-u,u]^d)$, with $u$ s.t. unit variance, $d = n$ and $\inputcov$ diagonal matrix with entries $(1 - ((i-1)/d)^{\kappa})^{1/\kappa}$ and $\kappa \geq 0$ s.t. $\trace(\inputcov) = n^{\beta}$ .
\end{itemize}
We primarily use the Laplace kernel\footnote{We choose the Laplace kernel because of its numerical
  stability and good performance on the high dimensional datasets
  studied in \cite{Belkin18c,Geifman20}. Other kernels can be found in
  the \suppmat~\ref{sec:appexp}.} with $\bw = \tr(\inputcov)$
unless otherwise specified
and study two 
sparse monomials as ground truth functions, $\fexp_1(\x) = 2 \xel{1}^2$ and
$\fexp_2(\x) = 2 \xel{1}^3$. In order to estimate the bias~$\| \EEobs \fhatinterpol - \ftrue \|^2_{\Ell_2(\prob_X)}$ of the minimum norm interpolant we generate noiseless observations and approximate the
expected squared error using $10000$ i.i.d. test samples. \\

In Figure~\ref{fig:fig2a} and \ref{fig:fig2b}, we
plot the dependence of the bias on the parameter $\beta$ which
captures the degree of high-dimensionality, i.e. how large dimension
$d$ is compared to the number of samples $n$. We vary $\beta$ by
fixing $n$ and increasing $d$ (see also~\suppmat~\ for plots for fixed
$d$ and varying $n$). In Figure \ref{fig:fig2a} we demonstrate the
important consequence of our unifying framework that the polynomial
barrier only depends on the growth of the effective dimension
$\tr(\inputcov)$, parameterized by $\beta$ independent of the specific
choice of the distribution.\footnote{The differences in the absolute
  bias values between data models are due to the differences of
  $\|\fhatinterpol - \ftrue\|_{\Ell_2(\prob_X)}$ that explicitly
  depend on the underlying input model $\prob_X$.} The horizontal lines
that indicate the bias of the optimal linear fit $\flin$, show how for
large $\beta$, kernel learning with the Laplace kernel performs
just as well as a linear function.  Figure \ref{fig:fig2b} shows the
bias curve as a function of $\beta$ with inputs drawn from $\prob_1$ for different
choices of $n$. The bias curves are identical from which we conclude
that we already enter the asymptotic regime for $d,n$ as low as $d\sim
50$ and $n\sim 2000$.

While our theoretical results only show lower bounds on the bias, Figure \ref{fig:fig2a} and \ref{fig:fig2b} suggest that as  $\beta$  increases, the bias in fact stepwise aligns with the best polynomial of lower and lower order. This has also been shown in \cite{Ghorbani19,Ghorbani20} for the uniform distribution from the product of two spheres. Indeed, with decreasing $\beta$, for the cubic polynomial in Figure \ref{fig:fig2b} we first learn linear functions (first descent in the curve). Since the best degree 2 polynomial approximation of $f^*_2$ around $0$ is a linear function, the curve then enters a plateau before descending to zero indicating that we successfully learn the ground truth. 


Figure \ref{fig:taudependence} illustrates how the bias depends on the
scaling $\bw$ for different $\beta$ with $n=2000$. We generate
samples using $\prob_1$ and the ground truth $\fexp_2$ and plot the bias
of the minimum norm interpolator for the Laplace, exponential inner
product and Gaussian kernel.  For the latter, the minimum is obtained
around $\bw = d$.  For the Laplace and exponential inner product
kernel, the bias achieves its minimum at the flat limit $\bw
\to\infty$.  Given that our lower bounds hold for both $\bw = d$ and
$\bw \to\infty$ (Theorems~\ref{thm:main_thm} and
\ref{thm:otherchoicesoftau}), we hypothesize that there might not
exist an intermediate scaling regime for $\bw$ that can
break the polynomial approximation barrier.

\subsection{Feature selection for high-dimensional kernel learning}
\label{sec:expreal}

In this section, we demonstrate how the polynomial approximation
barrier limits the performance in real world datasets and how one may
overcome this issue using feature selection.  How our theory motivates
feature selection can be most cleanly illustrated for complex ground
truth functions that only depend on a number of covariates that is
much smaller than the input dimension (which we refer to as
\emph{sparse})\footnote{for more general
  functions that are not sparse we still expect a U-curve, albeit
  potentially less pronounced, whenever the function is not a
  low-degree polynomial of order $2/\beta$}:

Based on our theoretical results we expect that for
sparse ground truths, the bias follows a U-shape as dimension
increases: until all relevant features are included, the bias first
decreases before it then starts to decrease due to the polynomial
approximation barrier that holds for large $d$ when asymptotics start
to kick in.
Since recent work shows that the variance vanishes
in high-dimensional regimes (see e.g. \cite{Liang20b,
  Ghorbani20}), 
we expect the risk to follow a U-shaped curve as well.  Hence,
performing feature selection could effectively yield much better
generalization for sparse ground truth functions.  We would like to
emphasize that although the described curve may ring a familiar bell,
this behavior \emph{is not} due to the classical bias-variance
trade-off, since the U-shaped curve can be observed even in the
noiseless case where we have zero variance.  
We now present experiments that demonstrate the U-shape of the risk
curve for both synthetic experiments on sparse ground truths and
real-world data.  We vary the dimensionality $d$ by performing feature
selection\footnote{we expect other approaches to
  incorporate sparsity such as automatic relevance determination
  \cite{Neal96,MacKay96} to yield a similar effect} using the algorithm proposed in the paper \cite{Chen17}. In order to study
the impact of high-dimensionality on the variance, we add different
levels of noise to the observations.

For the real-world experiments we are not able to decompose the risk
to observe separate trends of the bias and the variance. However, we can
argue that the ridge estimator should perform significantly
better than the minimum norm interpolator whenever the variance dominates.
Vice versa, if their generalization errors are close, the effect of the
variance vanishes. 
This is due to the fact that the ridge penalty decreases the variance and
increases the bias. Hence, for real-world experiments, we use the comparison
between the risks of the ridge estimator and minimum norm interpolator
to deduce that the bias is in fact dominating the risk for high dimensions.



%

\begin{figure*}[t]
  \centering
          \begin{subfigure}{0.32\linewidth}
        \centering
  \includegraphics[width=\textwidth]{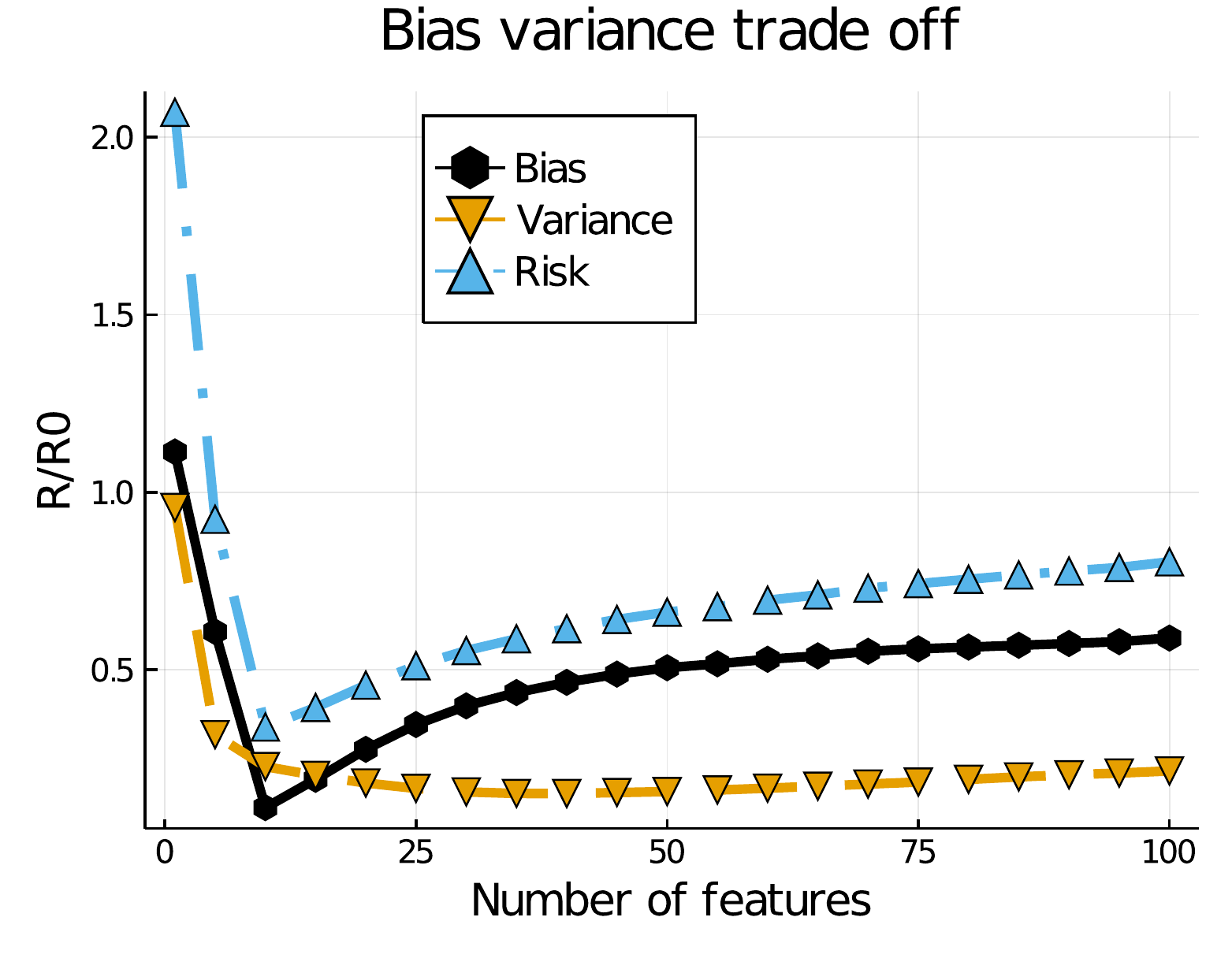}
    \caption{Bias variance trade-off}
    \label{fig:biasvar}
    \end{subfigure}  
    \hfill
  \begin{subfigure}{0.32\linewidth}
    \centering
  \includegraphics[width=\textwidth]{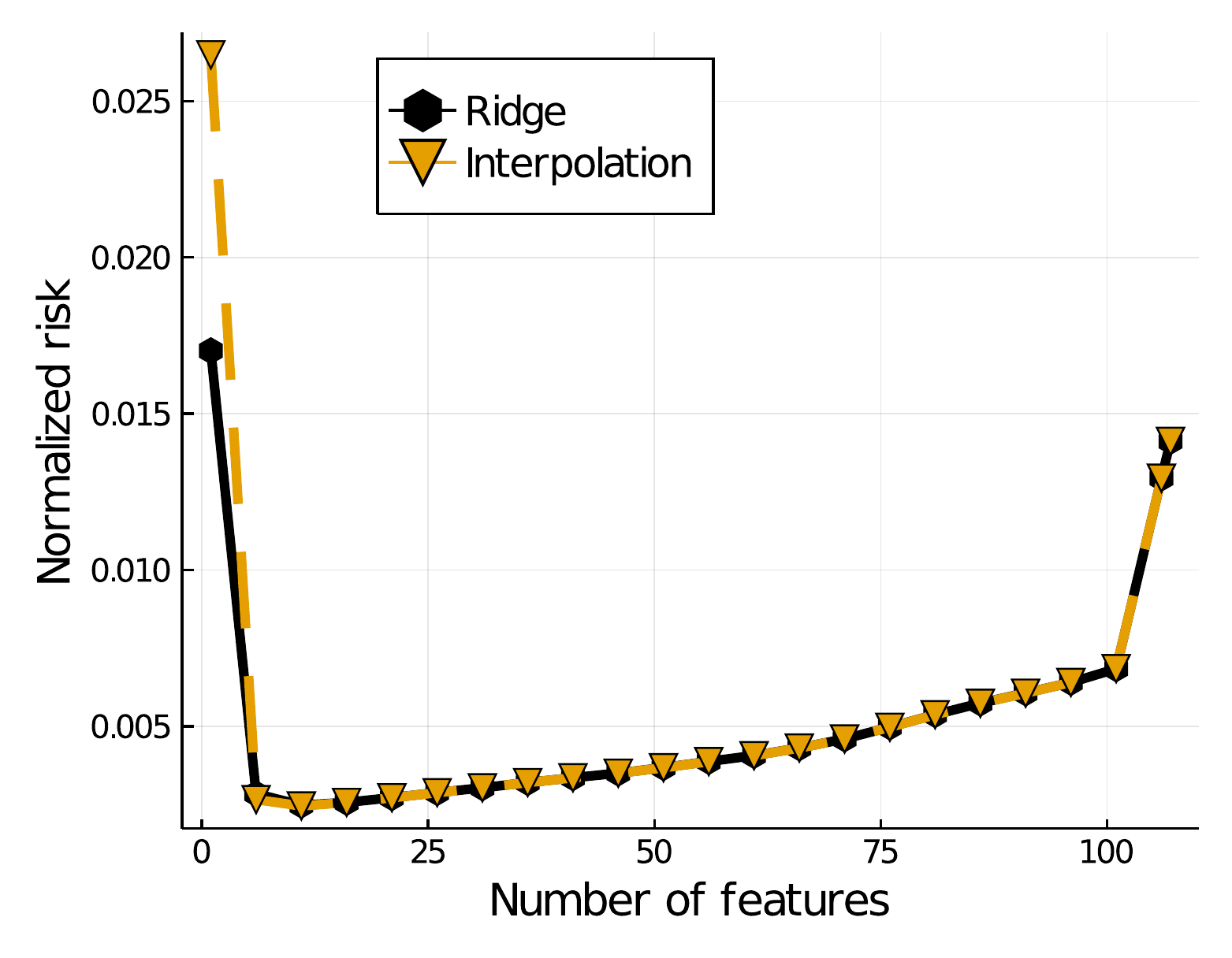}
    \caption{Residential housing - original}
    \label{fig:rh}
  \end{subfigure}
  \hfill
    \begin{subfigure}{0.32\linewidth}
    \centering
  \includegraphics[width=\textwidth]{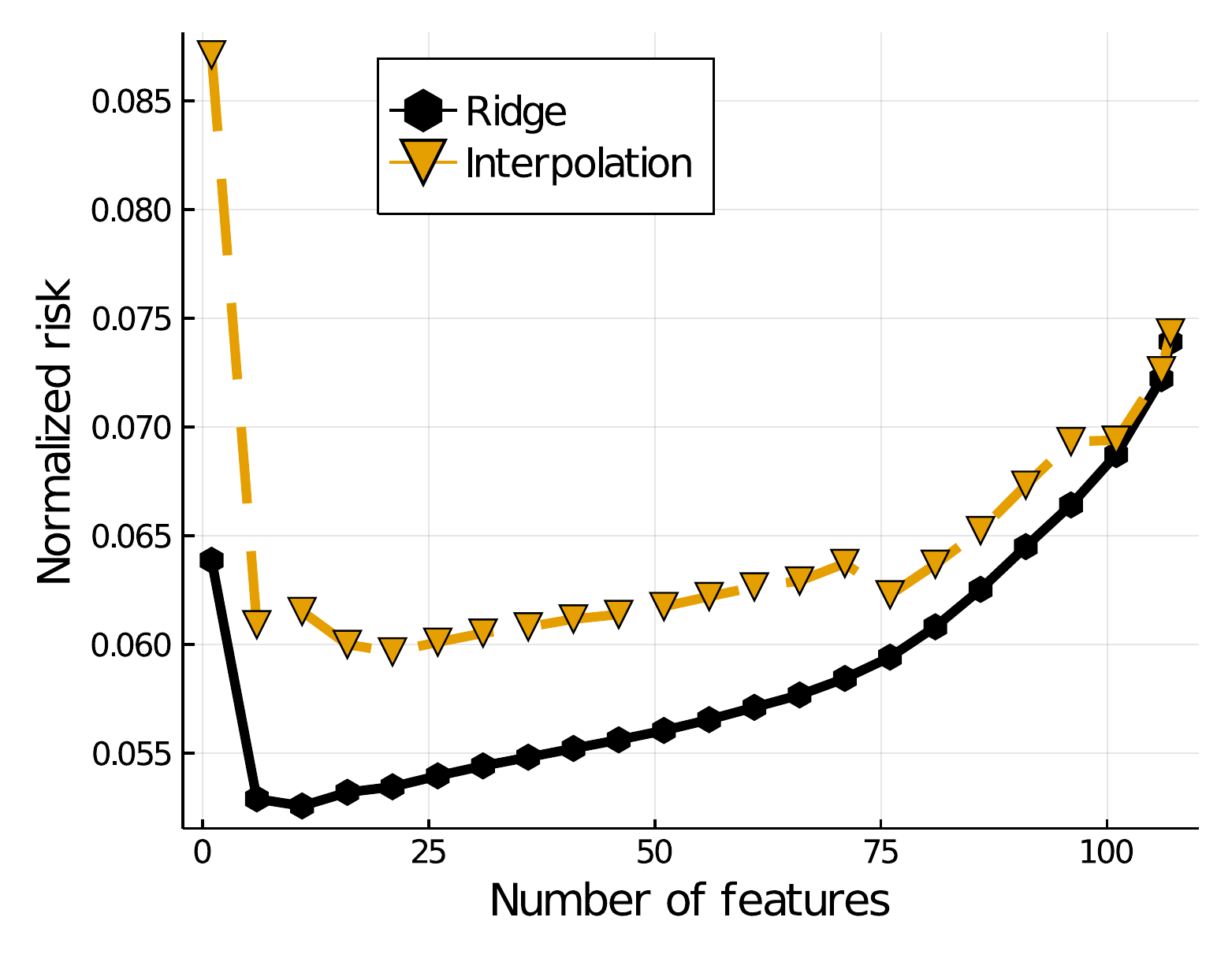}
    \caption{Residential housing - add. noise}
    \label{fig:rh2}
      \end{subfigure}
          \caption{\small{
             (a): The bias-variance trade-off of the minimum norm interpolant normalized by $\Bias(0)$ for 
             a synthetic experiment as a function of selected features (see details in Section \ref{sec:expreal}). The trends are reversed compared to the usual bias-variance curve as a function of model complexity and reflect our theoretical results: the bias term dominates the risk as dimension increases while the variance monotonically decreases with dimension. This behaviour can also be observed for the \textit{residential housing} dataset (b) without and (c) with additive synthetic noise where the risk of
            Ridge regression and interpolation follow similar trends that we hence attribute to the bias. }}
    \label{fig:realworld}
\end{figure*}
\paragraph{Sparse functions}
For our synthetic experiments, we draw $500$ input samples $x_1,\dots, x_n$ 
from $\prob_1$ and compute the minimum norm interpolator for $n=100$
different draws over noisy observations from a sparsely parameterized
function $y= 0.5\sum_{i=1}^4 \xel{2i+1}^2 - \sum_{i=1}^4 \xel{2i} +
\epsilon$ with uniform noise $\epsilon \sim \UU([-10,10])$.  We
increase $d$ by adding more dimensions (that are irrelevant for the
true function value) and compute the bias and variance of the minimum
norm interpolator. We approximate the $\Ell_2(\prob_x)$ norm using
Monte Carlo sampling from $\prob_1$. In Figure~\ref{fig:biasvar} we observe that as
we increase the number of selected features, the bias first
decreases until all relevant information is included and then
increases as irrelevant dimensions are added. This is in line with
our asymptotic theory that predicts an increasing bias due to the
progressively more restrictive polynomial approximation barrier.
Furthermore, the variance monotonically
decreases as expected from the literature and hence the risk in
follows the U-shaped curve
described above.


\paragraph{Real-world data}
We now explore the applicability of our results on real-world data
where the assumptions of the theorems are not necessarily satisfied.
For this purpose we select datasets where the number of features
is large compared to the number of samples. 
In this section we show results on the regression dataset
\textit{residential housing} (RH) with $n=372$ and $d= 107$ to predict
sales prices from the UCI website \cite{Dua17}. Further datasets
can be found in \suppmat~\ref{sec:app_expreal}.  In order to
study the effect of noise, we generate an additional dataset (RH-2)
where we add synthetic i.i.d.  noise drawn from the uniform
distribution on $[-1/2,1/2]$ to the observations. The plots in
Figure \ref{fig:realworld} are then generated as follows:  we increase
the number of features using a greedy forward selection procedure (see~\suppmat~\ref{sec:app_expreal} for further details ). We
then plot the risk achieved by the kernel ridge and ridgeless
estimate using the Laplace kernel on the new subset of features.

Figure~\ref{fig:rh} shows that the risks of the minimum norm interpolator
and the ridge estimator are identical, indicating that the risk is
essentially equivalent to the bias. Hence our first conclusion is
that, similar to the synthetic experiment, the bias follows a U-curve. For
the dataset RB-2 in Figure~\ref{fig:rh2}, we further observe that even
with additional observational noise, the ridge and ridgeless estimator
converge, i.e. the bias dominates for large $d$. 
We observe both trends in other high-dimensional datasets
discussed in~\suppmat~\ref{sec:app_expreal} as well. As a consequence, we can
conclude that even for some real-world datasets that do not necessarily
satisfy the conditions of our bias lower bound, feature selection is crucial
for kernel learning for noisy and noiseless observations alike.
We would like to note that this conclusion
does not necessarily contradict empirical work that
demonstrates good test performance of RBFs on other high-dimensional
data such as MNIST. In fact, the latter only suggests that linear or
polynomial fitting would do just as well for these datasets which has indeed been
suggested in \cite{Ghorbani20}.

%% file: sections/new_section5.tex
\section{Conclusion and future work}
\label{sec:discussion}

Kernel regression encourages estimators to have a certain structure by means
of the RKHS norm induced by the kernel. 
For example, the eigenvalue decay of $\alpha$-exponential kernels
results in estimators that tend to be smooth (i.e. Gaussian kernel)
or more spiky (i.e. small $\alpha<1$).
A far less discussed fact is that many kernels implicitly incorporate
additional structural assumptions. For example, rotational invariant
kernels are invariant under permutations and hence treat all
dimensions equally
Even though rotational invariance is a natural choice when no prior information on
the structure of the ground truth is available, this paper shows that the
corresponding inductive bias in high dimensions is in fact restricting
the average estimator to a polynomial.
In particular, we show in Theorems \ref{thm:main_thm} and
\ref{thm:otherchoicesoftau} that the lower bound on the bias is simply
the projection error of the ground truth function onto the space of
polynomials of degree at most $2\floor{2/\beta}$ respectively
$\floor{2/\beta}$. Apart from novel technical insights that result
from our unified analysis (discussed in Sec.~\ref{sec:contributions}),
our result also opens up new avenues for future research.

\paragraph{Future work} 
%


%
%
%
 Modern datasets which require sophisticated methods like deep neural
 networks to obtain good predictions are usually inherently
 non-polynomial and high-dimensional. Hence, our theory predicts that
 commonly used rotationally invariant kernels cannot perform well for
 these problems due to a high bias. In particular, our bounds are
 independent of properties like the smoothness of the kernel function
 and cannot be overcome by carefully choosing the eigenvalue decay.
 Therefore, in order to understand why certain highly overparameterized
 methods generalize well in these settings, 
 our results suggest that it is at least
 equally important to understand how prior information can be
 incorporated to break the rotational symmetry of the kernel
 function. Examples for recent contributions in this direction are
 kernels relying on convolution structures (such as CNTK
 \cite{Arora19,Novak19} or compositional kernels
 \cite{Daniely16,Shankar20,Mei21b}) for image datasets.

 Another relevant future research direction is to present a tighter
 non-asymptotic analysis that allows a more accurate characterization of
 the estimator in practice. 
 The presented results in this paper are
 asymptotic statements, meaning that they do not provide explicit
 bounds for fixed $n,d$.
 Therefore, for  given finite $n,d$ it is unclear which high-dimensional regime
 provides the most accurate characterization of the estimator's
 statistical properties. For instance, our current results do not
 provide any evidence whether the estimator follows the bias lower bounds for $n = d^{\beta}$ with
 $\beta = \log(n)/\log(d)$ or  $n = \gamma d$.
 We remark that the methodology used to prove
 the statements in this paper could also be used to derive
 non-asymptotic bounds, allowing us to further investigate this
 problem. However, we omitted such results in this paper for the sake
 of clarity and compactness of our theorem statements and proofs and
 leave this for future work.

\section{Acknowledgments}
We would like to thank the anonymous reviewers for their helpful
feedback, Xiao Yang for insightful discussions on the flat limit kernel and
Armeen Taeb for his comments on the manuscript.

%% file: sections/Appendix/Appendix_Experiments.tex
\section{Experiments}
\label{sec:appexp}
This section contains additional experiments not shown in the main text. \footnote{Our code is publicly available at \href{https://www.github.com/DonhauserK/High-dim-kernel-paper/}{https://www.github.com/DonhauserK/High-dim-kernel-paper/}}

\subsection{Polynomial Barrier}

In this section, we provide
additional experiments that discuss Theorem \ref{thm:main_thm}. In particular, we
investigate kernels beyond the Laplace kernel and study the behaviour
of the bias with respect to $\beta$ when $d$ is fixed and $n$
varies. The experimental setting is the same as the one in Section
\ref{sec:expsynthetic}.
\begin{figure*}[htbp]
    \centering
  \begin{subfigure}{0.3\linewidth}
    \centering
  \includegraphics[width=\textwidth]{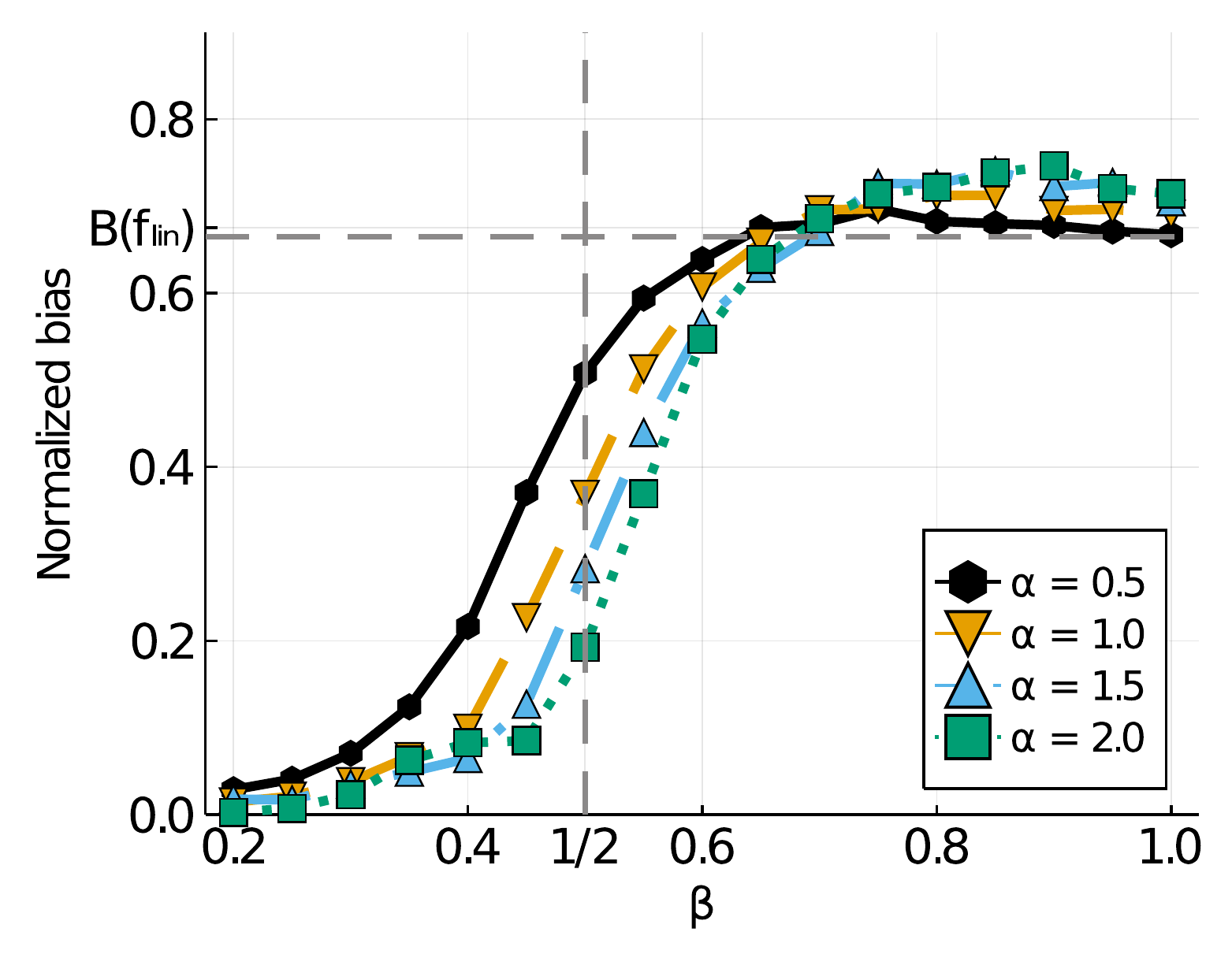}
\caption{$\prob_1$}
  \end{subfigure}
    \begin{subfigure}{0.3\linewidth}
    \centering
  \includegraphics[width=\textwidth]{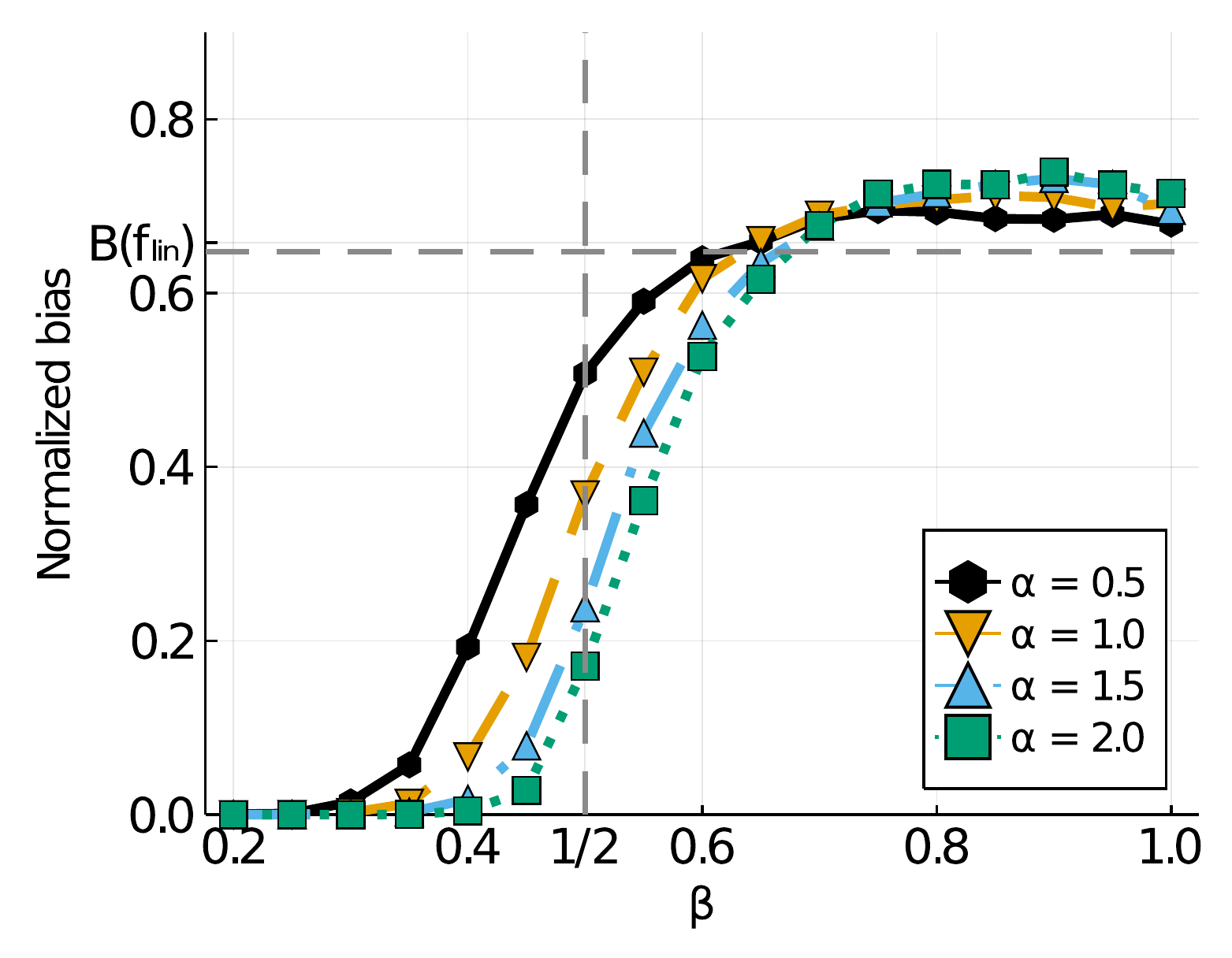}
\caption{$\prob_2$}
  \end{subfigure}
    \begin{subfigure}{0.3\linewidth}
    \centering
  \includegraphics[width=\textwidth]{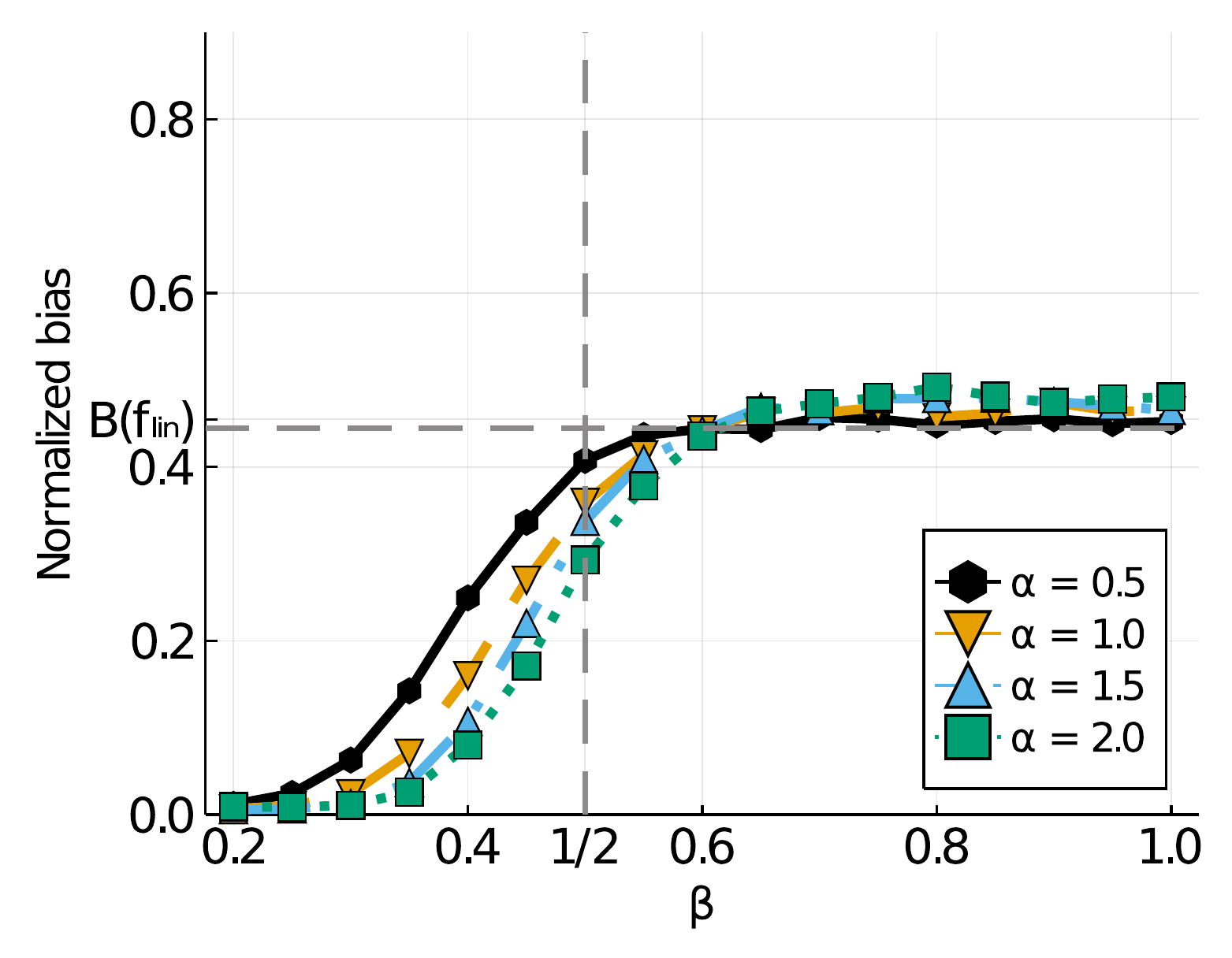}
\caption{$\prob_3$}
  \end{subfigure}
      \caption{\small{The bias of the minimum norm interpolant
    $\Bias(\fhatinterpol)$ normalized by $\Bias(0)$ as a function of
    $\beta$ for the $\alpha$-exponential kernel with different choices of $\alpha$ and with $n=4000$ i.i.d. samples drawn from (a) $\prob_1$, (b) $\prob_2$ and (c) $\prob_3$.}}
  \hfill
  \label{fig:app2}
  \vspace{-0.15in}
\end{figure*}

Instead of comparing the bias curves for different input distribution
as in Figure \ref{fig:fig2a}, Figure \ref{fig:app2} shows the bias
with respect to $\beta$ for the $\alpha$-exponential kernel,
i.e. $k(\xk,\xkd) = \exp(-\|\xk- \xkd\|_2^{\alpha})$, for different
choices of $\alpha$ and hence for kernels with distinct eigenvalue
decays ($\alpha = 2$ results in an exponential eigenvalue decay while
$\alpha <2$ in a polynomial eigenvalue decay).  Clearly, we can see
that the curves transition at a similar value for $\beta$ which confirms the
the discussion of Theorem~\ref{thm:main_thm} in Section \ref{sec:contributions}
where we argue that the polynomial approximation barrier occurs independently of the
eigenvalue decay.

Figure \ref{fig:app1} shows the bias of the minimum norm interpolant
$\Bias(\fhatinterpol)$ normalized by $\Bias(0)$ for the ground truth
function $\ftrue(\x) = 2\xel{1}^3$ and the Laplace kernel as in
Section \ref{sec:expsynthetic} with $\bw = \effdim$. We observe that the asymptotics
already kick in for $d \approx 40$ since all curves for $d \geq 40$ resemble each other.
This confirms the the trend in Figure \ref{fig:fig2b}.

\begin{figure*}[htbp]
    \centering
  \begin{subfigure}{0.3\linewidth}
    \centering
  \includegraphics[width=\textwidth]{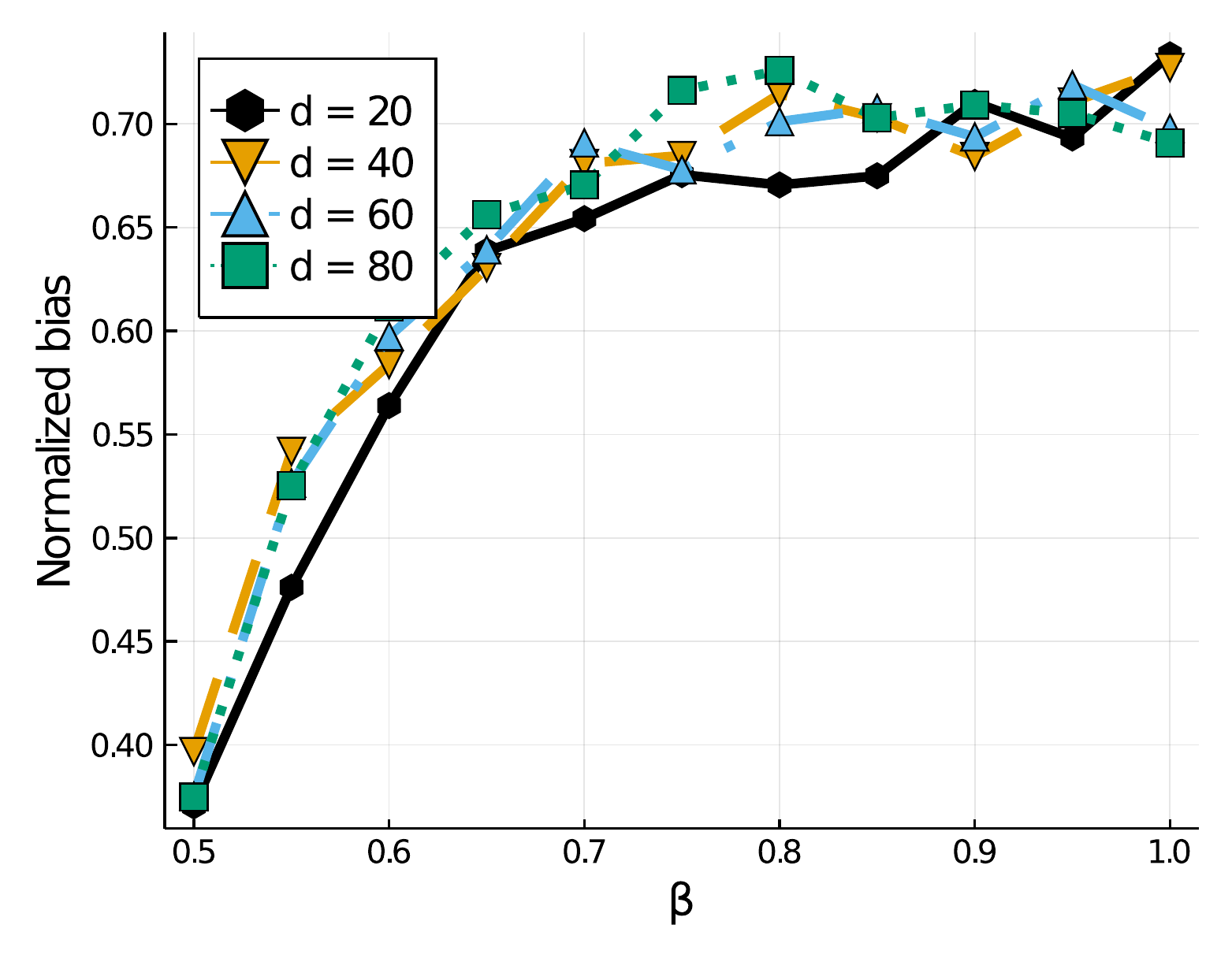}

  \end{subfigure}
      \caption{\small{The bias of the minimum norm interpolant
    $\Bias(\fhatinterpol)$ normalized by $\Risk(0)$ as a function of
    $\beta$ for different choices of  $d$ and samples generated from an isotropic Gaussian (as in Model $\prob_1$) with $n = \floor{d^{1/\beta}}$.   }}
  \hfill
  \label{fig:app1}
    \vspace{-0.15in}
\end{figure*}

\subsection{Feature selection - Synthetic}
\label{sec:app_expsynth}
The goal of this experiment is to compare the bias variance trade-off
of ridge regression and minimum norm interpolation. We use the same
experimental setting as the ones used for Figure \ref{fig:biasvar}
(see Section \ref{sec:expreal}). We set the bandwidth to $\bw
=\effdim$ and choose the ridge parameter $\lambda$ using $5$-fold
cross validation. While for small dimensions $d$, ridge regularization
is crucial to achieve good performance, the bias becomes dominant as
the dimension grows and the difference of the risks of both methods 
shrinks. This aligns well with Theorem
\ref{thm:main_thm} which predicts that the bias starts to increase
with $d$ for fixed $n$ once we enter the asymptotic regime.

%

\begin{figure}[H]
    \centering
  \begin{subfigure}{0.45\linewidth}
    \centering
  \includegraphics[width=0.8\textwidth]{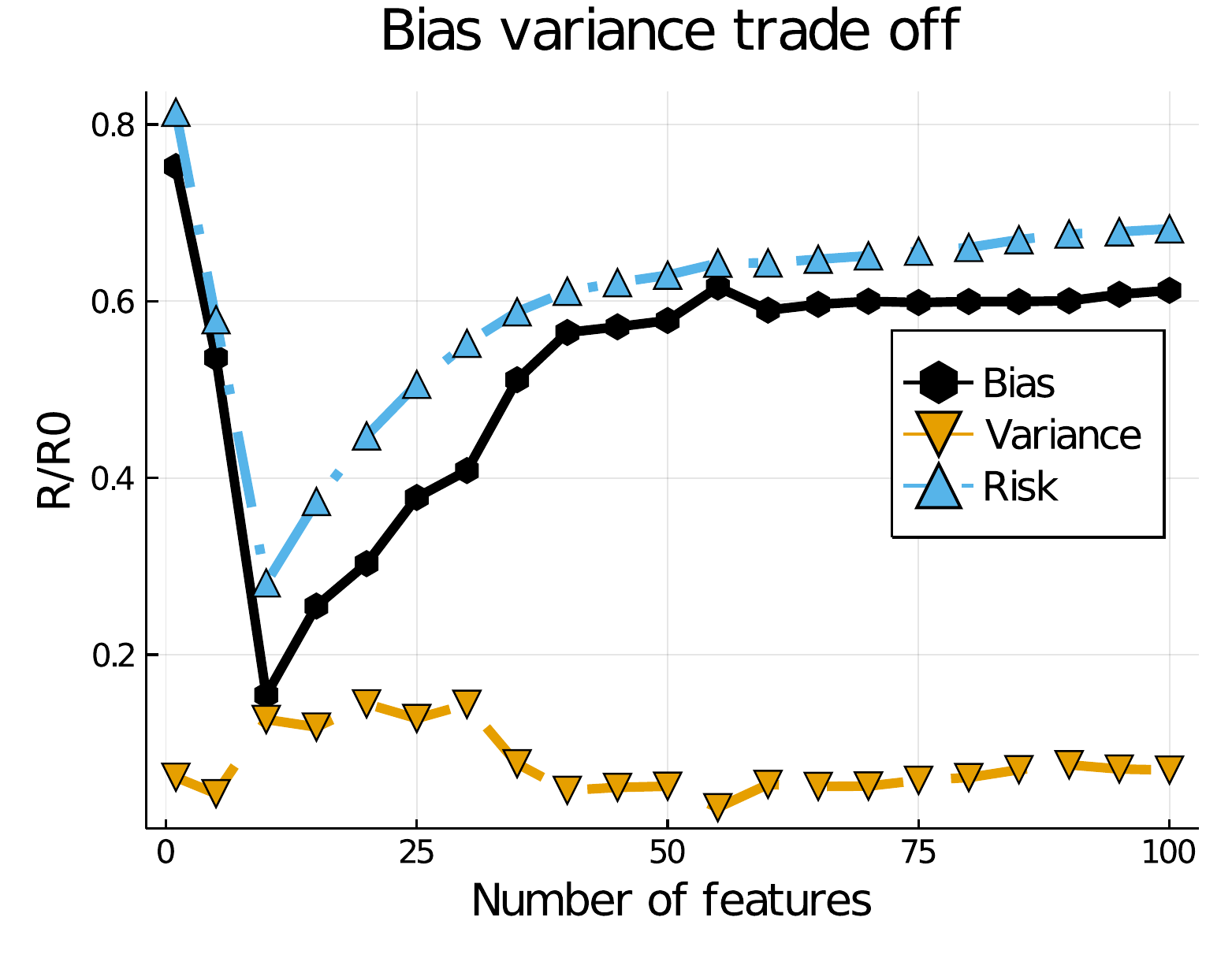}
    \caption{Ridge regression}
  \end{subfigure}
    \begin{subfigure}{0.45\linewidth}
    \centering
  \includegraphics[width=0.8\textwidth]{Plots_v2/fig7int.pdf}
    \caption{Minimum norm interpolation}
      \end{subfigure}
  \hfill
        \caption{\small{
          The bias-variance trade-off of the (a)
          ridge estimate and (b) minimum norm interpolant normalized
          by $\Bias(0)$ as a function of selected features for the
          synthetic experiment described in Section
          \ref{sec:expreal}. Figure (b) is exactly the same as Figure
          \ref{fig:biasvar}. }}
    \vspace{-0.15in}
\end{figure}

\subsection{Feature selection - Real world}
\label{sec:app_expreal}
We now present details for our real world experiments to emphasize the
relevance of feature selection when using kernel regression for 
practical applications, as discussed in Section \ref{sec:expreal}.

We consider the following data sets: 
\begin{enumerate}
 \item The \textit{residential housing} regression data set from the
   UCI website \cite{Dua17} where we predict the sales prices and
   construction costs given a variety of features including the floor
   area of the building or the duration of construction.
 \item The \textit{ALLAML} classification data set from the ASU
   feature selection website \cite{Li18} where we classify patients
   with acute myeloid leukemia (AML) and acute lymphoblastic leukemia
   (ALL) based on features gained from gene expression monitoring via
   DNA microarrays.
 \item The \textit{CLL\_SUB\_111} classification dataset from the ASU
   feature selection web-page \cite{Li18} where we classify
   genetically and clinically distinct subgroups of B-cell chronic
   lymphocytic leukemia (B-CLL) based on features consisting of gene
   expressions from high density oligonucleotide arrays. While the
   original dataset contains three different classes, we only use the
   classes 2 and 3 for our experiment to obtain a binary
   classification problem.
\end{enumerate}

Because the number of features in the \textit{ALLAML} and
\textit{CLL\_SUB\_111} datasets massively exceed the number of
samples, we run the feature selection algorithm in \cite{Chen17} and
pre-select the best $100$ features chosen by the algorithm. In order to reduce the computational expenses, we run the algorithm in batches of $2000$ features and iteratively remove all features except for the best $200$ features chosen by the algorithm. We do this until we reduce the total number of features to $2000$ and then select in a last round the final $100$ features used for the further procedure. Reducing the amount of features to $100$ is
important for the computational feasibility of greedy forward
features selection in our experiments.
The properties of the datasets are summarized in Table
\ref{tab:datasets}.

%

\begin{table}[H]
    \centering
    \begin{tabular}{c|ccc}
        Data set & \textit{Binary CLL\_SUB\_111} &\textit{ALLAML} & Residential Building Data Set
        \\
        \hline
         Features &  11,340 (100) & 7129 (100) & 107 \\
         Samples & 100 & 72 & 372\\
         Type & Binary classification & Binary classification & Regression\\
    \end{tabular}
    \caption{Real world datasets used for the experiments. The value
      in the brackets shows the number of features after a pre-selection
      using the algorithm presented in \cite{Chen17}. }
    \label{tab:datasets}
    \vspace{-0.1in}
\end{table}

\textit{Experimental setting:} As a first step, we normalise both the
vectors containing the single input features and the observations
separately using $\ell_1$ normalization. We use the Laplace kernel for
computing the ridge and ridgeless estimate in all experiments.
For each setting, we pick the bandwidth $\bw$ and
the penalty coefficient $\lambda$ (for the ridge estimator)
using cross validation.
 We increase the number of features by greedily adding the feature
that results in the lowest $5$-fold cross validation risk.
In addition, in order to study the effect of noise, we generate additional data sets
where we add synthetic i.i.d.  noise drawn from the uniform
distribution on $[-1/2,1/2]$ to the observations for the regression
tasks and flip 20\% of the label for the classification tasks.\\

\textit{Results of the experiments:} The following figures present the
results of our experiments on all datasets except for the ones predicting the sales
prices in the \textit{residential housing} dataset, which we presented
in Figure \ref{fig:rh},\ref{fig:rh2} in the main text. Similar to the
observations made in Section \ref{sec:expreal}, Figure
\ref{fig:figalm},\ref{fig:cll},\ref{fig:res2} show that the risk
reaches its minimum around $d \approx 25$, with significant
differences to the right at $d \approx 100$. In particular, this holds
for both ridge regression and interpolation, which again shows that
the bias becomes dominant as the dimension increases. Surprisingly, we
also note that the relevance of ridge regularization seems to be much
smaller for classification tasks than regression tasks.

%

\begin{figure}[hbt!]
    \centering
  \begin{subfigure}{0.45\linewidth}
    \centering
  \includegraphics[width=0.8\textwidth]{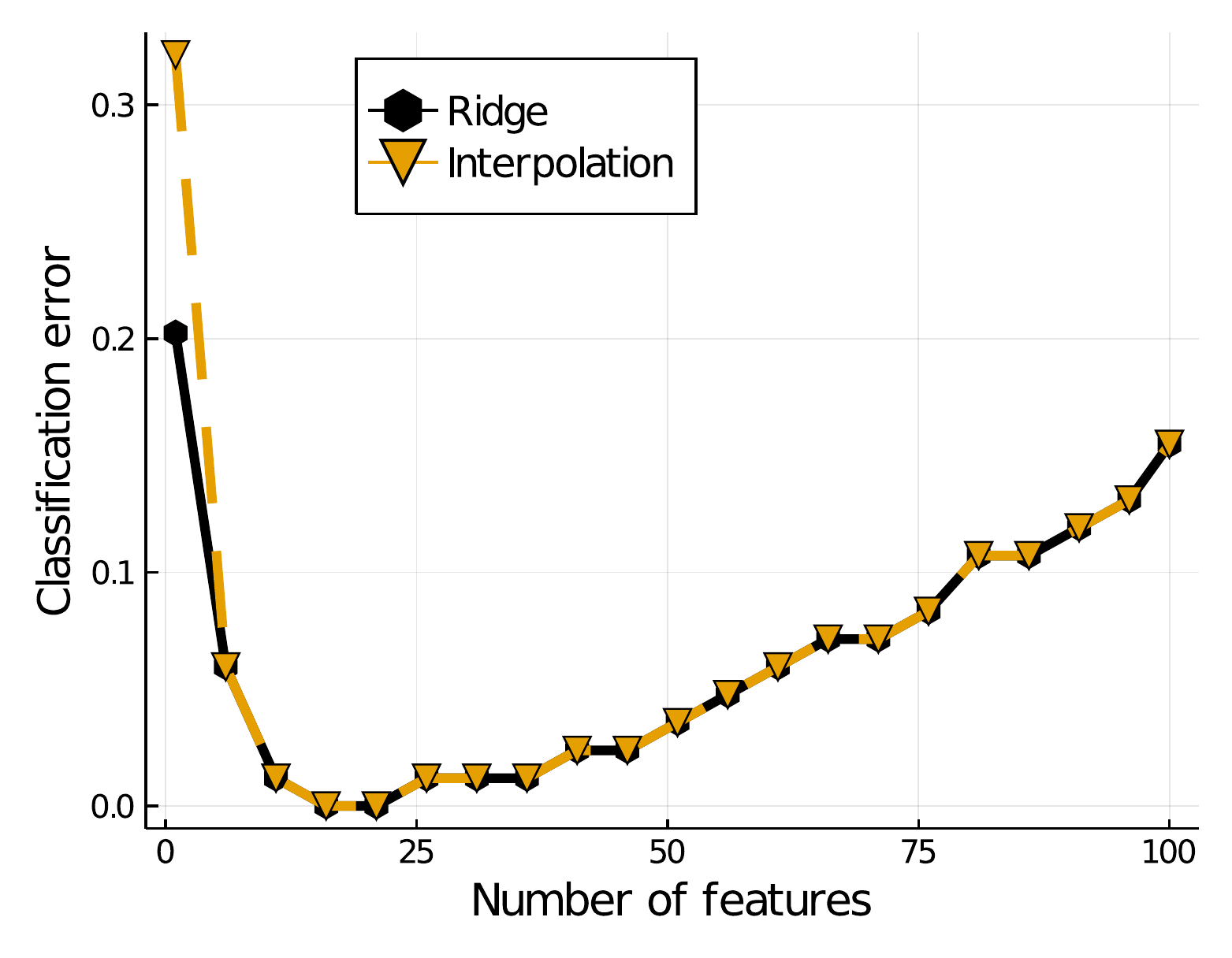}
    \caption{\textit{ALLAML} - original}
  \end{subfigure}
    \begin{subfigure}{0.45\linewidth}
    \centering
  \includegraphics[width=0.8\textwidth]{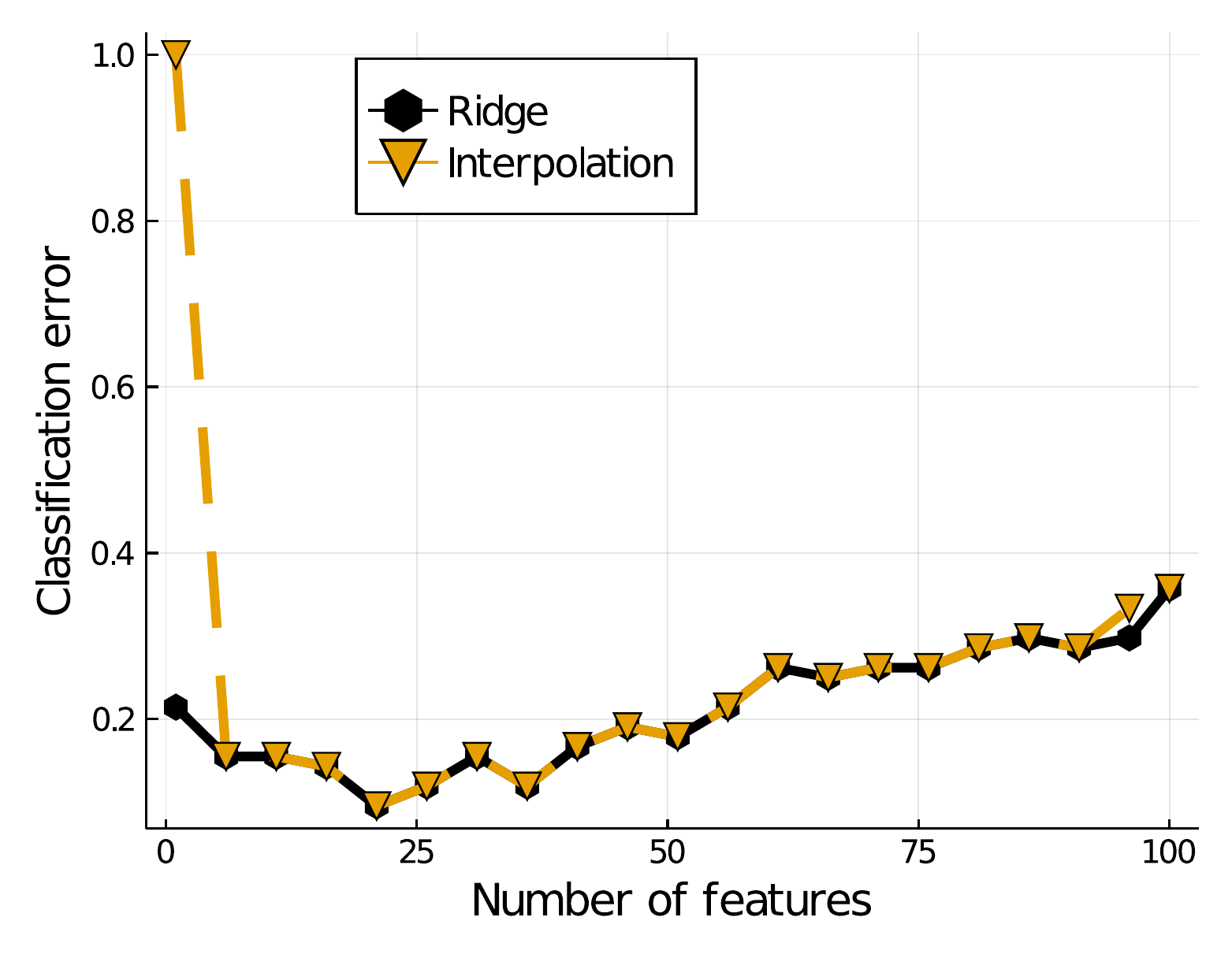}
    \caption{\textit{ALLAML} - 20\% label noise}
      \end{subfigure}
  \hfill
        \caption{ \small{The classification error of the minimum norm interpolator and ridge estimator for the \textit{ALLAML} dataset. }}
    \vspace{-0.15in}
      \label{fig:figalm}

\end{figure}

\begin{figure}[H]
    \centering
  \begin{subfigure}{0.45\linewidth}
    \centering
  \includegraphics[width=0.8\textwidth]{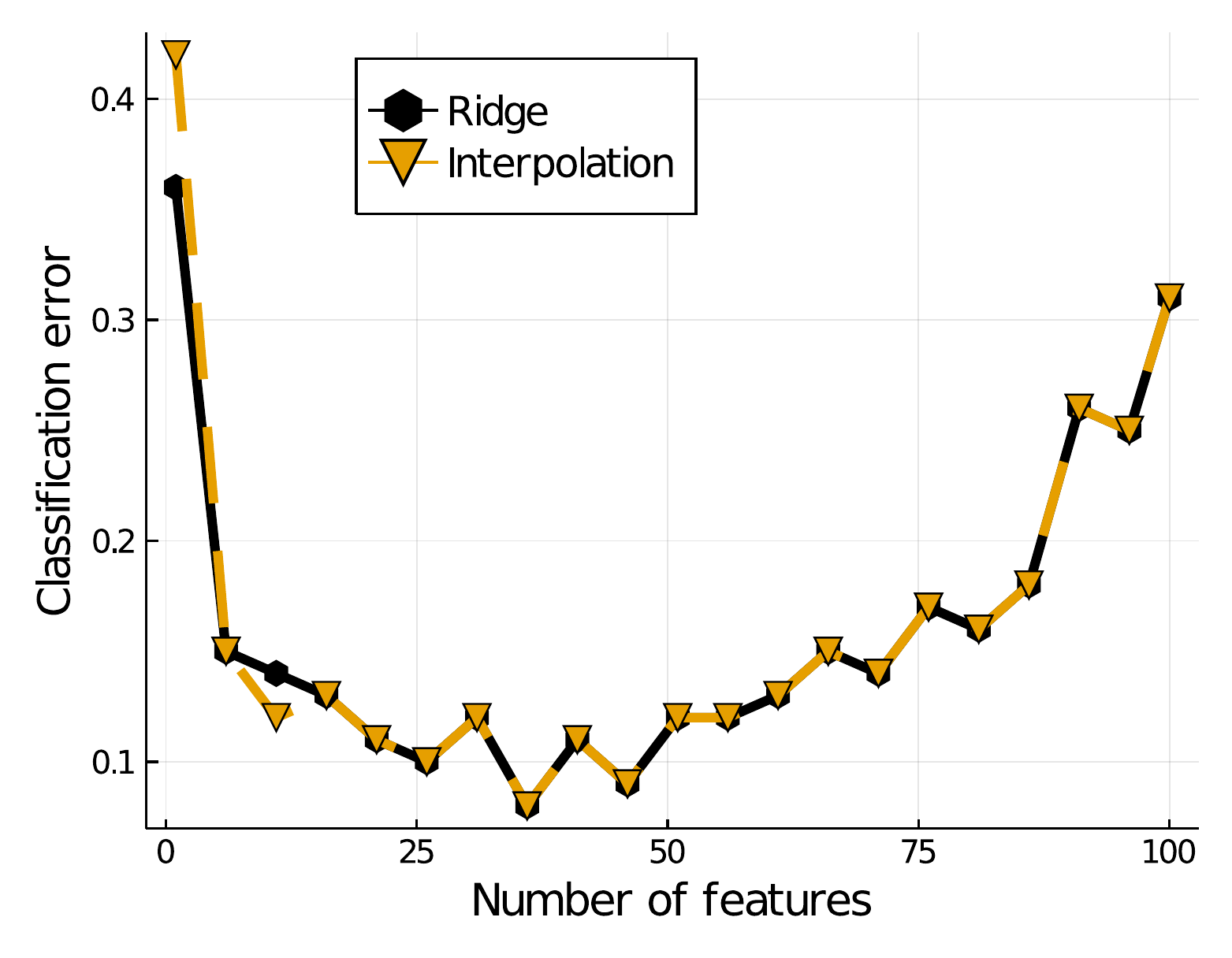}
    \caption{\textit{CLL\_SUB\_111} - original}
  \end{subfigure}
    \begin{subfigure}{0.45\linewidth}
    \centering
  \includegraphics[width=0.8\textwidth]{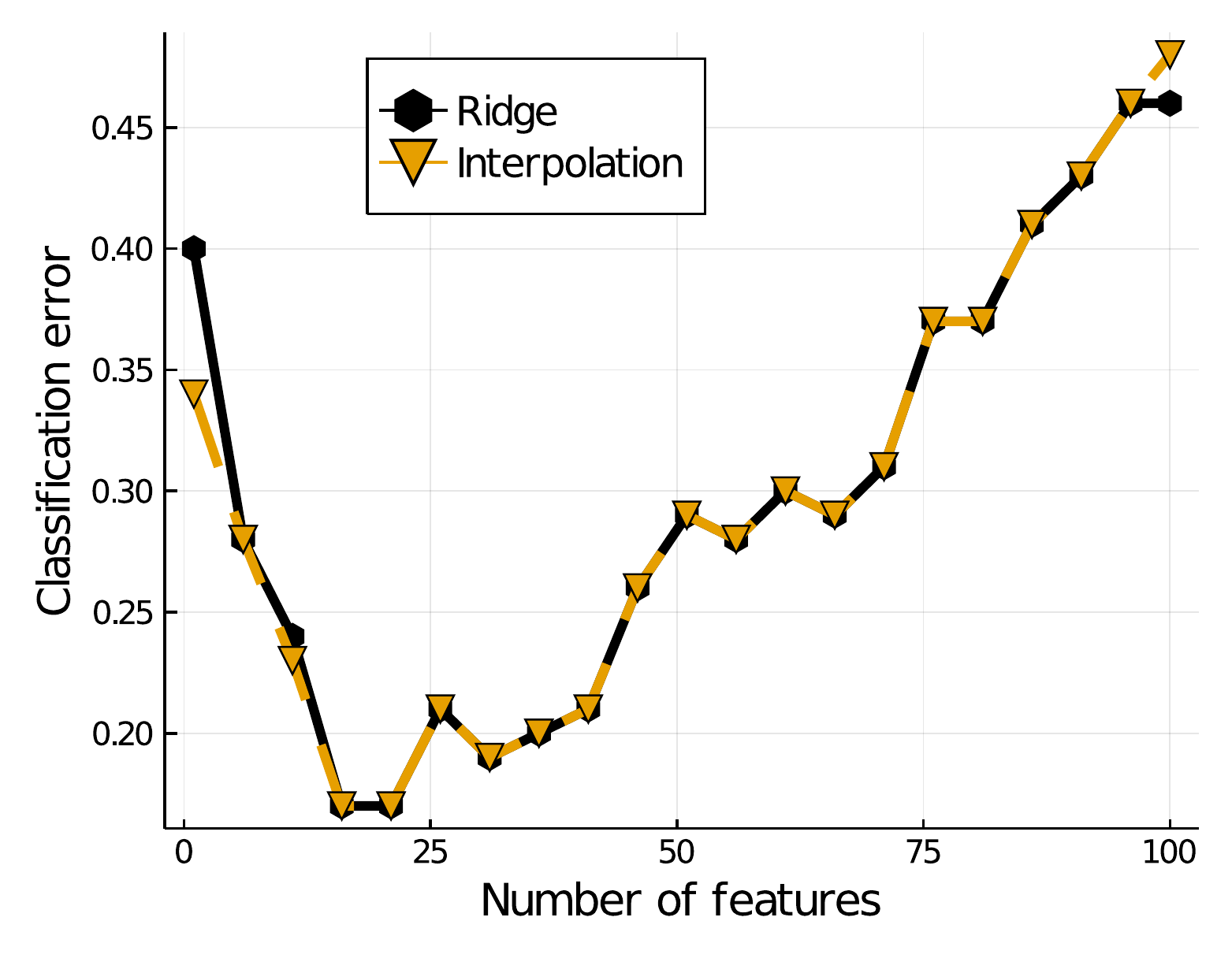}
    \caption{\textit{CLL\_SUB\_111} - 20\% label noise}
      \end{subfigure}
  \hfill
        \caption{\small{The classification error of the minimum norm interpolator and ridge estimator  for the \textit{CLL\_SUB\_111} dataset. }}
    \label{fig:cll}
\end{figure}

\begin{figure}[H]
    \centering
  \begin{subfigure}{0.45\linewidth}
    \centering
  \includegraphics[width=0.8\textwidth]{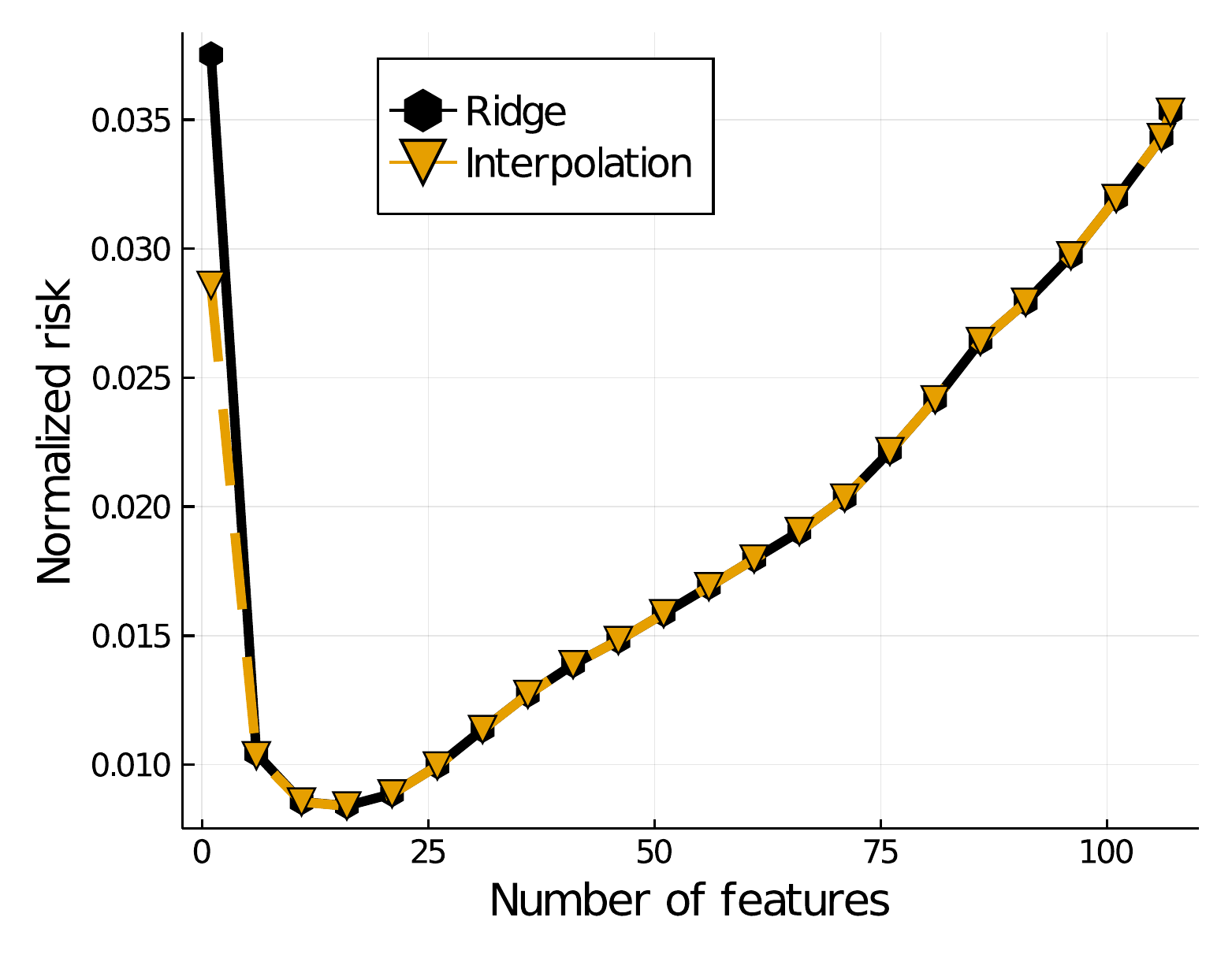}
    \caption{Residential housing - original}
  \end{subfigure}
    \begin{subfigure}{0.45\linewidth}
    \centering
  \includegraphics[width=0.8\textwidth]{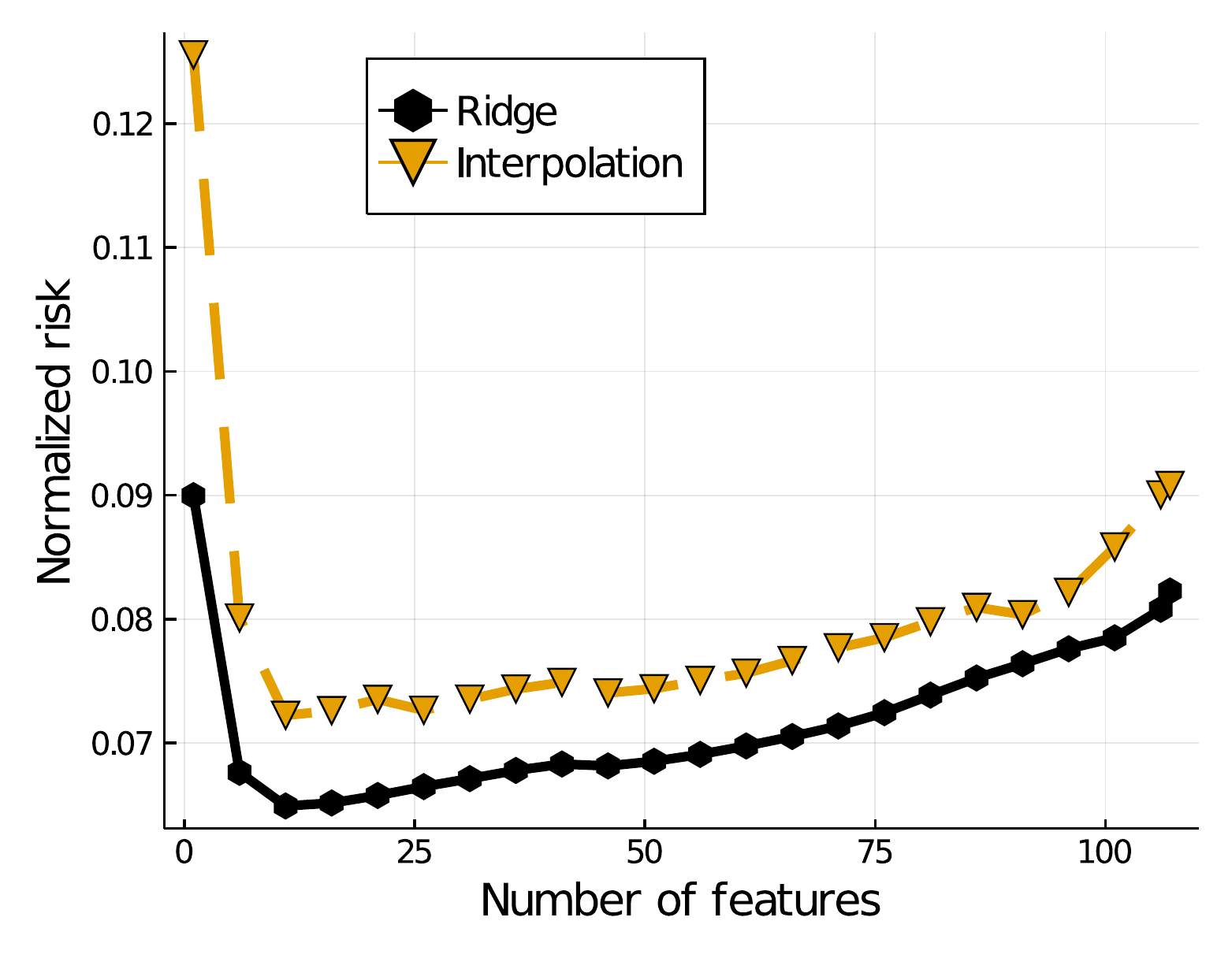}
    \caption{Residential housing - add. noise}
      \end{subfigure}
  \hfill
        \caption{\small{The risk $\Risk(\fhatinterpol)$ of the minimum norm interpolant respectively $\Risk(\fhatridge)$ of the ridge estimate normalized by $\Risk(0)$ for the \textit{residential housing} dataset with target construction costs.  }}
            \label{fig:res2}

    \vspace{-0.15in}
\end{figure}
\clearpage


%% file: sections/Appendix/Appendix_RKHS_Norm.tex
\section{Bounded Hilbert norm assumption}

\label{sec:norm}

This section gives a formal statement of Lemma \ref{lm:rkhs}. 
We begin with the conditions under which the lemma holds. We consider tensor product kernels of the form 
\begin{equation*}
\label{eq:kernflrkhs}
 \kerfuncseq{\xk}{\xkd}= \prod_{j=1}^d
\kerfuncentry(x_{(j)},x'_{(j)})
\end{equation*}
with inputs $\xk,\xkd \in \XX^{\otimes d} \subset \RR^d$ 
with $\XX$ compact and $\otimes d$ denotes the product space, for some kernel function $\kerfuncentry$ on $\XX$ which may 
change with $d$ (e.g. the scaling). In order to prevent the sequence of kernels $\kernf$ to diverge as $d \to \infty$, assume that there exists some probability measure on $\XX$ with full support such that the trace of the kernel operator is bounded by $1$, i.e. $\int \kerfuncentry(x,x) d\mu(x) \leq 1$. Let $\|. \|_{\Hilbertspace_{k}}$ be the Hilbert norm induced by $k$. Then,

\begin{lemma}[Formal statement of Lemma \ref{lm:rkhs}]
  \label{lm:rkhsformal}
Let $\kernf$ satisfy the above conditions. Then, for any $f$ that is a non-constant sparsely parameterized product
function $f(\xk) = \prod_{j=1}^m f_j(\xel{j})$ for some fixed $m \in
\N$, 
\vspace{-0.1in}

\begin{equation*}
  \| f \|_{\Hilbertspace_k} \overset{d \to \infty}{\to} \infty.
\end{equation*}
\end{lemma}

\begin{proof}
 For any $j>m$, define $f_j = 1$. First, we note that the proof follows trivially if any of the $f_j$ is not contained in the RKHS induced by $q$ since this implies that the Hilbert norm $\|f\|_{\Hilbertspace_{k}} = \infty$. Hence, we can assume that for all $j$, $f_j$ is contained in the RKHS for all $d$. Furthermore, because $k$ is a product kernel, we can write
 $\| f \|_{\Hilbertspace_{k}} = \prod_{j=1}^d \| f_j \|_{\Hilbertspace_{q}}$ where $\|. \|_{\Hilbertspace_{q}}$ is the Hilbert norm induced by $q$ on $\XX$. Because we are only interested to see whether the sequence of Hilbert norms diverge, without loss of generality we can assume that $m =1$, and hence,
\begin{equation}
\label{eq:rkhsnormeq}
 \| f \|_{\Hilbertspace_{k}} = \| f_1 \|_{\Hilbertspace_{q}} (\| 1\|_{\Hilbertspace_{q}})^{d-1}.
\end{equation}
  
Next, by Mercer's theorem there exists an orthonormal eigenbasis $\{ \phi_{i} \}_{i=1}^{\infty}$ in $\Ell_2(\XX,\mu)$ with corresponding eigenvalues $\{\lambda_{i} \}_{i=1}^{\infty}$ such that for any $g \in \Hilbertspace_q$,  $\| g \|_{\Hilbertspace_{q}} = \sum_{i=1}^{\infty} \frac{(\langle f, \phi_{i}\rangle)^2}{\lambda_{i}}$, where $\langle f, \phi_{i}\rangle = \int \phi_{i}(x) f(x) d\mu(x) $.
Note that  because the kernel $q$ depends on $d$, $\lambda_i$ and $\phi_i$ also depend on $d$. Next, because by assumption $f(x) =1$ is contained  in the RKHS, there exists $\alpha_i$ such that for every $x \in\XX$, $ 1 = \sum_{i=1}^{\infty} \alpha_{i} \phi_{i}(x)$ and $\sum_{i=1}^{\infty} \alpha_{i}^2 = 1$. Furthermore,
\begin{equation*}
\begin{split}
 1 \geq \int \kerfuncentry(x,x) d\mu(x)& = \int \sum_{i=1}^{\infty} \lambda_{i} \phi_{i}(x)\phi_{i}(x) d\mu(x) =  \sum_{i=1}^{\infty} \lambda_{i}.
\end{split}
\end{equation*}
Combining these results, we get that 
\begin{equation*}
\label{eq:rkhsnormeq2}
 \| 1\|_{\Hilbertspace_{q}} = \sum_{i =1}^{\infty} \frac{\alpha_{i}^2}{\lambda_{i}} \geq 1.
\end{equation*}
Furthermore, there exists $\beta_i$ such that $f_1(x) =
\sum_{i=1}^{\infty} \beta_{i} \phi_i(x)$. Again, because we are only interested to see whether the sequence of Hilbert norms diverge, without loss of generality
we can assume that $\sum_{i=1}^{\infty} \beta_{i}^2 = 1$ and hence
also $\| f_j \|_{\Hilbertspace_{q}} \geq 1$.

First, assume that there exists a subsequence such that $\| 1\|_{\Hilbertspace_{q}} \to 1$. This implies that
there exists a sequence $j_d \in \NN$ such that $\alpha_{j_d}^2 \to 1$
and $\lambda_{j_d} \to 1$. Next, because by assumption $f_1 \neq 1$,
there exists some constant $c_1>0$ such that for all $d$,
\begin{equation*}
  c_1 \leq \int (1-f_1(x))^2 d\mu(x)  = \sum_{i=1}^{\infty} (\alpha_{i} - \beta_{i})^2.
\end{equation*}
Together with the fact that $\alpha_{j_d}^2 \to 1$ it then follows that
$\sum_{i \neq j_d} \beta_{i}^2$ has to be asymptotically lower bounded
by some positive non-zero constant $c_2$ and hence
\begin{equation*}
 \|f_1\|_{\Hilbertspace_{q}} \geq \frac{c_2}{(1- \lambda_{j_d})} \to \infty.
\end{equation*}
This contradicts the assumption that $\|f_1\|_{\Hilbertspace_{q}}$ is
upper bounded by some constant for every $d$. Hence, we are only left with the case where $\|1\|_{\Hilbertspace_q} \geq c>1$, however, this case diverges due to Equation~\ref{eq:rkhsnormeq}. Hence, the proof is complete.
\end{proof}

%% file: sections/Appendix/Appendix_Theorem1.tex
\section{Proof of Theorem \ref{thm:main_thm}}
\label{sec:proofmainthm}



Before presenting the proof of the (generalized) theorem, we first state the key concentration inequalities used throughout 
the proof. It is an extension of Lemma A.2 in the paper \cite{ElKaroui10}, which iteself is a consequence of the concentration of
Lipschitz continuous functions of i.i.d random vectors.

\begin{lemma}
\label{lm:innerpr}
For any $\prob_X \in \probsetcov$ or $\probsetsphere$, let $\Xs \in \RR^{d\times n}$ consists of i.i.d.~vectors $x_i \sim \prob_X$ 
and $X \sim \prob_X$ be independent of $\x_i$.
For any constants $\epsilon > 0$, define the events
\begin{align}
 \EXset &\define \left\{\X ~\vert \max_{i,j }\left| \x_i^\top  \x_j / \tr(\Sigma_d) - \delta_{i, j}\right| \leq  n^{-\beta/2} (\log(n))^{(1+\epsilon)/2}\right\}\\
\Uset &\define \left\{ \Xrv~ \vert~ \left|~\|\Xrv\|_2^2/\trace(\inputcov)- 1\right|  \leq  n^{-\beta/2} (\log n)^{(1+\epsilon)/2} ~ \mathrm{and}~ \max_{i} |\x_i^\top \Xrv |/\trace(\inputcov)   \leq  n^{-\beta/2} (\log n)^{(1+\epsilon)/2} \right\}
\end{align}
Then, there exists some constant $C>0$ such that for $n$ sufficiently large, 
\begin{align}
\prob(\EXset) &\geq  1- n^2 \exp(-C  (\log(n))^{(1+\epsilon)}) \\
\mathrm{and}~~~\prob(\Uset\vert \EXset) &\geq  1- (n+1)^2 \exp(-C  (\log(n))^{(1+\epsilon)}) 
\end{align}
In particular, the event $\EXset$ holds almost surely with respect to the sequence of data sets $\Xs$ as $n\to \infty$, that is the probability that for infinitely many $n$, $\EXset$ does not hold, is zero. 
\end{lemma}
The proof of the lemma can be found in Section \ref{sec:technical_prel}.

\begin{proof}[Proof of Theorem \ref{thm:main_thm}]

The proof of the Theorem is primarily separated into two parts
\begin{itemize}
\item We first state Theorem \ref{prop:main_thm_ext} which shows under the weaker
  Assumption~\Ckone~that the results of
  \ref{thm:main_thm} hold for the ridge estimate $\fhatridge$ for non-vanishing $\lambda>0$
  or the ridgeless estimate whenever the
  eigenvalues of $\kermat$ are asymptotically lower bounded.
\item We finish the proof for the ridgeless estimate by invoking Theorem~\ref{prop:main_thm_ext}
  and showing that $\kermat$ indeed has asymptotically lower
  bounded eigenvalues under the stricter assumptions
  \Akone-\Akthree~imposed in Theorem~\ref{thm:main_thm}.
\end{itemize}
For the clarity we denote with \Akthree~the $\beta$-dependent assumptions in Theorem \ref{thm:main_thm}
\begin{itemize}
   \item[\Akthreebracket] \textit{$\beta$-dependent assumptions:} $g_i$ is $(\floor{2/\beta}+1-i)$-times
  continuously differentiable in a neighborhood of $(\cbw,\cbw)$ and there
  exists $\jthresh > \floor{2/\beta}$ such that
  $\rotkerfunc_{\jthresh}(\cbw,\cbw) >0$.
\end{itemize}
We start by introducing the following weaker assumptions that 
allows us to jointly treat $\alpha$-exponential
kernels and kernels satisfying Assumption \Akone-\Akthree~
when the kernel eigenvalues are lower bounded in
Theorem~\ref{prop:main_thm_ext}. Note that this assumption 
implies that the kernel is rotationally invariant.


\begin{itemize}

\item[\Ckonebracket] \textit{Relaxation of Assumption \Akone-\Akthree}: 
    Define the neighborhood $\convergesetres{\delta}{\delta'} \subset \R^d \times \R^d$  as
  \begin{equation*}
    \convergesetres{\delta}{\delta'} := \{(\xk,\xkd) \in  \R^d \times \R^d \mid (\|\xk\|_2^2, \|\xkd\|_2^2) \in [1-\delta, 1+\delta] \times
  [\cbw-\delta, \cbw+\delta], \xk^\top \xkd \in [-\delta',\delta']\}.
  \end{equation*}
  
   The kernel function $\kernf$ is rotationally invariant and there exists a function $\rotkerfunc$ such that $k(\xk.\xkd) = \rotkerfunc(\|\xk\|_2^2, \|\xkd\|_2^2, \xk^\top \xkd)$. Furthermore, $\rotkerfunc$ can be expanded as a power series of the form
  \begin{equation}
    \label{eq:kerneldefgap}
    \kerfunc{\xk}{\xkd}= \rotkerfunc(\|\xk\|_2^2, \|\xkd\|_2^2, \xk^\top \xkd) = \sum_{j=0}^{m} \rotkerfunccoeffj{j}(\|\xk\|_2^2,\|\xkd\|_2^2) (\xk^\top \xkd)^j + (\xk^\top \xkd)^{m+1}  \remainder (\|\xk\|_2^2, \|\xkd\|_2^2, \xk^\top \xkd)
   \end{equation}
   with $\mc= \floor{2/\beta}$ that converges in a neighborhood $\convergeset{\delta,\delta'}$ of the sphere for some $\delta, \delta'>0$ and where $g_i$ is
  $(\floor{2/\beta}+1-i)$-times continuously differentiable in an
   neighborhood of $(\cbw,\cbw)$ and the remainder term $ \remainder $ is a continuous function around the point $(\cbw,\cbw,0)$.

\end{itemize}

\begin{theorem}[Polynomial approximation barrier]
  \label{prop:main_thm_ext}
  Assume that the kernel $\kernf$, respectively its restriction onto the unit sphere, satisfies Qssumption \Ckone~ and that the eigenvalues of $\kermat+\lambda \idmat_n$ are almost surely lower bounded by a positive constant with respect to the sequence of datasets $\Xs$ as $n\to \infty$.
  Furthermore, assume that the ground truth $\ftrue$ is bounded
  and the input distribution satisfies \Adone-\Adtwo.
  Then, for $m =
  2 \floor{2/\beta}$ for $\prob_X \in \probsetcov$ and $m =
  \floor{2/\beta}$ for $\prob_X \in \probsetsphere$, the following results hold for both the ridge~\eqref{eq:krr} and ridgeless estimator~\eqref{eq:constraint_min_p} $\fhatridge$ wiht $\lambda \geq 0$. 
  \begin{enumerate}

  \item The bias of the kernel estimators $\fhatridge$ is
    asymptotically lower bounded, for any $\eps > 0$, 
    \vspace{-0.05in}
  \begin{equation}
    \label{eq:apbiaslb}
    \Bias(\fhatridge) \geq \underset{\polyf \in
      \polyspace{m}}{\inf} \| \ftrue -
    \polyf \|_{\Ell_2(\prob_X)} - \eps  ~~ a.s. ~\mathrm{ as }~ n \to \infty.
  \end{equation}
\item We can find a polynomial $\polyf$ such that for any $\epsilon,\epsilon'>0$, there exists $C>0$ such that asymptotically with probability $\geq 1 - n^2 \exp(-C  (\log(n))^{1+\epsilon'})$ over the draws of $X$, 
\begin{equation}
\label{eq:appolyapprweak}
     \left| \EEobs \fhatridge(X) - \polyf(X) \right| \leq \epsilon  ~~ a.s. ~\mathrm{ as }~ n \to \infty.
\end{equation}
Furthermore, for bounded kernel functions on the support of $\prob_X$ the averaged estimator $\EEobs \fhatridge$ converges in $\Ell_2(\prob_X)$ to a polynomial $\polyf  \in \polyspace{m}$,
  \vspace{-0.1in}
  \begin{equation}
  \label{eq:appolyapproxlp}
    \left\| \EEobs \fhatridge - \polyf \right\|_{\Ell_2(\prob_X)} \to 0  ~~ a.s. ~\mathrm{ as }~ n \to \infty.
  \end{equation}
  \end{enumerate}
  \vspace{-0.1in}
\end{theorem}

The proof of this theorem can be found in Section~\ref{sec:prop1_proof}.
%
Theorem~\ref{prop:main_thm_ext} states Theorem~\ref{thm:main_thm}
under the assumption that $(\kermat+\lambda I)$ has asymptotically
lower bounded eigenvalues and the weaker Assumption \Ckone.  For the
proof of Theorem~\ref{thm:main_thm}, it remains to show that
Assumptions \Akone-\Akthree~of Theorem ~\ref{thm:main_thm} and the
$\alpha$-exponential kernel both
\begin{itemize}
\item[(a)] satisfy Assumption \Ckone~ and 
\item[(b)] induce kernel matrices with almost surely asymptotically positive lower bounded eigenvalues
\end{itemize}


Point (a) is relatively simple to prove and deferred to
Section~\ref{sec:AC_lemmas}. The bulk of the work in fact lies in
showing (b) separately for the case for \Akone-\Akthree~ and
$\alpha$-exponential kernels with $\alpha \in (0,2)$ in the following
two propositions, as these two cases require two different proof
techniques. 




\begin{prop}
\label{prop:EVlowerbound}

Assume that the kernel $\kernf$, respectively its restriction onto the unit sphere, satisfies Assumption \Akone-\Akthree~ and the distribution $\prob_X$ satisfies \Bdone-\Bdtwo.
Then, for any $\gamma >0$ and $m = \floor{2/\beta}$, conditioned on $\EXset$, 
\begin{equation}
\label{eq:lambdamin}
    \lambda_{\min}(K) \geq g(\cbw,\cbw,\cbw) - \sum_{i =0}^m g_i(\cbw,\cbw) - \gamma >0
\end{equation}
where $\lambda_{\min}(K)$ is the minimum eigenvalue of the kernel matrix $K$. 
\end{prop}



\begin{prop}
\label{prop:alphaexp}
Assume that the Assumptions B.1-B.2 hold true. Then, the minimum
eigenvalue of the kernel matrix of the $\alpha$-exponential kernel
with $\alpha\in(0,2)$ is lower bounded by some 
positive constant almost surely as $n \to \infty$.  
\end{prop}

The proof of the Propositions~\ref{prop:EVlowerbound} and~\ref{prop:alphaexp} can be found in
the Sections~\ref{sec:proofa1toa3} and~\ref{sec:proofalphaexp} respectively which concludes the proof of
the theorem.
\begin{remark}
 The almost sure statement in Proposition \ref{prop:alphaexp} can also be replaced with an in probability statement as in Lemma \ref{lm:innerpr} and hence also the statements in Theorem \ref{thm:main_thm}.
\end{remark}

\end{proof}

\subsection{Proof of Theorem \ref{prop:main_thm_ext}}
\label{sec:prop1_proof}
As a result of Lemma \ref{lm:innerpr} it is sufficient to condition  throughout the rest of this proof on the intersection of the events $\EXset$ and the event where the eigenvalues of the kernel matrix $\kermat$ are lower bounded by a  positive constant. 

For simplicity of notation, we define $\z_i  = \frac{\x_i}{\sqrt{\bw}}$ and let $\Z$ be the $d\times n$ matrix with column vectors $\z_i$. Define the random variable $Z = X/\sqrt{\bw}$ with $\Xrv \sim \prob_X$ and  denote with $\prob_Z$ the probability distributed of $\Zrv$. Define the event $\UsetZ$ in the same way as $\Uset$ for the normalised inputs $\z_i,\Zrv$ and $\EZset$ like $\EXset$.  In the latter, we denote with $ a \lesssim b$ that there exists a constant $C>0$ such that $a \leq C b$ with $C$ independent of $n,d$.  Furthermore, we make heavily use of the closed form solution for the estimator $\fhatridge$, 
\begin{equation*}
 \EEobs \fhatridge(X) = \fstar(\Xs)^\top(\kermat +\lambda \idmat_n)^{-1}\kvec
\end{equation*}
with $\kvec \in \RR^n$ the vector with entries $(\kvec)_i = \kernf_{\bw}(\x_i,X) = k(\z_i,\Zrv)$ and $\fstar(\Xs)$ the vector with entries $\fstar(\Xs)_i = \fstar(x_i)$. This equation holds true for any $\lambda \geq 0$ and is a well known consequence of the representer theorem.
%

The idea of the proof is to decompose the analysis into the term emerging from the error in the high probability region $\UsetZ$ and the error emerging from the low probability region $\UsetZ^{\comp}$. The proof essentially relies on the following lemma.
\begin{lemma}
\label{lm:lmmainthm}
 We can construct a polynomial $\polf$ of degree $\leq m$ such that for $n \to \infty$, 
 \begin{enumerate}
  \item $\vert \polf(\Zrv) - \EEobs \fhatridge(\sqrt{\bw}\Zrv) \vert \to 0$,  uniformly for all $\Zrv \in \UsetZ$
  \item $\| \indicator_{\Zrv \in \UsetZ^{\comp}} \polf \|_{\Ell_2} \to 0$
 \end{enumerate}
\end{lemma}
The proof of the lemma can be found in Section \ref{sec:prooflmmain}.
As a result, Equation~\ref{eq:appolyapprweak} follows immediately and Equation~\ref{eq:appolyapproxlp} is a consequence of
\begin{equation*}
\begin{split}
  \left\| \EEobs \fhatridge - \polyf \right\|_{\Ell_2(\prob_Z)}^{2} 
 \leq \EE_{\Zrv} \indicator_{\Zrv \in \UsetZ}(\fubias(\sqrt{\bw}\Zrv) - \polyf_z(\Zrv))^2 +  \EE_{\Zrv} \indicator_{\Zrv \notin \UsetZ}(\polyf_z(\Zrv))^2+ \EE_{\Zrv} \indicator_{\Zrv \notin \UsetZ}(\fubias(\sqrt{\bw}\Zrv))^2  
\end{split}
\end{equation*}
The first two terms vanish due to Lemma \ref{lm:lmmainthm}. To see that the third term vanishes, note that for $n$ sufficiently large, 
\begin{equation*}
\begin{split}
&\EE_{\Zrv} \indicator_{\Zrv \notin \UsetZ}(\fubias(\sqrt{\bw}\Zrv))^2 =  \EE_{\Zrv} \indicator_{\Zrv \notin \UsetZ}( y^\top(\kermat +\lambda \idmat_n)^{-1}\kvec)^2 \\ \lesssim &\frac{n^2}{c_{\lambda_{\min}}^2} \EE_{\Zrv} \indicator_{\Zrv \notin \UsetZ}  \max_i \vert \kernf(\z_i, \Zrv)\vert^2 \lesssim n^2 P(\UsetZ^{\comp}) \to 0
\end{split}
\end{equation*}
where we have used in the first inequality that by assumption $\vert \ftrue \vert$ is bounded on the support of $\prob_X$ and that $\lambda_{\min}(K + \lambda I_n)\geq \minEV >0  $ and in the second inequality that $\vert \kernf \vert$ is bounded.  Finally, the convergence to zero is due Lemma \ref{lm:innerpr}.

Next, the lower bound for the bias.
Due to Lemma \ref{lm:alphaexpbound}, we have that
\begin{equation*}
 \underset{\Zrv \in \UsetZ}{\max} ~ \left| \polf(\Zrv) - \EEobs\fhatridge(\sqrt{\bw}\Zrv) \right| = \underset{\Zrv \in \UsetZ}{\max} ~ \left| \polf(\Zrv) -\fstar(\sqrt{\bw}\Z)^\top (K + \lambda \In)^{-1} \kvec \right| \to 0  ,
\end{equation*}
and hence, for any $\gamma_1 >0 $ and $n$ sufficiently large, 
\begin{equation*}
   \EE_{\Zrv} \indicator_{\Zrv \in \UsetZ} \left( \ftrue(\sqrt{\bw}\Z)^\top (K + \lambda \In)^{-1} \kvec - \ftrue(\sqrt{\bw}\Zrv) \right)^2  \geq \EE_{\Zrv}\indicator_{\Zrv \in \UsetZ} \left( \polf(\Zrv) - \ftrue(\sqrt{\bw}\Zrv) \right)^2   - \gamma_1  ,
\end{equation*}
Furthermore,  due to the second statement in Lemma \ref{lm:alphaexpbound}, we know that $\EE_{\Zrv}\indicator_{\Zrv \notin \UsetZ} (\polf(\Zrv))^2 \to 0 $ and because $\ftrue$ is bounded by assumption, we can see that $\EE_{\Zrv}\indicator_{\Zrv \notin \UsetZ} (\polf(\Zrv) - \ftrue(\sqrt{\bw}\Zrv))^2 \to 0 $. Since $\fhatridge$ only depends linearly on the observations $\y$, we have $\Bias(\fhat) =  \EE_{\Zrv}( \ftrue(\sqrt{\bw}\Z)^\top (K + \lambda \In)^{-1} \kvec - \fstar( \sqrt{\bw} \Zrv) )^2$. Thus, as a result, for any $\gamma_2 >0$, 
\begin{equation*}
    \begin{split}
        \Bias(\fhat) 
        &\geq \EE_{\Zrv}\indicator_{\Zrv \in \UsetZ} \left( \polf(\Zrv) - \ftrue(\sqrt{\bw}\Zrv) \right)^2 - \gamma_1\\
        &\geq \EE_{\Zrv}\indicator_{\Zrv \in \UsetZ} \left( \polf(\Zrv) - \ftrue(\sqrt{\bw}\Zrv) \right)^2  
+ \EE_{\Zrv}\indicator_{ \Zrv \notin \UsetZ} \left( \polf(\Zrv) - \ftrue(\sqrt{\bw}\Zrv) \right)^2 - \gamma_1 - \gamma_2\\
        &= \EE_{\Zrv} \left( \polf(\Zrv) - \ftrue(\sqrt{\bw}\Zrv) \right)^2  -\gamma_1 - \gamma_2 \quad a.s.
    \end{split}
\end{equation*}
Thus, the result follows from the definition of the infimum. 
\qed

%


\subsection{Proofs for the lower bound of the eigenvalues}
\label{sec:PropB1}

\subsubsection{Proof of Proposition \ref{prop:EVlowerbound}}
\label{sec:proofa1toa3}
We use the same notaiton as used in the proof of Theorem \ref{prop:main_thm_ext}. 
As a result of Lemma \ref{lm:innerpr} it is sufficent to condition on $\EXset$ throughout the rest of this proof. 
The proof follows straight forwardly from the following Lemma \ref{lm:elkarouiextension} which gives an asymptotic description of the kernel matrix $K$ based on a similar analysis as the one used in the proof of Theorem~2.1 and 2.2 in the paper \cite{ElKaroui10}. In essence, it is again a consequence of the concentration inequality from Lemma \ref{lm:innerpr} and the stronger Assumption \Akone-\Akthree~and in particular the power series expansion of $g$. We denote with $\circ i$ the $i$-times Hadamard product.

\begin{lemma}
\label{lm:elkarouiextension}
Given that the assumption in Proposition \ref{prop:EVlowerbound} hold. 
For $m = \floor{2/\beta}$,
  \begin{equation*}
    \opnorm{ K - M} \to 0 ~ 
\end{equation*}  
with
\begin{equation}
\label{eq:elkarouiM}
\begin{split}
    \mmat&= \idmat\left(g(\cbw,\cbw,\cbw)  - \sum_{q =0}^m g_q(\cbw,\cbw) \right) +  \sum_{q=0}^m (\Z^\top\Z)^{\circ q} \circ \kermatg_{g_q},
\end{split}
\end{equation}
where $\kermatg_{g_q}$ is the positive semi-definite matrix with entries $(\kermatg_{g_q})_{i,j} = g_l(\| z_i \|_2^2, \| z_j\|_2^2)$.
\end{lemma}
The proof of the lemma can be found in Section \ref{sec:proofev}.
The proof of Proposition \ref{prop:EVlowerbound} then follows straight forwardly when using Schur's product theorem which shows that 
\begin{equation*}
 \idmat\left(g(\cbw,\cbw,\cbw)  - \sum_{q =0}^m g_q(\cbw,\cbw) \right) +  \sum_{q=0}^m (\Z^\top\Z)^{\circ q} \circ \kermatg_{g_q} \succeq \idmat\left(g(\cbw,\cbw,\cbw)  - \sum_{q =0}^m g_q(\cbw,\cbw) \right)
\end{equation*}
where we use that $g_i$ is positive semi-definite by Assumption \Akone. 
To see that the eigenvalues are lower bounded, we thus simply need to show that $g(\cbw,\cbw,\cbw)  - \sum_{q =0}^m g_q(\cbw,\cbw) > 0$. This holds because the positive semi-definiteness of $g_q$ implies that $g_q(\cbw,\cbw) \geq0$ and hence $g(\cbw,\cbw,\cbw) = \sum_{q=0}^{\infty} g_q(\cbw,\cbw)$ is a sum of positive coefficients and because by Assumption \Akthree~there exists $\jthresh > \floor{2/\beta}$ such that $g_{\jthresh}(\cbw,\cbw) >0$. Hence, there exists a positive constant $c>0$ such that $\lambda_{\min}(M) \geq c$. We can conclude the proof when applying Lemma \ref{lm:elkarouiextension}, which implies that $\lambda_{\min}(\kermat) \to \lambda_{\min}(M) $ as $n\to \infty$. 
 \qed

\subsubsection{Proof of Proposition \ref{prop:alphaexp}}
\label{sec:proofalphaexp}
We use the same notaiton as used in the proof of Theorem \ref{prop:main_thm_ext} and define $\Dalpha$ to be the $n\times n$ matrix with entries $(\Dalpha)_{i,j} = d_{\alpha}(z_i, z_j) := \| z_i - z_j \|_2^{\alpha}$. We separate the proof into two steps. In a first step, we decompose $\kernf$ in the terms
\begin{equation*}
 \exp(-\lvert\lvert \xk-\xkd\rvert\rvert_2^{\alpha}) =
\exp(\tilde{k}(\xk,\xkd)) \exp(-\lvert\lvert
\xk-\xkd\rvert\rvert_2^{\alpha} - \tilde{k}(\xk,\xkd)) 
\end{equation*}
such that $\exp(-\lvert\lvert \xk-\xkd\rvert\rvert_2^{\alpha} -
\tilde{k}(\xk,\xkd))$ and $\exp(\tilde{k}(\xk,\xkd))$ are both positive semi-definite kernel functions. In particular, we construct $k$ such that the eigenvalues of the kernel matrix $A$ of  $\exp(\tilde{k}(\xk,\xkd))$ evaluated at $\Zs$ are almost surely lower bounded by a positive constant. The proposition is then a straight forward consequence and shown in the last step.
\\\\
\textbf{Step 1:} 
We can see from Chapter~3
Theorem~2.2 in \cite{Berg84} that $d_{\alpha}$ is a conditionally negative semi-definite function, that is that for any  $m\in \NN\setminus 0$ and any $\{x_1,...,\x_m\}$, the corresponding kernel matrix $A$ is conditionally negative semi-definite. A matrix $A$ is conditionally negative semi-definite if for every $v \in \mathbb{R}^n$ with $1^\top v =
0$, $v^\top A v \leq 0$.
As shown in Chapter~3 Lemma~2.1
in \cite{Berg84}, a kernel function $\phi(\xk,\xkd)$ is conditionally
negative semi-definite, if and only if for any $z_0$, $(\xk,\xkd) \to  \phi(\xk,\z_0) +
\phi(\z_0,\xkd) -\phi(\xk,\xkd) - \phi(\z_0,\z_0)$
is a positive semi-definite
function. Hence, for any $\z_0\in \RR^n$, the kernel defined by 
\begin{equation}
 \tilde{\kernf}(\xk,\xkd) =   d_{\alpha}(\xk,\z_0) +
d_{\alpha}(\z_0,\xkd) -d_{\alpha}(\xk,\xkd) - d_{\alpha}(\z_0,\z_0)
\end{equation}
 is positive semi-definite. 
%
%
%
The goal is now to show that we can find a vector $z_0$ such that the kernel matrix $A$ of $\tilde{\kernf}$ evaluated at $\Zs$ has eigenvalues almost surely lower bounded by some  positive constant. Essentially, the statement is a consequence of the following lemma, bounding the eigenvalues of $\Dalpha$.
\begin{lemma}
\label{lm:alphaexpbound}
Assume that $\prob_X$ satisfies the Assumption \Bdone-\Bdtwo. Conditioned on $\EXset$, for any $n$ sufficiently large, all eigenvalues of the matrix $\Dalpha$ are bounded away from zero by a positive  constant $c>0$, i.e. $\underset{i \leq n}{\min} \left|\lambda_i(\Dalpha) \right| \geq c$. 
\end{lemma}
The proof of the lemma can be found in Section \ref{sec:proofalpha}.
In particular, note that we can use the same argument as used in Lemma \ref{lm:innerpr} to show that there exists almost surely over the draws of $\Zs$ as $n \to \infty$ an additional vector $z_0$, such that for any two vectors $\zk,\zkd \in \Zs \cup\{\z_0\}$, 
\begin{equation}
\label{eq:tmpconineq}
    | \zk^\top \zkd - \delta_{\zk = \zkd} | \lesssim n^{-\beta/2} (\log(n))^{(1+\epsilon)/2}.
\end{equation}
Throughout the rest of this proof, we conditioned on the event $\EXset$ and the additional event that Equation \eqref{eq:tmpconineq} holds, and remark that the intersection of these two events holds true almost surely as $n\to \infty$. It is then straight forward to show that the eigenvalues of the matrix
\begin{equation*}
 \Dalpha(\Zs,\z_0) = \begin{pmatrix} 
     \Dalpha & d_{\alpha}(\Zs, \z_0) \\ 
     d_{\alpha}(\Zs, \z_0)^\top & d_{\alpha}(\z_0,\z_0) 
     \end{pmatrix}
\end{equation*}
are also bounded away from zero by a positive constant $\tilde{c}>0$. 
 Therefore, for any $v \in \RR^n$, 
     \begin{equation}
     \label{eq:alpha1}
   \begin{pmatrix} v^{T} &  -1^\top  v\end{pmatrix}   \Dalpha(\Zs,\z_0)\begin{pmatrix} v^{T} \\  -1^\top v \end{pmatrix} \leq -\tilde{c} \| \begin{pmatrix} v^{T} \\  1^\top v \end{pmatrix}\|_2^2,
\end{equation}
where we have used that $1^T \begin{pmatrix} v^{T} \\  -1^\top v \end{pmatrix} = 0$. As a result, we can see that
\begin{equation}
\label{eq:alpha2}
\begin{split}
      &\begin{pmatrix} v^{T} &  -1^\top  v\end{pmatrix}   \Dalpha(\Zs,\z_0)
      \begin{pmatrix} v^{T} \\  -1^\top v \end{pmatrix} \\  
                = &v^\top  \Dalpha v - v^\top  \left[\frac{1}{n} 1 1^\top d_{\alpha}(\z_0, \Zs) \right] v - v^\top  \left[ \frac{1}{n} d_{\alpha}(\z_0, \Zs)^\top  1 1^\top  \right] v  + v^\top  \left[ \frac{1}{n^2} 1 1^\top  d_{\alpha}(\z_0,\z_0) 1 1^\top  \right] v. \\ 
                = & v^\top \underbrace{\left[  \Dalpha  -  \frac{1}{n} 1 1^\top d_{\alpha}(\z_0, \Zs)   -    \frac{1}{n} d_{\alpha}(\z_0, \Zs)^\top  1 1^\top     +    \frac{1}{n^2} 1 1^\top  d_{\alpha}(\z_0,\z_0) 1 1^\top   \right]}_{= -A} v,
\end{split}
\end{equation}
where $A$ is exactly the kernel matrix of $\tilde{k}$ evaluated at $\Zs$. 
Hence, combining Equation \eqref{eq:alpha1} and \eqref{eq:alpha2} gives
\begin{equation*}
 v^T A v \geq \tilde{c}~ (v^Tv + v^T11^Tv) \geq  \tilde{c} ~v^Tv.
\end{equation*}
We can conclude the first step of the proof when applying the Courant–Fischer–Weyl min-max principle which shows that $A$ has lower bounded eigenvalues $\geq \tilde{c}$. \\\\
\textbf{Step 2:}
We can write $\exp(-d_{\alpha}(x,\xkd)) = \exp(\tilde{k}(x,\xkd)) \exp(\phi(\xk,\xkd)) $ with $\phi(x,\xkd) := -d_{\alpha}(x,\xkd) - \tilde{k}(x,\xkd) =  d_{\alpha}(\z_0,\z_0)    -d_{\alpha}(x,z_0)  -d_{\alpha}(z_0,\xkd) $. It is straight forward to verify that $\exp(\phi(x,\xkd))$ is a positive semi-definite function. Hence, due to Schur's product theorem, the following sum is a sum of positive semi-definite functions
\begin{equation*}
    \exp(-d_{\alpha}(\xk,\xkd)) = \exp(\tilde{k}(\xk,\xkd)) \exp(\phi(x,\xkd)) = \sum_{l=0}^{\infty} \frac{1}{l!} \tilde{k}(x,\xkd)^l \exp(\phi(x,\xkd)).
\end{equation*} 
 It is sufficient to show that the eigenvalues of the kernel matrix $M$ of $\tilde{k}(x,\xkd) \exp(\phi(x,\xkd))$, evaluated at $\Zs$, are lower bounded. Let $B$ be the kernel matrix of $\exp(\phi(x,\xkd))$ evaluated at $\Zs$, we have that $M = A \circ B$, where $\circ$ is the Hadamard product and $A$ the kernel matrix of $\tilde{k}$ from the previous step. 
 We make the following claim from which the proof follows trivially using the fact shown in the first step that the eigenvalues of $A$ are lower bounded by a positive  constant.
\\\\
\textbf{Claim:} \textit{$B = \frac{1}{2e^2} 11^\top  + \tilde{B}$ with $\tilde{B}$ a positive semi-definite matrix.}
\\\\
\textbf{Proof of the claim:}
 Let $\psi$ be the vector with entries $\psi_i = \exp(-  d_{\alpha}(\z_i,\z_0))$. Furthermore, let ${\gamma = \exp(d_{\alpha}(\z_0, \z_0))}$. We can write
\begin{equation*}
    B = \gamma \left(1 \psi^\top \right) \circ \left( \psi 1^\top  \right) = \gamma \psi \psi^\top .
\end{equation*}
Next, using Lemma \ref{lm:innerpr} and the fact that
 $d_{\alpha}(x,\xkd) =  2^{\alpha/2} + O(\frac{\lvert\lvert x-\xkd\rvert\rvert_2^2}{2} -1 )$, we can see that $ {\gamma \geq \exp( 2^{\alpha/2}/2) >1}  $. Hence, it is sufficient to show that $  \psi \psi^\top  - \frac{1}{2e^{2}} ~ 1 1^\top  $ is positive semi-definite. This is true if and only if $1^\top  \psi \psi^\top  1 \geq \frac{1}{2e^{2}} 1^\top 11^\top  1$, which is equivalent to saying that $ \left( \sum_{i=1}^{n} \exp(-d_{\alpha}(\z_i, \z_0)) \right)^{2} \geq \frac{n^2}{2e^{2}}$. Using again the same argument as for $\gamma$, we can see that $\underset{i}{\max}\left|  2^{\alpha/2} -  d_{\alpha}(z_i, \z_0) \right| \to 0  $ for any $i $, which completes the proof.
\qed

\subsection{Proof of Corollary \ref{cor:kernels}}
\label{sec:proofcrkernels}

 First, note that the Assumption \Akone-\Akthree~straight forwardly hold true for the exponential inner product kernel with $\kernf(\xk,\xkd) = \exp(\xk^\top \xkd) = \sum_{j=0}^{\infty} \frac{1}{j!} (\xk^\top \xkd)^j$ and for the Gaussian kernel with 
\begin{equation*}
 \kernf(\xk,\xkd) = \exp(-\|\xk -\xkd \|_2^2 ) = \sum_{j=0}^{\infty} \frac{2^j}{j!} (\xk^\top \xkd)^j \exp(-\|\xk\|_2^2)\exp(-\|\xkd\|_2^2). 
\end{equation*}
Next, note that the $\alpha$-exponential kernel with $\alpha <2$ is already explicitly covered in Theorem \ref{thm:main_thm}. Hence, the only thing left to show is that Theorem \ref{thm:main_thm} also applies to ReLU-NTK.

We use the definition of the Neural Tangent Kernel presented in \cite{Arora19,Lee18}. Let $L$ be the depth of the NTK and $\sigma: \mathbb{R} \to \mathbb{R}$ the activation function which is assumed to be almost everywhere differentiable. For any $i>0$, define the recursion
\begin{equation*}
  \begin{split}\Sigma^{(0)}(\xk,\xkd) &:= \xk^\top  \xkd \\
\Lambda^{(i)}(\xk,\xkd) &:= \begin{pmatrix} \Sigma^{(i-1)}(\xk,\xk) & \Sigma^{(i-1)}(\xk,\xkd) \\ \Sigma^{(i-1)}(\xk,\xkd) & \Sigma^{(i-1)}(\xkd,\xkd)\end{pmatrix} \\
\Sigma^{(i)}(\xk,\xkd) &:= c_{\sigma} \underset{(u,v) \sim \mathcal{N}(0, \Lambda^{(i)})}{\mathbb{E}}~\left[ \sigma(u) \sigma(v) \right] 
\end{split}
\end{equation*}
with $c_{\sigma} := \left[ \underset{v \sim \mathcal{N}(0, 1)}{\mathbb{E}}~\left[ \sigma(v)^2 \right] \right]^{-1}$. 
Furthermore, define
\begin{equation*}
     \dot{\Sigma}^{(i)} := c_{\dot{\sigma}} \underset{(u,v) \sim \mathcal{N}(0, \Lambda^{(i)})}{\mathbb{E}} ~\left[ \dot{\sigma}(u) \dot{\sigma}(v) \right] 
\end{equation*}
with $c_{\dot{\sigma}} := \left[ \underset{v \sim \mathcal{N}(0, 1)}{\mathbb{E}}~\left[ \dot{\sigma}(v)^2 \right] \right]^{-1}$ 
where $\dot{\sigma}$ is the derivative of $\sigma$. The NTK $\kntk$ of depth $L\geq 1$ is then defined as
\begin{equation*}
     \kntk(\xk,\xkd) := \sum_{i=1}^{L+1} \Sigma^{(i-1)}(\xk,\xkd) \prod_{j=i}^{L+1} \dot{\Sigma}^{(j)}(\xk,\xkd). 
\end{equation*}
We call a function $\sigma$ \textit{k-homogeneous}, if for any $\xk \in \mathbb{R}$ and any $a >0$, $\sigma(a\xk) = a^k \sigma(\xk)$. We can now show the following result from which the corollary follows. 

\begin{prop}
\label{prop:ntkass}
Assume that the activation function $\sigma$ is $k$-homogeneous and both the activation function and its derivative possess a Hermite-polynomial series expansion (see \cite{Daniely16}) where there exits $\jthresh \geq \floor{2/\beta}$ such that the $\jthresh$-th coefficient $a_{\jthresh} \neq 0$. Then, the NTK satisfies the Assumption \Akone-\Akthree~and hence Theorem \ref{thm:main_thm} applies. 
\end{prop}

In fact, we can easily see that any non linear activation function which is homogeneous and both the activation function and its derivative possesses a Hermite polynomial extension satisfies the assumptions in Proposition \ref{prop:ntkass}. In particular, this includes the popular ReLU activation function $\sigma(x) = \max(x,0)$ where the explicit expression for the Hermite polynomial extension can be found in \cite{Daniely16}.
\qed

\subsubsection{Proof of Proposition \ref{prop:ntkass}}


Essentially, the proof follows from the power series expression of the NTK $\kntk$ given in the following lemma. 
\begin{lemma} 
\label{lm:ntk1}
The NTK $\kntk$ possesses a power series expansion
\begin{equation*}
\label{eq:ntk1}
    \kntk(\xk,\xkd) = \sum_{j=0}^{\infty} (\xk,\xkd)^j g_j(\lvert\lvert \xk \rvert\rvert_2^2, \lvert\lvert \xkd \rvert\rvert_2^2),
\end{equation*}
which converges for any $\xk,\xkd \in \RR^d$ with $\xk,\xkd \neq 0$. Furthermore, for any $\uk,\ukd \in \RR_+$,
\begin{equation*}
    \label{eq:ntkgi}
    g_j(\uk,\ukd) = \sum_{i = -\infty}^{\infty} \eta_{j,i} (\uk \ukd)^{i/2}
\end{equation*}
and $\eta_{j,i} \geq 0 $. 
\end{lemma} 
The proof of the lemma can be found in Section \ref{sec:proofntk}.
It is straight forward to verify from the proof of Lemma \ref{lm:ntk1} that $\Sigma^{(i)}$ and $\dot{\Sigma}^{(i)}$ are compositions of continuous functions and thus  $\kntk$ is also continuous for any $x,x' \neq 0$. Next, note that the Lipschitz continuity (Assumption \Aktwo) follows straight forwardly from Equation \eqref{eq:ntklipschitz} in the proof of Lemma \ref{lm:alphaexpbound}. 
In order to show that any $g_j$ from Lemma \ref{lm:ntk1} is smooth, recall that 
\begin{equation*}
    g_j(\uk,\ukd) = \sum_{l=-{\infty}}^{\infty} \eta_{j,l}^{(i+1)} (\uk\ukd)^{l/2} =: h_j(x y).
\end{equation*}
Therefore, $h_j$ is a Puiseux power series with divisor $2$. Furthermore, the function $\tilde{h}_j(t) := h_j(t^2) = \sum_{l=-{\infty}}^{\infty} \eta_{j,l}^{(i+1)} (t)^{l}$ is a Laurent series which converges for every $t \neq 0$. Hence, we can conclude that $\tilde{h}_j$ is smooth for any $t \neq 0$ and thus also $h_j$. Finally, because $(\uk,\ukd) \to \uk\ukd$ is trivially also a smooth function, we can conclude that any $g_j$ is a smooth function for any $\uk,\ukd \neq 0$. 
Next, since for any $l \in \mathbb{Z}$, $(\uk,\ukd) \to \alpha (\uk\ukd)^{l/2}$ is trivially a positive semi-definite (PSD) function whenever $\alpha \geq 0$ and sums of PSD functions are again PSD, we can conclude that the $g_j(\uk,\ukd) = \sum_{l=-{\infty}}^{\infty} \eta_{j,l}^{(i+1)} (\uk\ukd)^{l/2} $ is PSD for any $j$. Therefore, we can conclude that Assumption \Akone~ holds as well. 

The only thing left to show is Assumption \Akthree. While we have already shown that $g_j$ are smooth in a neighborhood of $(1,1)$, we still need to show that there exists $\jthresh> \floor{2/\beta}$ such that $g_{\jthresh}(1,1) >0$. However, this follows from the fact that by assumption there exists $\jthresh> \floor{2/\beta}$ such that $a_{\jthresh} \neq 0 $ where $a_j$ are the Hermite coefficients of the activation function $\sigma$.
\qed

%% file: sections/Appendix/Appendix_Theorem2.tex
\section{Different scalings $\bw$}
\label{sec:Aotherbw}

In this section, we present results for different choices of the scaling beyond the standard choice $\bw \asymp \effdim$. In Subsection \ref{sec:flatlimit}, we give a proof of Theorem \ref{thm:otherchoicesoftau} describing the \textit{flat limit}, i.e. the limit of the interpolant where for any fixed $n,d$, $\bw \to \infty$. Furthermore, in order to get a more comprehensive picture, we additionally present straight forward results for other choices of $\bw$ in Section \ref{sec:simplertaucase}.

\subsection{Proof of Theorem \ref{thm:otherchoicesoftau}} 
\label{sec:flatlimit}
We use again the same notation as used for the proof of Theorem \ref{prop:main_thm_ext} where we  set $\z_i = \x_i/\sqrt{\effdim}$ and let $\Zrv = \Xrv/\sqrt{\effdim}$ be the random variable with $X \sim \prob_X$. We can again condition throughout the proof on the event $\EXset$. In particular, we assume throughout the proof that $n$ is sufficiently large since we are only interested in the asymptotic behaviour.  Furthermore, recall the definition of $\Dalpha$, which is the $n\times n$ matrix with entries $(\Dalpha)_{i,j} = \| \z_i-\z_j \|_2^{\alpha}$ and denote with $\Dalpha^{-1}$ its inverse. In addition, denote with $ \dvec$ the vector with entries $\dveci{i} = \| \z_i - \Zrv \|_2^{\alpha}$ and with $d_{\alpha}$ the function $d_{\alpha}(\zk,\zkd) = \| \zk-\zkd\|_2^{\alpha}$.


First, although the limit $\lim_{\tau \to \infty} K^{-1}$ does not exists, we can apply Theorem 3.12 in~\cite{Lee2014} to show that the flat limit interpolator $\flatlimit := \lim_{\bw \to \infty} \fhatinterpol$ of any kernel satisfying the assumption in Theorem \ref{thm:otherchoicesoftau} exists and has the form 
\begin{equation*}
    \flatlimit(\Zrv) = \begin{pmatrix}Y & 0 \end{pmatrix} \begin{pmatrix} -\Dalpha & 1 \\ 1^T & 0 \end{pmatrix}^{-1} \begin{pmatrix} \dvec \\ 1 \end{pmatrix}.
\end{equation*}
Furthermore, for the $\alpha$-exponential kernel, we use Theorem 2.1 in~\cite{Blumenthal60} to show that it satisfies the assumptions imposed on the eigenvalue decay in Theorem \ref{thm:otherchoicesoftau}.

\begin{remark}
 The estimator $\flatlimit$ is also called the \textit{polyharmonic spline interpolator}. This estimator is invariant under rescalings of the input data which is also the reason why we can rescale the input data by $\sqrt{\effdim}$, i.e. consider $ \z_i =  \x_i/\sqrt{\effdim}$ as input data points. 
\end{remark}

We already know from lemma \ref{lm:alphaexpbound} that the matrix $\Dalpha$  has $n-1$ negative eigenvalues and one positive eigenvalue. In particular,  we have shown that $ \left| \lambda_i(\Dalpha) \right| \geq \tilde{c} >0$, for $i$, where $\tilde{c}$ is some positive constant. Next, note that because $\Dalpha$ has full rank, we can conclude from Theorem~3.1 in \cite{Ikramov00} that the matrix $\begin{pmatrix} -\Dalpha & 1 \\ 1^T & 0 \end{pmatrix}$ has $n$ positive eigenvalues and one negative. Hence, 
$$ \det \begin{pmatrix} - \Dalpha & 1 \\ 1^T & 0 \end{pmatrix} = \det(-\Dalpha)(1^T \Dalpha^{-1} 1) \neq 0 $$
and $1^T\Dalpha^{-1}1 > 0$. In particular, this allows us to use the block matrix inverse to show that 
\begin{equation*}
   \begin{pmatrix} -\Dalpha & 1 \\ 1^T & 0 \end{pmatrix}^{-1}  = \begin{pmatrix}  -\Dalpha^{-1}+ \frac{\Dalpha^{-1} 1 1^T \Dalpha^{-1}}{1^T \Dalpha^{-1} 1}  & \Dalpha^{-1} 1 \frac{1}{1^T \Dalpha^{-1} 1} \\ (\Dalpha^{-1}1 \frac{1}{1^T \Dalpha^{-1} 1})^T &  \frac{1}{1^T \Dalpha^{-1} 1} \end{pmatrix}
\end{equation*}
and therefore, 
\begin{equation*}
    \flatlimit(\Zrv) = \y^T \underbrace{\left[ -\Dalpha^{-1}+ \frac{\Dalpha^{-1} 1 1^T \Dalpha^{-1}}{1^T \Dalpha^{-1} 1}  \right]}_{=: A} \dvec +  \frac{\y^T \Dalpha^{-1} 1}{1^T \Dalpha^{-1} 1}. 
\end{equation*}
Next, using the Binomial series expansion, we can see that for any $q\in \NN$,
\begin{equation*}
    \begin{split}
    d_{\alpha}(\z_i,\Zrv) =\sum_{j=0}^{q} {j \choose \alpha/2} 2^{\alpha/2 }\left(\frac{1}{2} \| \z_i - \Zrv\|_2^2 - 1\right)^j + O\left(\left( \frac{1}{2}\| \z_i - \Zrv\|_2^2 - 1\right)^{q+1}\right).
        \end{split}
\end{equation*}
Furthermore, by Lemma \ref{lm:innerpr} and the fact that $\| \z_i - \Zrv\|_2^2 = \z_i^\top \z_i + \Zrv^\top \Zrv - 2 \z_i^\top \Zrv$, we can see that for $q = \floor{2/\beta}$, $nO\left(\left( \frac{1}{2}\| \z_i - \Zrv\|_2^2 - 1\right)^{q+1}\right) \to 0$. Hence, assuming that the absolute eigenvalues $|\lambda_i(A)|$ of $A$ are all upper bounded by a non-zero positive constant, we can use exactly the same argument as used in the proof of Theorem \ref{prop:main_thm_ext} to conclude the proof.

Thus, we only need to show that the eigenvalues $\vert \lambda_i(A) \vert$ are upper bounded. We already know from Lemma \ref{lm:alphaexpbound} that there exists some constant $c>0$ independent of $n$, such that $\opnorm{ \Dalpha^{-1}} \leq c$. Thus, we only need to show that $\opnorm{ \frac{\Dalpha^{-1} 1 1^T \Dalpha^{-1}}{1^T \Dalpha^{-1} 1}}$ is almost surely upper bounded. Because $\frac{\Dalpha^{-1} 1 1^T \Dalpha^{-1}}{1^T \Dalpha^{-1} 1}$ is a rank one matrix, we know that
\begin{equation*}
    \opnorm{ \frac{\Dalpha^{-1} 1 1^T \Dalpha^{-1}}{1^T \Dalpha^{-1} 1}} = \frac{ 1^T \Dalpha^{-2} 1}{1^T \Dalpha^{-1} 1}.
\end{equation*}
First, we show that 
\begin{equation*}
    1^T \Dalpha^{-2} 1 = O(1) \quad 
\end{equation*}  
Let $\lambda_2,...,\lambda_{n}$ be the $n-1$ negative eigenvalues of $\Dalpha^{-1}$ and $\lambda_1$ the only positive eigenvalue. Furthermore, let $v_i$ be the corresponding orthonormal eigenvectors. Next, let $\alpha_i \in\mathbb{R}$ be such that $1 = \sum_{i = 1}^n \alpha_i v_i$. Since $\| 1\|_2 = \sqrt{n}$, we know that $\alpha_i \leq \sqrt{n}$. Then, 
\begin{align*}
    1^T \Dalpha^{-1} 1 &= \alpha_1^2 \frac{1}{\lambda_1} - \sum_{i = 2}^{n} \alpha_i^2 \frac{1}{|\lambda_i|} >0 \\
   \textrm{and}~~ 1^T \Dalpha^{-2} 1 &=  \alpha_1^2 \frac{1}{\lambda_1^2} + \sum_{i = 2}^{n} \alpha_i^2 \frac{1}{\lambda_i^2}.
\end{align*}
where we use that $1^T \Dalpha^{-1} 1>0$ which we already know from the discussion above. Because by the binomial expansion, $d_{\alpha}(\z_i,\z_j) = 2^{\alpha/2} +  O(\frac{1}{2}\|\z_i -\z_j \|_2^2-1)$, Lemma \ref{lm:innerpr} implies that $\underset{i\neq j}{\max} d_{\alpha}(\z_i,\z_j) \to 2^{\alpha/2}$ and hence, $\frac{1}{n} 1^T \Dalpha 1 \gtrsim n $, for any $n$ sufficiently large. Therefore, $\lambda_{1} \geq n$. Hence, there exists some constant $c>0$ independent of $n$, such that $\alpha_1^2 \frac{1}{\lambda_1} \leq c $. As a consequence, 
\begin{equation*}
    c \geq \alpha_1^2 \frac{1}{\lambda_1} \geq \sum_{i = 1}^{n-1} \alpha_i^2 \frac{1}{|\lambda_i|} \geq \tilde{c}\sum_{i = 1}^{n-1} \alpha_i^2 \frac{1}{\lambda_i^2},
\end{equation*} 
where we use in the last inequality that $\vert \lambda_i\vert \geq \tilde{c}$. Furthermore,  $ \alpha_1^2 \frac{1}{\lambda_1^2} = O(\frac{1}{n})$. Thus, we conclude that $ 1^T \Dalpha^{-2} 1 = O(1)$. 

In order to prove the result, we are only left to study the case where $1^T \Dalpha^{-2} 1 \geq 1^T \Dalpha^{-1} 1 \to 0$.
We prove by contradiction and assume that $\frac{ 1^T\Dalpha^{-2} 1}{1^T \Dalpha^{-1} 1}  \to  \infty$ and $1^T \Dalpha^{-1} 1 \to 0$. 
Let $\gamma = 1^T \Dalpha^{-1} 1$ and $v = \Dalpha^{-1} 1$. Furthermore, let $\tilde{v} = \begin{pmatrix} -\gamma \\ 0 \\ \vdots \\0  \end{pmatrix}$. We know that $1^T(v +\tilde{v}) = 0$ and hence,
\begin{equation*}
    (v + \tilde{v})^T \Dalpha  (v + \tilde{v}) \leq - \tilde{c} (v + \tilde{v})^T (v + \tilde{v}) ~~ 
\end{equation*} 
Next, note that our assumption imply that for $n$ sufficiently large,  $v^\top v -\gamma^2 \geq 1/2 v^\top v ~~$. 
Therefore,
\begin{align*}
      - \frac{\tilde{c}}{2} v^Tv  &\geq (v+\tilde{v})^T \Dalpha (v+\tilde{v})  = v^T \Dalpha v + 2 \tilde{v}^T \Dalpha v + \tilde{v}^T \Dalpha \tilde{v} =  1^T \Dalpha^{-1} 1 + 2 \tilde{v}^T 1 + \gamma^2 d_{\alpha}(X_1,X_1) =  \gamma - 2 \gamma, 
\end{align*} 
and  using the fact that $\gamma = 1^T \Dalpha^{-1} 1$ is positive, we get that $1 \geq \frac{\tilde{c}}{2} \frac{v^Tv}{\gamma}~~ $. However, this contradicts the assumption that $\frac{v^Tv}{\gamma} \to \infty$ and hence we can conclude the proof.\qed

%

\subsection{Additional results}
\label{sec:simplertaucase}
In this section, we present some additional results for different choices of the scaling. The results presented in this section are straight forward but provide a more complete picture for different choices of the scaling $\bw$. We use again the same notation as used in \suppmat~\ref{sec:prop1_proof}.

First, we show the case where $\bw \to 0$. We assume that $k$ is the $\alpha$-exponential kernel with $\alpha\in (0,2]$, i.e. $\kernf(\xk, \xkd) = \exp(-\| \xk-\xkd\|_2^{\alpha})$. 
\begin{lemma} 
\label{lm:simplertaucase}
Let $\prob_X$ satisfy Assumption \Bdone-\Bdtwo~ and assume that the bandwidth $\bw/\effdim =  O(n^{-\theta})$ with $\theta >0$. Furthermore, assume that the ground truth function $f^\star$ is bounded. 
Then, conditioned on the event $\EXset$, for any $\lambda \geq 0$,  with probability $\geq 1- (n+1)^2 \exp(-C  (\log(n))^{1+\epsilon})$ over the draws of $\Xrv \sim \prob_X$,
\begin{equation*}
    \EEobs \fhatridge (\Xrv) \to 0 ~.
\end{equation*} 
\end{lemma}
\begin{proof}

Let $\tilde{\tau} = \tau/\effdim$ and define $\z_i = \x_i/\sqrt{\effdim}$ and $\Zrv = \Xrv/\sqrt{\effdim}$. 
Lemma \ref{lm:innerpr} shows that $\| \z_i - \z_j\|_2^2$ concentrates around $2 (1-\delta_{i,j})$. Hence, $k_{\bw}(\x_i,\x_j) = \exp(-\tilde{\bw}^{-\alpha/2} \| \z_i - \z_j\|_2^{\alpha} ) \to \delta_{i,j}~~$ because $\tilde{\tau} \to 0$. In fact, due to the assumption that $\tilde{\tau} = O(n^{-\theta})$ with $\theta >0$, we can see that 
\begin{equation*}
    \opnorm{  K - \textrm{I}_n} \leq n ~ \max_{i \neq j} | \exp(-n^{\theta \alpha/2}\| \z_i - \z_j\|_2^{\alpha} )| \to 0 ~~,
\end{equation*}
and with probability $\geq 1- (n+1)^2 \exp(-C  (\log(n))^{1+\epsilon})$ over the draws of $\Xrv$,
\begin{equation*}
    \|\kvec \|_1 \leq n ~ \max_{i} | \exp(-n^{\theta \alpha/2}\| \z_i - \Zrv\|_2^{\alpha} )| \to 0.
\end{equation*}
 Hence, the result follows immediately from $\fhatridge(X) = y^T (\kermat +\lambda \idmat)^{-1} k_{\bw}(\Xs,\Xrv)$.
\end{proof}

We can also show a similar result for the case where $\bw \to \infty$ and $\lambda$ does not vanish. 

\begin{lemma} 
\label{lm:simplertaucase2}
Let $\prob_X$ satisfy Assumption \Bdone-\Bdtwo~ and assume that the bandwidth $\bw/ \effdim =  O(n^{\theta})$ with $\theta >\frac{2}{\alpha}$. Furthermore, assume that $\lambda = \Omega(1)$ and that the ground truth function $\ftrue$ is bounded. 
Then, conditioned on the event $\EXset$, for any $\lambda \geq 0$,  with probability $\geq 1- (n+1)^2 \exp(-C  (\log(n))^{1+\epsilon})$ over the draws of $\Xrv \sim \prob_X$,
\begin{equation*}
    \EEobs \fhatridge (\Xrv) \to c,
\end{equation*} 
with $c = f(\Xs)^T(11^T + \lambda \idmat_n)^{-1} 1$
\end{lemma}


\begin{proof}
 We use the same notation as in Lemma \ref{lm:simplertaucase2}.
Again due to Lemma \ref{lm:innerpr}, we find that
\begin{equation*}
\label{eq:case3eqmat}
    \opnorm{  K - 11^T } \lesssim n ~ \max_{i \neq j} | \exp(-n^{-\theta\alpha/2}\| \z_i - \z_j\| ) -1| \to 0 ~~.
\end{equation*}
As a result, we observe that the kernel matrix $K$ converges asymptotically to the rank one matrix $11^T$. We remark that such a phenomenon can also be observed in Theorem~2.1 and 2.2 in \cite{ElKaroui10} when $\trace(\Sigma_d)/d \to 0$. 
As a consequence, we observe that the eigenvalues of $K^{-1}$ diverge as $n \to \infty$. However, because by assumption, $\lambda = \Omega(1)$, we can still conclude that the eigenvalues of $(K +\lambda \textrm{I}_n)^{-1}$ are upper bounded by a constant independent of $n$.  Hence, with probability $\geq 1- (n+1)^2 \exp(-C  (\log(n))^{1+\epsilon})$ over the draws of $\Xrv \sim \prob_X$,
\begin{equation*}
\begin{split}
    \| (K+\lambda \textrm{I}_n)^{-1} \kvec - (11^T+\lambda \textrm{I}_n)^{-1} 1\| _1 &\lesssim n ~ \max_{i} | \exp(-n^{-\theta \alpha /2} c) -1| \\
        &\lesssim  n^{-\theta \alpha /2 +1} + O(n^{-\theta \alpha +1})\to 0,
    \end{split}
\end{equation*}
where we have used that $\theta \alpha/2 > 1$. 

The only thing left to show is that $f(\Xs)^T(11^T + \lambda \idmat_n)^{-1} 1$ does not diverge. For this, let $a \In + b 11^T$ be the inverse of $11^T+\lambda \textrm{I}_n$. As a result of a simple computation we find that $a = \frac{\lambda ( \lambda + 1 +(n-1)) + (n-1)}{\lambda ( \lambda + 1 +(n-1))}$ and $b = \frac{-1}{\lambda ( \lambda + 1 +(n-1))} $. Hence,
\begin{equation*}
\begin{split}
     \|(11^T+\lambda \textrm{I}_n)^{-1} 1 \|_1 = n\vert a + (n-1) b\vert &= \sqrt{n} \vert \lambda b\vert = \frac{n}{ ( \lambda + 1 +(n-1))} \to 1,
\end{split}
\end{equation*}
which completes the proof. 

\end{proof}

%% file: sections/Appendix/Appendix_Technical_Preliminaries.tex
\section{Technical lemmas}
\subsection{Proof of Lemma \ref{lm:innerpr}}

\label{sec:technical_prel}
We begin with the following Lemma, which is a direct consequence of the results in Appendix~A in \cite{ElKaroui10}.
\begin{lemma}[Concentration of quadratic forms]
\label{lm:concentration}
Suppose the vector $\x \in \mathbb{R}^d$, of some dimension $d \in \mathbb{N}_+$, is a random vector with either 
\begin{enumerate}
 \item i.i.d entries $\xel{i}$ almost surely bounded $|\xel{i}| \leq c$ by some constant $c>0$ and with zero mean and unit variance.
 \item standard normal distributed i.i.d entries.
\end{enumerate}
 Let $M$ be any symmetric matrix with $\opnorm{M} = 1$ and let $M = M_+ - M_{-}$ be the decomposition of $M$ into two positive semi-definite matrices $M_+$ and $M_-$ with $\opnorm{M_+}, \opnorm{M_-} \leq 1$. Then, there exists some positive constants $C_1,C_2,C_3$ independent of $M$ such that for any $r > \zeta = C_1/\trace(M_+)$,
\begin{equation*}
    \begin{split}
        P(|\x^T M x/\trace(M_+) - \trace(M)/\trace(M_+)| > r)  \lesssim  \exp(-C_2\trace(M_+)(r/2 - \zeta)^2) +  \exp(-C_3\trace(M_+)).
    \end{split}
\end{equation*}
\end{lemma}

%

\paragraph{Case 1: Distribution $\prob_X \in \probsetcov$}
Following the same argument as the one used in Corollary A.2 in \cite{ElKaroui10},
we can use Lemma \ref{lm:concentration} to show that there exists constants $C_2,C_3 >0 $ such that for $n \to\infty$,
\begin{equation*}
 P(|\x_i^T \x_j/\trace(\inputcov) - \delta_{i,j}| > r)  \leq C_3[\exp(-C_2\trace(\inputcov)(r/2 - \zeta)^2) +  \exp(-C_2\trace(\inputcov)].
\end{equation*}
We now make us of the Borel-Cantelli Lemma. For any $\epsilon > 0$, let $r(n) = \frac{1}{\sqrt{2}} n^{-\beta/2} \ (\log(n))^{(1+\epsilon)/2} $ and note that because $ \trace(\inputcov) \asymp n^{\beta}$, $\zeta$ decays at rate $n^{-\beta}$ and in particular, for any $n$ sufficiently large, $r(n)/2 > \zeta \to 0$. Hence, we can see that there exists some constant $C>0$ such that for any $n$ sufficiently large
\begin{equation*}
\label{eq:px}
    \begin{split}
        P(|\x_i^T \x_j/\trace(\inputcov) - \delta_{i,j}| > r(n)) \leq \exp(-C \ (\log(n))^{1+\epsilon}).
    \end{split}
\end{equation*}
Next, using the union bound, we get
\begin{equation*}
    \begin{split}
        P(\underset{i,j}{\max} |\x_i^T \x_j/\trace(\inputcov) - \delta_{i,j}| > r(n)) \leq n^2 \exp(-C \ (\log(n))^{1+\epsilon}).
    \end{split}
\end{equation*}
And because $\epsilon>0$, for any $N \in \mathbb{N}_+$, we have that \begin{equation*}
\sum_{n=N}^{\infty}  n^2 \exp(-C \ (\log(n))^{1+\epsilon}) < \infty.
\end{equation*}
which allows us to apply the Borel-Cantelli Lemma. Hence,
\begin{equation}
\underset{i,j}{\max} \left|\x_i^T \x_j /\trace(\inputcov) - \delta_{i,j}\right| \leq
\frac{1}{\sqrt{2}} n^{-\beta/2}  (\log(n))^{(1+\epsilon)/2} ~~ a.s. ~ \text{as}~ n \to \infty,
\label{eq:concentrationeq}
\end{equation}
which concludes the first step of the proof.
 Next, we already know from the previous discussion that for any $n$ sufficiently large, \\$P(\mathcal{E}_{\textbf{X}}) \geq 1-  n^2 \exp(-C \ (\log(n))^{(1+\epsilon)/2})$. Furthermore, 
 because $X$ is independently drawn from the same distribution as $\xind{i}$, 
 $P(\EXset\cup \Uset) \geq 1 -  (n+1)^2 \exp(-C \ (\log(n))^{(1+\epsilon)/2})$. Hence, for any $n$ sufficiently large,
 
 \begin{equation*}
 \begin{split}
      P(\mathcal{E}_{\textbf{X}, \Xrv} | \mathcal{E}_{\textbf{X}})  = \frac{P(\EXset \cup \Uset)}{P(\mathcal{E}_{\textbf{X}})}  &\geq   \left[1 -  (n+1)^2 \exp(-C \ (\log(n))^{(1+\epsilon)}) \right]
      \end{split}
 \end{equation*}
This completes the first case of the proof, i.e. where $\prob_X \in \probsetcov$.

%
%

\paragraph{Case 2: Distribution $\prob_X \in \probsetsphere$}
First, note that the case where $i = j$ is clear. Let $s_i = \frac{\x_i}{\|\x_i\|_2}$ and $\z_i = \frac{\x_i}{\sqrt{\trace(\inputcov)}}$. Since we are in the Euclidean space, the inner product is given by
\begin{equation*} 
(s_i^Ts_j)^2 = \frac{(\z_i^T\z_j)^2}{\|\z_i\|_2^2 \|\z_j\|_2^2 }.
\end{equation*}
Due to Equation \eqref{eq:concentrationeq}, we have that $|~ \|\z_i\|_2^2 -1 | \leq n^{-\beta/2}  (\log(n))^{(1+\epsilon)/2}  ~ a.s. ~ \text{as}~ n \to \infty$ and \\
$(\z_i^T\z_j)^2 \leq ( n^{-\beta/2}  (\log(n))^{(1+\epsilon)/2})^2a.s. ~\text{as}~ n \to \infty$. Therefore,
\begin{equation*} 
(s_i^Ts_j)^2 \leq  ( n^{-\beta/2}  (\log(n))^{(1+\epsilon)/2})^2~~ a.s. ~ \text{as}~ n \to \infty.
\end{equation*}
The rest of the proof then follows straight forwardly. 
\qed

\begin{proof}[Proof of Lemma \ref{lm:concentration}]
Let 
\begin{equation*}
\label{eq:concentration_at _infty}
f: \mathbb{R}^d \to \mathbb{R}, \x \to \sqrt{\x^T M_+ \x / \trace(M_+)} =\frac{1}{\sqrt{\trace(M_+)}}\|M_+^{1/2} \x\|_2.
\end{equation*}
We know that $f$ is $\lambda_{\max}(M_+)/\sqrt{\trace(M_+)}$-Lipschitz continuous and hence also a $\sqrt{\opnorm{M}}/\sqrt{\trace(M_+)}$-Lipschitz continuous. For the case where the entries $x$ are bounded i.i.d. random variables, we can use the simple fact that the norm is convex in order to apply Corollary~4.10 in \cite{Ledoux2001} and Proposition 1.8 in in \cite{Ledoux2001}. For the case where the entries are normally distributed we can apply Theorem~\rom{5}.\rom{1} in \cite{Milman86}. As a result, we can see that there exists a constant $C_4>0$ independent of $M$, such that
\begin{equation*}
    \label{eq:concentration}
    P \left( \left|\sqrt{x^T M x/\trace(M_+)} - \sqrt{\trace(M)/\trace(M_+)} \right| > r \right) \leq 4 \exp(4 \pi) \exp(-C_4 \trace(M_+) r^2).
\end{equation*}
The proof then follows straight forwardly following line by line the proof of Lemma~A.2 in \cite{ElKaroui10}. 

\end{proof}

\subsection{Proof of Lemma \ref{lm:lmmainthm}}
\label{sec:prooflmmain}
For any $ j\leq \mc+1,$ $ \alpha = (i_1,i_2)$, let $g_j^{(\mathbf{\alpha})}$ denote the partial derivatives $g_j^{(\mathbf{\alpha})}(x,y) = \frac{\partial^{|\alpha|}}{\partial t_1^{i_1} t_2^{i_2}} g_j(t_1,t_2)|_{x,y}$.
Define $\mct = \floor{2/\beta}$
First of all, note that due to Lemma \ref{lm:innerpr}, for any $\delta, \delta'>0$ and $n$ sufficiently large, for any $\Zrv \in \UsetZ$ 
\begin{equation*}
\begin{split}
 \textrm{for all}~ i \neq j: ~~(\zind{i},\zind{j}) &\in \convergesetres{\delta}{\delta'},\\
 \textrm{for all}~ i:  ~~~(\zind{i}, \Zrv) &\in \convergesetres{\delta}{\delta'}.
 \end{split}
\end{equation*}
As a result, we can make use of Assumption \Ckone. We are heavily going to make use of this fact throughout the proof.
The proof is separated into two steps where we first show 1. and then 2. using the expression for $\polf$ from the first step. 
\\\\
\textbf{Proof of the first statement}  
We construct a polynomial $\polf(\Zrv)$ using the power series expansion of $k$ from Assumption \Akone~ and in addition the Taylor series approximation of $g_\lidc$ around the point $(\cnbw,\cnbw)$. For any $n$ sufficiently large, 
we can write 
 

\begin{equation}
\label{eq:talyorapprox}
    \begin{split}
    \textrm{for all}~ i ~~ \kernf(\z_i, \Zrv) &=  \sum_{\lidc=0}^{\mct} (\z_i^\top\Zrv)^\lidc g_\lidc(\|\z_i\|_2^2,\|\z_j\|_2^2) 
    +(\zind{i}^\top \Zrv)^{\mct+1}  \remainder (\|\zind{i}\|_2^2, \|\Zrv\|_2^2, \zind{i}^\top \Zrv) \\
       &= \underbrace{\sum_{\lidc =0}^{\mct} (\z_i^\top\Zrv)^{\lidc} \sum_{ l_1+l_2 \leq \mct-\lidc} \frac{ g_{\lidc}^{(l_1,l_2)}(\cnbw,\cnbw)}{l_1
    ! l_2!} (\z_i^\top\z_i -\cnbw )^{l_1} (\Zrv^\top\Zrv -\cnbw)^{l_2}}_{=: 
    (\mvec_Z)_i} \\ 
    &+ \sum_{\lidc =0}^{\mct}(\z_i^\top\Zrv)^{\lidc} \sum_{ l_1+l_2  = \mct+1-\lidc} \frac{ g_{\lidc}^{(l_1,l_2)}(\eta^{\lidc,i}_{l_1,l_2})}{l_1
    ! l_2!} (\z_i^\top\z_i -\cnbw )^{l_1} (\Zrv^\top\Zrv -\cnbw )^{l_2}   \\ 
    &+(\zind{i}^\top \Zrv)^{\mct+1}  \remainder (\|\zind{i}\|_2^2, \|\Zrv\|_2^2, \zind{i}^\top \Zrv).
    \end{split}
\end{equation}
where $\eta^{\lidc,i}_{l_1,l_2} \in B_r(\cnbw,\cnbw)$ are points contained in the closed ball around the point $(\cnbw,\cnbw)$ with radius \\$r^2 = (\|\z_i\|_2^2-\cnbw)^2 + (\|\Zrv\|_2^2-\cnbw)^2 \to 0$. Hence, using the fact that $g_i$ is $\mct+1-i$-times continuously differentiable,  we can see that any $\|g_{\lidc}^{(l_1,l_2)}(\eta^{\lidc,i}_{l_1,l_2})\|$ is almost surely upper bounded by some constant. 

Let $\mvec_Z$ be the vector defined in Equation~\eqref{eq:talyorapprox} We define the polynomial $\polf$ as
\begin{equation*}
\polyf(\Zrv) := \ftrue(\X)^\top (K + \lambda I_n)^{-1} \mvec_Z
\end{equation*}
Note that $\polf$ is a linear combination of the terms $(\Zrv^\top\Zrv)^{p_1}(\z_i^\top\Zrv)^{p_2}$ with $p_1 + p_2 \leq \mct$, and hence a polynomial of $\Zrv$ of degree at most $2\mct$. If $g_\lidc$ are constant, i.e. the kernel is an inner product kernel, $\mvec_Z$ contains only the terms  $(\z_i^\top\Zrv)^{l}$ and hence $\polyf(\Zrv)$ is a polynomial of $\Zrv$ of degree at most $\mct$. 
 
 Next, because by assumption $\vert \ftrue \vert \leq \ubfunc$ is bounded on the support of $\prob_X$ by some constant $\ubfunc$, all entries of $\ftrue(\X)$ are bounded, and hence,
\begin{equation}
\label{eq:prop1_eq1}
\begin{split}
  \left| \EEobs \fhatridge(\sqrt{\bw} \Zrv) - \polf(\Zrv) \right| &= \left|\ftrue(\X)^\top (K + \lambda I_n)^{-1} (\kvec - \mvec_Z) \right| \leq  \ubfunc \| (K + \lambda I_n)^{-1} (\kvec -\mvec_Z) \|_1,
\end{split}
\end{equation}
where we have used that $\EEobs \fhatridge(\sqrt{\bw} \Zrv) = \ftrue(\X)^\top (K + \lambda I_n)^{-1} \kvec$. Further, by assumption $\lambda_{\min}(K + \lambda I_n)\geq \minEV >0  $, and hence,
\begin{equation*}
   \|(K + \lambda I_n)^{-1} (\kvec - \mvec_Z) \|_1 \leq \sqrt{n}\|(K + \lambda I_n)^{-1} (\kvec - \mvec_Z) \|_2 \leq \frac{\sqrt{n}}{\minEV}  \| \kvec - \mvec_Z\|_2  \leq \frac{n}{\minEV} \underset{i}{\max} \vert (\mvec_Z)_i- (\kvec)_i \vert.
\end{equation*}
Equation \eqref{eq:talyorapprox} yields
\begin{equation*}
    \begin{split}
    n \max_i\vert (\mvec_Z)_i - (\kvec)_i \vert &\leq  n \max_{i} \underbrace{ \left|\sum_{q =0}^{\mct} (\z_i^\top\Zrv)^{q} \sum_{ l_1+l_2  = \mct+1-q} \frac{ g_{q}^{(l_1,l_2)}(\eta^{q,i}_{l_1,l_2})}{l_1
    ! l_2!} (\z_i^\top\z_i -1 )^{l_1} (\Zrv^\top\Zrv -1)^{l_2}  \right| }_{=: B_1^{i}} \\ &+ n \max_{i\neq j}\underbrace{\left|(\zind{i}^\top \Zrv)^{\mct+1}  \remainder (\|\zind{i}\|_2^2, \|\Zrv\|_2^2, \zind{i}^\top \Zrv)   \right| }_{=: B_2^{i}}.\\ 
    \end{split}
\end{equation*}
In order to conclude the first step of the proof, we only need to show that both terms go to zero. 
First, we show that the term $n \max\limits_{i} B^{i}_1 \to 0$. Recall that $\left| g_{q}^{(l_1,l_2)}(\eta^{q,i}_{l_1,l_2}) \right|$ is upper bounded as $n\to \infty$ independent of $i$. Hence, we can apply Lemma \ref{lm:innerpr} which shows that for any integers $q,l_1$ and $l_2$ such that $q+l_1+l_2 = \mct+1$,
\begin{equation*}
    \left| (\z_i^\top\z_i -1 )^{l_1} (\Zrv^\top\Zrv -1 )^{l_2} (\z_i^\top\Zrv)^{q}\right| \lesssim (n^{-\beta/2}(\log(n))^{(1+\epsilon)/2})^{\mct+1} ~ 
\end{equation*}
which holds true for any positive constant $\epsilon >0$.
Finally, because $\mct = \floor{2/\beta}$,  $(\beta/2)(\mct+1) > 1$, and hence 
\begin{equation*}
    n \max_{i} B^{i}_1 \lesssim n (n^{-\beta/2} (\log(n))^{(1+\epsilon)/2})^{\mct+1} \to 0 .
\end{equation*} 
Furthermore, because $ \remainder $ is a continuous function and $\z_i,\Zrv$ are contained in a closed neighborhood around of $(1,1,0)$, $r(\|z_i\|,\|\Zrv\|,\z_i^\top\Zrv)$ is upper bounded by some constant independent of $i$ as $n \to \infty$. Therefore, we also have that
\begin{equation*}
n \max_{i} B^{i}_2 \lesssim n \max_{i} \left| (\z_i^\top\Zrv)^{q+1} \right| \lesssim  n \left(n^{-\beta/2}(\log(n))^{(1+\epsilon)/2}\right)^{\mct+1} ~  \to 0,
\end{equation*}
where we have again used Lemma \ref{lm:innerpr}. Hence, we can conclude the first step of the proof when observing that we have only assumed that $\Zrv \in \UsetZ$ and hence the convergence is uniformly. 
\\\\
\textbf{Proof of the second statement} 
We can see from the definition of $\polf$ and the subsequent discussion that
\begin{equation*}
    \begin{split}
        \left\| \polf \indicator_{\Zrv \in \UsetZ^{\comp}} \right\|_{\Ell_2(\prob_Z)} 
        &= \left\| \ftrue(\X)^\top (K + \lambda I_n)^{-1} \mvec_Z \indicator_{ \Zrv \in \UsetZ^{\comp}} \right\|_{\Ell_2(\prob_Z)} \lesssim n~\underset{i}{\max} \left\| (\mvec_Z)_i \indicator_{ \Zrv \in \UsetZ^{\comp}} \right\|_{\Ell_2(\prob_Z)} ,
    \end{split}
\end{equation*}
and furthermore, 
\begin{equation*}
    \max_{i} \left\| (\mvec_Z)_i \indicator_{\Zrv \in \UsetZ^{\comp}} \right\|_{\Ell_2(\prob_Z)} \lesssim \sum_{q + l_{1} + l_{2} \leq \mct} \left\| (z_{i}^\top \Zrv)^{q} (z_{i}^\top z_{i} - 1)^{l_{1}} (\Zrv^\top \Zrv - 1)^{l_{2}} \indicator_{ \Zrv \in \UsetZ^{\comp}} \right\|_{\Ell_2(\prob_Z)}.
\end{equation*}
We can decompose for $n$ sufficiently large,
\begin{equation*}
    \begin{split}
        &~~\left\| (z_{i}^\top \Zrv)^{q} (\|\z_i\|_2^2 - 1)^{l_{1}} \left( \left\|\Zrv \right\|_2^{2} - 1 \right)^{l_{2}} \indicator_{ \Zrv \in \UsetZ^{\comp}} \right\|_{\Ell_2(\prob_Z)} \\
&\leq \left\| ((\|\z_i\|_2^2 - 1) + 1)^{\frac{q}{2}} \left\|\Zrv \right\|_2^{q} (\|\z_i\|_2^2 - 1)^{l_{1}} \left( \left\|\Zrv \right\|^{2} - 1 \right)^{l_{2}} \indicator_{\Zrv \in \UsetZ^{\comp}} \right\|_{\Ell_2(\prob_Z)}\\
        &\lesssim \left\| \left\|\Zrv \right\|_2^{q} \left( \left\|\Zrv \right\|_2^{2} - 1 \right)^{l_{2}} \indicator_{\Zrv \in \UsetZ^{\comp}} \indicator_{ \left\|\Zrv\right\|_2^2 \leq 2}
\right\|_{\Ell_2(\prob_Z)} + \left\| \left\|\Zrv \right\|_2^{q} \left( \left\|\Zrv \right\|_2^{2} - 1 \right)^{l_{2}} \indicator_{ \left\|\Zrv\right\|_2^2 > 2} \right\|_{\Ell_2(\prob_Z)},
    \end{split}
\end{equation*}
where we have used Lemma \ref{lm:innerpr} in the second inequality. 
The first term vanishes trivially from the concentration inequality. Indeed, for $n$ sufficiently large, we have that
\begin{equation*}
    \begin{split}
        \left\| \left\|\Zrv \right\|_2^{q} \left( \left\|\Zrv \right\|_2^{2} - 1 \right)^{l_{2}} \indicator_{ \Zrv \in \UsetZ^{\comp}} \indicator_{ \left\{ \left\|\Zrv \right\|_2^{2} \leq 2 \right\}} \right\|_{\Ell_2(\prob_Z)}   
        \lesssim P(\UsetZ^{\comp}| \EZset)  \lesssim (n+1)^2 \exp(-C(\log (n))^{1 + \epsilon}) .
    \end{split}
\end{equation*}
For the second term, note that we can see from the proof of Lemma~\ref{lm:concentration} that there exists some constant $c > 0$, such that for $n$ sufficiently large,
\begin{equation*}
    P\left( \left\|\Zrv \right\|_2^{2} > r \right) \leq \exp(-c n^{\beta} r).
\end{equation*}
We can now apply integration by parts to show that 
\begin{equation*}
    \begin{split}
            &~~~~\left\| \left\|\Zrv \right\|_2^{q} \left( \left\|\Zrv \right\|_2^{2} - 1 \right)^{l_{2}} \indicator_{ \left\{ \left\|\Zrv \right\|_2^{2} > 2 \right\}} \right\|_{\Ell_2(\prob_Z)}^{2} \leq \int \left\|\Zrv \right\|_2^{q + 2l_{2}} \indicator_{ \left\{ \left\|\Zrv \right\|_2^{2} > 2 \right\}} dP\\
            & \lesssim \left[r^{q+2l_2}P(\|\Zrv\|_2^2 > r)\right]_{2}^{\infty} - \int_{2}^{\infty} r^{q + 2l_{2}-1} P\left( \left\|\Zrv \right\|_2^{2} > r \right) dr \lesssim \exp(-2cn^{\beta})\\
    \end{split}
\end{equation*}
Hence, combining these terms, we get the desired result
\begin{equation}
 n~\underset{i}{\max} \left\| (\mvec_Z)_i \indicator_{ \Zrv \in \UsetZ^{\comp}} \right\|_{\Ell_2(\prob_Z)} \lesssim n (n+1)^2  \exp(-C(\log (n))^{1 + \epsilon})  \to 0 ~~\textrm{as}~ n\to\infty.
\end{equation}

\qed

\subsection{Proof of Lemma \ref{lm:elkarouiextension}}
\label{sec:proofev}
As in \cite{ElKaroui10}, we separately analyze the off and on diagonal terms of $K$. Let $A$ be the off-diagonal matrix of $K$, with diagonal entries $A_{i,j} = (1- \delta_{i,j}) K_{i,j}$ and let $D$ be the diagonal matrix of $K$ with entries $D_{i,j} = \delta_{i,j} K_{i,j}$. We have
\begin{align*}
    \textrm{for all}~ i: ~~~&D_{i,i} := K_{i,i} = g(\|\z_i\|_2^2, \|\z_i\|_2^2,\z_i^\top\z_i), \\
    \textrm{for all}~ i\neq j:~~~ &A_{i,j} := K_{i,j} = g(\|\z_i\|_2^2, \|\z_j\|_2^2,\z_i^\top\z_j).
\end{align*}
Similarily, decompose $M$ into its off-diagonal, $M_A$, and its diagonal $M_D$. We have
\begin{equation}
\opnorm{K - M} \leq \opnorm{A - M_A} + \opnorm{D - M_D}
\end{equation}
We begin with the first term. Note that $ M_A$ has off-diangoal entries $(M_A)_{i,j} := \sum_{q =0}^\mc (\z_i^\top\z_j)^{q} g_q(\|z_i\|_2^2,\| z_j\|_2^2)$, and hence, 

\begin{equation*}
    \opnorm{ M_A - A} \leq \| M_A - A\|_F \leq n \underset{i,j}{\max} |(M_{A})_{i,j} - A_{i,j} | \lesssim n\underset{i,j}{\max}| (\z_j^\top\z_j -1 )^{m+1} | \to 0,
\end{equation*}
where we have the same argument as used in the proof of Lemma \ref{lm:lmmainthm} and the fact that the Assumptions \Akone-\Akthree ~imply Assumption \Ckone, as shown in Lemma \ref{lm:AC_assumptions}.

Next, note that we can write
\begin{equation*}
    M_D := \left[g(\cbw, \cbw, \cbw) - \sum_{q=0}^m g_q(\cbw, \cbw)~ + \sum_{q=0}^m \|\z_i\|_2^{2q} g_q(\|\z_i\|_2^2,\| \z_i\|_2^2) \right] \mathrm{I}_n
\end{equation*} Because $D-M_D$ is a diagonal matrix, for $n$ sufficiently large,
\begin{equation*}
\begin{split}
    \opnorm{ D - M_D}&= \max_{i} \left|g(\| \z_i\|_2^2, \| \z_i\|_2^2,\z_i^\top\z_i) - 
    g(\cbw, \cbw, \cbw) + \sum_{q=0}^m g_q(\cbw, \cbw)~ -  \sum_{q=0}^m \|\z_i\|_2^{2q} g_q(\|\z_i\|_2^2,\| \z_i\|_2^2)\right|
        \\
    &\leq \underbrace{\delta_L \sqrt{3} (\| \z_i\|_2^2 -\cbw)}_{T_1} + \sum_{q=0}^m\underbrace{[\|\z_i\|_2^{2q} - \cnbw^q] g_q(\|\z_i\|_2^2,\| \z_i\|_2^2)}_{T_2}+ \underbrace{[ g_q(\|\z_i\|_2^2,\| \z_i\|_2^2) - g_q(\cbw, \cbw)]}_{T_3} 
\end{split}
\end{equation*}
where we have used that by assumption $g$ is $\delta_L$-Lipschitz continuous on the restriction \\ $\{(x,x,x) | x \in [1-\delta_L, 1+ \delta_L]\} \subset \Omega$ for some $\delta_L >0$. Clearly $T_1 \to 0$  due to Lemma \ref{lm:innerpr}. Furthermore, by Assumption \Ckone, for any $q \leq m$, $g_q$ is continuously differentiable  and hence also Lipschitz continuous in a closed ball around $(\cbw,\cbw)$. Thus, $T_3 \to 0$. Hence, it is only left to show that $T_2 \to 0$, which is a consequence of the following claim.
\\\\
\textbf{Claim:} \textit{For any $\epsilon>0$ and any $q>0$,
\begin{equation*}
 \underset{i}{\max}~ \left| \frac{(\x_i^\top \x_i)^q}{ (\trace(\Sigma_d))^q} - 1 \right| \leq  c_q\max[n^{-\beta/2} (\log(n))^{(1+\epsilon)/2}, n^{-q\beta/2}(\log(n))^{q((1+\epsilon)/2)}] 
 \end{equation*}
where $c_q$ is a constant only depending on $q$.}
\\\\
\textbf{Proof of the claim:}
In order to prove the claim, recall that due to Lemma \ref{lm:innerpr}, for every $q>0$
\begin{equation*}
 \underset{i}{\max}~ \left| \left[\frac{(\x_i^\top \x_i)}{\trace(\Sigma_d)} - 1\right]^q \right| \leq ( n^{-\beta/2} (\log(n))^{(1+\epsilon)/2})^q.
 \end{equation*}
We prove the claim by induction. The case where $q=1$ holds trivially with $c_1 = 1$. For $q>1$, 
\begin{equation*}
    \underset{i}{\max}~  \left| \left[\frac{(\x_i^\top \x_i)}{\trace(\Sigma_d)} - 1\right]^q \right| =  \underset{i}{\max}~  \left| \frac{(\x_i^\top \x_i)^q}{ (\trace(\Sigma_d))^q} + \sum_{j=1}^q (-1)^j {q \choose j }  \left(\frac{\x_i^\top \x_i}{\trace(\Sigma_d)}\right)^{q-j}\right| \leq  (n^{-\beta/2} (\log(n))^{(1+\epsilon)/2})^q.
\end{equation*} 
Next, by induction, $ \underset{i}{\max}~ \left| \frac{(\x_i^\top \x_i)^j}{ (\trace(\Sigma_d))^j} - 1 \right| \leq  c_1\max\left[n^{-\beta/2} (\log(n))^{(1+\epsilon)/2}, n^{-j\beta/2}(\log(n))^{j((1+\epsilon)/2)}\right]$ for any $j< q$. Furthermore, $\sum_{j=1}^q(-1)^j {q\choose j} = -1$, which shows that
\begin{equation*}
\begin{split}
 &\underset{i}{\max}~ \left| \frac{(\x_i^\top \x_i)^q}{ (\trace(\Sigma_d))^q} + \sum_{j=1}^q (-1)^j {q \choose j } \left( \frac{(\x_i^\top \x_i)^q}{ (\trace(\Sigma_d))^q}\right)^{q-j} \right| \\ =  &\underset{i}{\max}~ \left| \frac{(\x_i^\top \x_i)^q}{ (\trace(\Sigma_d))^q} + \sum_{j=1}^q (-1)^j {q \choose j } \right[ \left( \frac{(\x_i^\top \x_i)^q}{ (\trace(\Sigma_d))^q}\right)^{q-j} -1 \left]+ \sum_{j=1}^q (-1)^j {q \choose j }  \right|\\
 \geq &\underset{i}{\max}~ \left| \frac{(\x_i^\top \x_i)^q}{ (\trace(\Sigma_d))^q} -1 \right| - \sum_{j=1}^q {q \choose j }  c_j\max\left[n^{-\beta/2} (\log(n))^{(1+\epsilon)/2}, n^{-j\beta/2}(\log(n))^{j((1+\epsilon)/2)}\right].
 \end{split}
 \end{equation*}
 Hence, combining these two results, 
 \begin{equation*}
      \underset{i}{\max}~ \left| \frac{(\x_i^\top \x_i)^q}{ (\trace(\Sigma_d))^q} -1\right| \leq (n^{-\beta/2} (\log(n))^{(1+\epsilon)/2})^q + \sum_{j=1}^q {q \choose j }  c_j\max \left[n^{-\beta/2} (\log(n))^{(1+\epsilon)/2}, n^{-j\beta/2}(\log(n))^{j((1+\epsilon)/2)} \right],
 \end{equation*}
which completes the induction and thus the proof
\qed

\subsection{Proof of Lemma \ref{lm:alphaexpbound}}
\label{sec:proofalpha}

We use the well known formula 

\begin{equation*}
    t^{\alpha}=c_\alpha\int_0^\infty(1-e^{-t^2x^2})x^{-1-\alpha}\df x
\end{equation*}
with
\begin{equation*}
    c_\alpha:=\left(\int_0^\infty(1-e^{-x^2})x^{-1-\alpha}\df x\right)^{-1}>0,
\end{equation*}
which holds for all $t \geq 0$.
Hence, for $t:=\Vert z_i-z_j\Vert_2^\alpha\geqslant0$, we can write
$$\Vert z_i-z_j\Vert_2^\alpha=c_\alpha\int_0^\infty(1-e^{-x^2\Vert z_i-z_j\Vert_2^2})x^{-1-\alpha}\df x.$$
We first study $\mu \in\mathbb{R}^n$ with $\sum_{1\leqslant i\leqslant n}\mu_i=0$. We have 
 \begin{align*}
    \mu^\top \Dalpha \mu = \sum_{1\leqslant i,j\leqslant n}\Vert z_i-z_j\Vert_2^\alpha\mu_i\mu_j & = \sum_{1\leqslant i,j\leqslant n}\mu_i\mu_jc_\alpha\int_0^\infty(1-e^{-x^2\Vert z_i-z_j\Vert_2^2})x^{-1-\alpha}\df x\\
    & = -c_\alpha\int_0^\infty x^{-1-\alpha} \sum_{1\leqslant i,j\leqslant n}\mu_i\mu_j e^{-x^2\Vert z_i-z_j\Vert_2^2}\df x.
\end{align*}
Next, note that the Gaussian kernel satisfies Assumptions \Akone-\Akthree~ since\\
${\exp(-\|\xk-\xkd\|_2^2) = \sum_{j=0}^{\infty} \frac{2^j}{j!}
(\xk^\top \xkd)^j \exp(-\|\xk\|_2^2 - \|\xkd\|_2^2)}$. Hence, we can conclude from Proposition \ref{prop:EVlowerbound} that for every $x \in
\mathbb{R}_+$ there exists a constant $c_{G,x}$, such that $\sum_{1\leqslant i,j\leqslant n}\mu_i\mu_j
e^{-x^2\Vert z_i-z_j\Vert_2^2} \geq \| \mu \|_2^2 c_{G,x} > 0$ almost surely as $n\to \infty$. Thus, we can conclude that
there exists a constant $\tilde{c} >0 $ independent of $n$, such that
\begin{equation}
\label{eq:eqalphaexpbound}
    \mu^\top \Dalpha \mu \leq -\tilde{c}\| \mu \|_2^2 ~~a.s.~ \text{as}~ n\to\infty.
\end{equation}

Because the set of vectors $\mu^T1 = 0$ span a $n-1$ dimensional subspace, we can apply the Courant–Fischer–Weyl min-max principle from which we can see that the second largest eigenvalue of the matrix $\Dalpha$ satisfies
$\lambda_{2} < -\tilde{c}$ almost surely as $n\to\infty$. Since the sum of the eigenvalues
$\sum_{i=1}^{n} \lambda_{i} = \trace(\Dalpha) = 0$, $\lambda_{1} > (n-1) \tilde{c}$ almost surely as $n\to\infty$, which concludes the proof.
\qed

\subsection{Proof of Lemma \ref{lm:ntk1}}
\label{sec:proofntk}

We start the proof with a discussion of existing results in the literature.
As shown in Appendix~E.1 in \cite{Arora19}, the homogeneity of $\sigma$ allows us to write
\begin{equation}
\label{eq:ntksimp1}
     \Sigma^{(i)}(\x,\xkd) = c_{\sigma} \left(\Sigma^{(i-1)}(\x,\x) \Sigma^{(i-1)}(\xkd,\xkd)  \right)^{k/2}t_{\sigma}\left( \frac{\Sigma^{(i-1)}(\x,\xkd)}{\sqrt{\Sigma^{(i-1)}(\x,\x) \Sigma^{(i-1)}(\xkd,\xkd)}} \right),
\end{equation}
with 
\begin{equation*}
\label{eq:hatsimga}
   t_{\sigma}(\rho) = \underset{(u,v) \sim \mathcal{N}(0, \tilde{\Lambda}(\rho))}{\mathbb{E}}~\left[ \sigma(u) \sigma(v) \right] 
\end{equation*}
and 
$\tilde{\Lambda}(\rho) = \begin{pmatrix} 1& \rho \\ \rho & 1 \end{pmatrix} $. Furthermore, because $\dot{\sigma}$ is a $k-1$-homogeneous, we can analogously write
\begin{equation*}
\label{eq:ntkdotsigma}
    \dot{\Sigma}^{(i)} =c_{\dot{\sigma}} \left(\Sigma^{(i-1)}(\x,\x) \Sigma^{(i-1)}(\xkd,\xkd)  \right)^{(k-1)/2} t_{\dot{\sigma}}\left( \frac{\Sigma^{(i-1)}(\x,\xkd)}{\sqrt{\Sigma^{(i-1)}(\x,\x) \Sigma^{(i-1)}(\xkd,\xkd)}} \right),
\end{equation*}
with $t_{\dot{\sigma}}(\rho) = \underset{(u,v) \sim \mathcal{N}(0, \tilde{\Lambda}(\rho))}{\mathbb{E}}~\left[ \dot{\sigma}(u) \dot{\sigma}(v) \right] $. The function $t_{\sigma}$ is called the dual of the activation function (Definition~4 in \cite{Daniely16}). In particular, since by assumption $\sigma$ and $\dot{\sigma}$ have a Hermite polynomial extension, Lemma~11 in \cite{Daniely16} provides some useful properties which hold for both $t_{\sigma}$ and $t_{\dot{\sigma}}$. We only state them for $t_{\sigma}$:
\begin{enumerate}
    \item Let $a_i \in \mathbb{R}$ be the coefficients of the Hermite polynomial extension of $\sigma$, then $t_{\sigma}(\rho) = \sum_{i=0}^{\infty} a_i^2 \rho^i$
    \item The function $t_{\sigma}$ is continuous in $[-1,1]$ and smooth in $(-1,1)$
    \item The image of $t_{\sigma}$ is $[-\gamma, \gamma ]$ with $\gamma = \underset{v \sim \mathcal{N}(0, 1)}{\mathbb{E}}~\left[ \sigma(v)^2 \right] = 1/c_{\sigma}$
    \item We have that $t_{\sigma}(1) = \underset{v \sim \mathcal{N}(0, 1)}{\mathbb{E}}~\left[ \sigma(v)^2 \right] = c_{\sigma}^{-1}$
\end{enumerate}
Based on this discussion, we now prove the lemma.
 In a first step, we derive a closed form expression for $\Sigma^{(i)}(\x,\x)$. Based on the discussion above and particularly  Equation \eqref{eq:ntksimp1} we can see that
\begin{equation*}
    \Sigma^{(i)}(\x,\x) = \left(\Sigma^{(i-1)}(\x,\x)\right)^k, 
\end{equation*} 
and by induction, we get that
\begin{equation}
\label{eq:ntklipschitz}
    \Sigma^{(i)}(\x,\x) = (\x^\top \x)^{k^i}.
\end{equation}
Therefore, Equation \eqref{eq:ntksimp1} becomes 
\begin{equation*}
\label{eq:ntkproof2}
    \Sigma^{(i)}(\x,\xkd) = c_{\sigma} (\x^\top \x)^{\frac{k^{i}}{2}}(\xkd^\top \xkd)^{\frac{k^{i}}{2}}t_{\sigma}\left( \frac{\Sigma^{(i-1)}(\x,\xkd)}{(\x^\top \x)^{\frac{k^{i-1}}{2}}(\xkd^\top \xkd)^{\frac{k^{i-1}}{2}}} \right).
\end{equation*} 
The goal is now to show that whenever $\xk,\xkd \neq 0$, $\Sigma^{(i)}(\x,\xkd)$ can be expressed as a sum of the form
\begin{equation}
\label{eq:ntkgi}
    \begin{split}
        \Sigma^{(i)}(\xk,\xkd) 
        &= \sum_{j=0}^{\infty} (\xk^\top \xkd)^j \sum_{l=-{\infty}}^{\infty} \eta_{j,l}^{(i)} (\lvert\lvert \xk\rvert\rvert_2^2 \lvert\lvert \xkd\rvert\rvert_2^2)^{l/2}
    \end{split}
\end{equation} 
with $\eta_{j,l}^{(i)} \geq 0$. 
We prove by induction. In a first step, note that the case where $i=0$ holds trivially true.
Next, assume that Equation~\eqref{eq:ntkgi} holds true for $\Sigma^{(i-1)}$. Due to the above discussion, $t_{\sigma}$ can be expressed as a Taylor series around $0$ with positive coefficients $a_i^2$. Thus, 
\begin{equation*}
    \begin{split}
\Sigma^{(i)}(\x,\xkd) &= c_{\sigma} (\x^\top \x)^{\frac{k^{i}}{2}}(\xkd^\top \xkd)^{\frac{k^{i}}{2}} \sum_{m=0}^{\infty} a_m^2 \left(\frac{\Sigma^{(i-1)}(\x,\xkd)}{(\lvert\lvert \xk\rvert\rvert_2^2\lvert\lvert \xkd\rvert\rvert_2^2)^{\frac{k^{i-1}}{2}}} \right)^m \\
&= c_{\sigma} (\lvert\lvert \x\rvert\rvert_2^2)^{\frac{k^{i}}{2}}(\lvert\lvert \xkd\rvert\rvert_2^2)^{\frac{k^{i}}{2}} \sum_{m=0}^{\infty} a_m^2 \left(\frac{\sum_{j=0}^{\infty} (\x^\top \xkd)^j \sum_{l=-\infty}^{\infty} \eta_{j,l}^{(i-1)} (\lvert\lvert \xk \rvert\rvert_2^2\lvert\lvert \xkd\rvert\rvert_2^2)^{l/2}}{(\lvert\lvert \xk\rvert\rvert_2^2\lvert\lvert \xkd\rvert\rvert_2^2)^{\frac{k^{i-1}}{2}}} \right)^m \\
&=c_{\sigma} (\lvert\lvert \x\rvert\rvert_2^2)^{\frac{k^{i}}{2}}(\lvert\lvert \xkd\rvert\rvert_2^2)^{\frac{k^{i}}{2}} \sum_{m=0}^{\infty} a_m^2 \left(\sum_{j=0}^{\infty} (\x^\top \xkd)^j \sum_{l=-\infty}^{\infty} \eta_{j,l}^{(i-1)} (\lvert\lvert \xk \rvert\rvert_2^2\lvert\lvert \xkd\rvert\rvert_2^2)^{l/2 - k^{i-1}/2} \right)^m \\
&= \sum_{j=0}^{\infty} (\xk^\top \xkd)^j \sum_{l=-{\infty}}^{\infty} \eta_{j,l}^{(i)} (\lvert\lvert \xk\rvert\rvert_2^2 \lvert\lvert \xkd\rvert\rvert_2^2)^{l/2}.
    \end{split}
\end{equation*} 
In order to guarantee that the last equation holds true, we need to show that the above multi-sum converges absolutely. To see this, first of all note that
$\eta^{(i)}_{j,l}\geq 0$ because by assumption
$\eta_{j,l}^{(i)}\geq0$. Furthermore, given that $d>1$,
for any $\xk,\xkd$ we can find $\tilde{\xk},\tilde{x}'$ such that
$\lvert\lvert \xk \rvert\rvert_2^2 =\lvert\lvert \tilde{\xk}
\rvert\rvert_2^2$, $\lvert\lvert \xkd \rvert\rvert_2^2 =\lvert\lvert
\tilde{\xk}'\rvert\rvert_2^2$ and $\tilde{\xk}^\top \tilde{\xk}' = \lvert
\xk^\top \xkd \rvert$. Hence, the above multi-sum only consists of positive coefficients when evaluating at $\tilde{\xk},\tilde{\xk}'$ and thus converges absolutely which completes the induction.

Next, note
that any of the properties 1-4 from the above discussion also
hold true for $t_{\dot{\sigma}}$. Therefore, we can use exactly the same
argument for $\dot{\Sigma}^{(i)}$ to show that for any $\xk, \xkd \neq
0$,
\begin{equation*}
    \dot{\Sigma}^{(i)}(\xk,\xkd) = 
    \sum_{j=0}^{\infty} (\xk^\top \xkd)^j \sum_{l=-{\infty}}^{\infty} \dot{\eta}_{j,l}^{(i)} (\lvert\lvert \xk\rvert\rvert_2^2 \lvert\lvert \xkd\rvert\rvert_2^2)^{l/2}.
\end{equation*}
with $\dot{\eta}_{j,l}^{(i+1)} \geq 0$. Finally, we can conclude the proof because 
\begin{equation*}
     \kntk(\xk,\xkd) := \sum_{i=1}^{L+1} \Sigma^{(i-1)}(\xk,\xkd) \prod_{j=i}^{L+1} \dot{\Sigma}^{(j)}(\xk,\xkd)
\end{equation*}
and when using the same argument as used in the induction step above to show that the resulting multi-sum converges absolutely. 
\qed

\subsection{Additional lemmas}
\label{sec:AC_lemmas}

\begin{lemma}
\label{lm:AC_assumptions}
Any kernel which satisfies Assumption \Akone~and \Akthree~also satisfies Assumption \Ckone.
\end{lemma}

\begin{proof}
 The only point which does not follow immediately is to show that $\remainder $ is a continuous function. For this, write $g$ as a function of the variables $x,y,z$, i.e.
 \begin{equation*}
  \rotkerfunc(x,y,z) =  \sum_{j=0}^{\infty} \rotkerfunccoeffj{j}(x,y) z^j.
   \end{equation*}
For every $x,y$, define the function $\rotkerfunc_{x,y}(z) = g(x,y,z)$. Due to the series expansion, we can make use of the theory on the Taylor expansion which  implies that for any $x,y$, $\rotkerfunc_{x,y}$ is a smooth function in the interior of $\convergeset{\delta}$ (using the definition from Assumption \Akone). Hence, we can conclude that there exists a function $\remainder _{x,y}$ such that 
\begin{equation*}
 \rotkerfunc_{x,y}(z) =   \sum_{j=0}^{m} \rotkerfunccoeffj{j}(x,y) z^j + (z)^{m+1} \remainder _{x,y}(z)
\end{equation*}
In particular, the smoothness of $\rotkerfunc_{x,y}(z)$ implies that $\remainder _{x,y}$ is continuous in the interior of $\{z : (x,y,z) \in \convergeset{\delta} \}$. Next, define
$\remainder (x,y,z) = \remainder _{x,y}(z)$ and note that that the continuity of $\rotkerfunc$ implies that $\remainder $ is continuous everywhere except for the plane $z=0$. Finally, because $ \remainder _{x,y}(0)$ exists point wise, we conclude that $\remainder $ exists and is a continuous function in the interior of $\convergeset{\delta}$.
\end{proof}

\begin{lemma}
\label{lm:AC_rbf}
Any RBF kernel $k(\xk,\xkd) = h(\|\xk - \xkd\|_2^2)$ with $h$ locally analytic around $2$ satisfies Assumption \Ckone.  
\end{lemma}

\begin{proof}
  Because by assumption $h$ has a local Taylor series around $2$, we can write
\begin{equation*}
\label{eq:lemmaRBFass1}
    h(\| \xk - \xkd \|_2^2) = \sum_{j=0}^{\infty} h_j ~ (\| \xk - \xkd \|_2^2-2)^j,
\end{equation*}
which converges absolutely for any $\left|\| \xk - \xkd \|_2^2-2 \right| \leq \tilde{\delta}$ where $\tilde{\delta}>0$ is the convergence radius of the Taylor series approximation.  Next, using $\| \xk - \xkd \|_2^2 -2 = \| \xk \|_2^2 -1 + \|  \xkd \|_2^2 -1 -2 \xk^\top\xkd$, we can make use of the Binomial series, which gives
\begin{equation}
\label{eq:lemmaRBFass12}
    h(\| \xk - \xkd \|_2^2) = \sum_{j=0}^{\infty}  h_j \sum_{i=0}^j  \sum_{l=0}^i {j \choose i} {i \choose l} (\| \xk \|_2^2 -1)^l  (\|  \xkd \|_2^2 -1)^{i-l} (-2 \xk^\top\xkd)^{j-i}.
\end{equation}
The goal is now to show that this multi series converges absolutely. Whenever $d>1$, we can choose $\tilde{\xk},\tilde{\xk}'$ from the set of convergent points such that $ \| \tilde{\xk}  \|_2^2 -1>0$ , $\| \tilde{\xk}'  \|_2^2 -1 >0$ and $ \tilde{\xk}^\top\tilde{\xk}' <0$. As a result, we can see that for any $j$, the sum
\begin{equation*}
\sum_{i=0}^j  \sum_{l=0}^i {j \choose i} {i \choose l} (\|  \tilde{\xk} \|_2^2 -1)^l  (\|  \tilde{\xk}' \|_2^2 -1)^{i-l} (-2 \tilde{\xk}^\top\tilde{\xk}')^{j-i}
\end{equation*}
is a sum of non negative summands. Hence, we get that the sum
\begin{equation*}
\label{eq:commentF1}
\sum_{j=0}^{\infty}  \sum_{i=0}^j  \sum_{l=0}^i h_j {j \choose i} {i \choose l} (\|  \tilde{\xk} \|_2^2 -1)^l   (\|  \tilde{\xk}' \|_2^2 -1)^{i-l} (-2 \tilde{\xk}^\top\tilde{\xk}')^{j-i}
\end{equation*} 
converges absolutely. Thus, we can arbitrarily reorder the summands:
\begin{equation}
\label{eq:lemmaRBFass2}
\begin{split}
    h(\|  \tilde{\xk} - \tilde{\xk}' \|_2^2) &=  \sum_{j=0}^{\infty} (\tilde{\xk}^\top\tilde{\xk}')^{j}  \sum_{i=j}^{\infty}   \sum_{l=0}^{i-j}  h_{i} (-2)^j {i \choose j} {i-j \choose l} (\|  \tilde{\xk} \|_2^2 -1)^l  (\|  \tilde{\xk}' \|_2^2 -1)^{i-l-j} \\
    &=:\sum_{j=0}^{\infty}(\tilde{\xk}^\top\tilde{\xk}')^{j} g_j(\|  \tilde{\xk} \|_2^2, \|  \tilde{\xk}' \|_2^2)
    \end{split}
\end{equation} 
In particular, we obtain that the sum from Equation \eqref{eq:lemmaRBFass2} converges absolutely for any $\xk,\xkd$ with \\$ \| \tilde{\xk}  \|_2^2 -1> \left|\| \xk  \|_2^2 -1\right|$ , $\| \tilde{\xk}'  \|_2^2 -1 > \left|\| \xkd  \|_2^2 -1 \right|$ and $ -\tilde{\xk}^\top\tilde{\xk}' > \left|\xk^\top\xkd\right|$. In fact, we can always find $\delta,\delta' >0$ such that any  $(\xk,\xkd) \in \mathbb{R}^d \times \mathbb{R}^d$, with $(\| \xk\|_2^2, \| \xkd\|_2^2) \in  [1-\delta, 1+\delta] \times [1-\delta, 1+\delta]$ and $\xk^\top \xkd \in  [-\delta',\delta']$ satisfies these constraints and hence the sum from Equation \eqref{eq:lemmaRBFass12} converges absolutely. We can then conclude the proof when noting that the functions $g_i$ are implicitly defined using the Taylor series expansion. Therefore, we can conclude that $g_i$ are smooth functions in a neighborhood of $(1,1)$. The rest of the proof follows then trivially.
\end{proof}

%% file: main.bbl
\begin{thebibliography}{43}
\providecommand{\natexlab}[1]{#1}
\providecommand{\url}[1]{\texttt{#1}}
\expandafter\ifx\csname urlstyle\endcsname\relax
  \providecommand{\doi}[1]{doi: #1}\else
  \providecommand{\doi}{doi: \begingroup \urlstyle{rm}\Url}\fi

\bibitem[Arora et~al.(2019)Arora, Du, Hu, Li, Salakhutdinov, and Wang]{Arora19}
Arora, S., Du, S.~S., Hu, W., Li, Z., Salakhutdinov, R.~R., and Wang, R.
\newblock On exact computation with an infinitely wide neural net.
\newblock In \emph{Advances in Neural Information Processing Systems
  (NeurIPS)}, volume~32, pp.\  8141--8150, 2019.

\bibitem[Bartlett et~al.(2020)Bartlett, Long, Lugosi, and Tsigler]{Bartlett20}
Bartlett, P.~L., Long, P.~M., Lugosi, G., and Tsigler, A.
\newblock Benign overfitting in linear regression.
\newblock 117\penalty0 (48):\penalty0 30063--30070, 2020.

\bibitem[Belkin et~al.(2018)Belkin, Ma, and Mandal]{Belkin18c}
Belkin, M., Ma, S., and Mandal, S.
\newblock To understand deep learning we need to understand kernel learning.
\newblock In \emph{Proceedings of the International Conference on Machine
  Learning (ICML)}, volume~80, pp.\  541--549, 2018.

\bibitem[Belkin et~al.(2019)Belkin, Rakhlin, and Tsybakov]{Belkin19}
Belkin, M., Rakhlin, A., and Tsybakov, A.~B.
\newblock Does data interpolation contradict statistical optimality?
\newblock In \emph{Proceedings of the International Conference on Artificial
  Intelligence and Statistics (AISTATS)}, pp.\  1611--1619, 2019.

\bibitem[Berg et~al.(1984)Berg, Christensen, and Ressel]{Berg84}
Berg, C., Christensen, J. P.~R., and Ressel, P.
\newblock \emph{Harmonic Analysis on Semigroups}.
\newblock Springer, 1984.

\bibitem[Blumenthal \& Getoor(1960)Blumenthal and Getoor]{Blumenthal60}
Blumenthal, R.~M. and Getoor, R.~K.
\newblock Some theorems on stable processes.
\newblock \emph{Transactions of the American Mathematical Society}, 95\penalty0
  (2):\penalty0 263--273, 1960.

\bibitem[B{\"u}hlmann \& Van De~Geer(2011)B{\"u}hlmann and Van
  De~Geer]{Buhlmann11}
B{\"u}hlmann, P. and Van De~Geer, S.
\newblock \emph{Statistics for high-dimensional data: methods, theory and
  applications}.
\newblock Springer Science \& Business Media, 2011.

\bibitem[Caponnetto \& De~Vito(2007)Caponnetto and De~Vito]{Caponnetto07}
Caponnetto, A. and De~Vito, E.
\newblock Optimal rates for the regularized least-squares algorithm.
\newblock \emph{Foundations of Computational Mathematics}, 7\penalty0
  (3):\penalty0 331--368, 2007.

\bibitem[Chen et~al.(2017)Chen, Stern, Wainwright, and Jordan]{Chen17}
Chen, J., Stern, M., Wainwright, M.~J., and Jordan, M.~I.
\newblock Kernel feature selection via conditional covariance minimization.
\newblock In \emph{Advances in Neural Information Processing Systems
  (NeurIPS)}, pp.\  6946--6955, 2017.

\bibitem[Christmann et~al.(2007)Christmann, Steinwart, et~al.]{Christmann07}
Christmann, A., Steinwart, I., et~al.
\newblock Consistency and robustness of kernel-based regression in convex risk
  minimization.
\newblock \emph{Bernoulli}, 13\penalty0 (3):\penalty0 799--819, 2007.

\bibitem[Daniely et~al.(2016)Daniely, Frostig, and Singer]{Daniely16}
Daniely, A., Frostig, R., and Singer, Y.
\newblock Toward deeper understanding of neural networks: The power of
  initialization and a dual view on expressivity.
\newblock In \emph{Advances in Neural Information Processing Systems
  (NeurIPS)}, volume~29, 2016.

\bibitem[Dobriban \& Wager(2018)Dobriban and Wager]{Dobriban15}
Dobriban, E. and Wager, S.
\newblock High-dimensional asymptotics of prediction: Ridge regression and
  classification.
\newblock \emph{The Annals of Statistics}, 46\penalty0 (1):\penalty0 247 --
  279, 2018.

\bibitem[Driscoll \& Fornberg(2002)Driscoll and Fornberg]{Driscoll02}
Driscoll, T.~A. and Fornberg, B.
\newblock Interpolation in the limit of increasingly flat radial basis
  functions.
\newblock \emph{Computers \& Mathematics with Applications}, 43\penalty0
  (3-5):\penalty0 413--422, 2002.

\bibitem[Dua \& Graff(2017)Dua and Graff]{Dua17}
Dua, D. and Graff, C.
\newblock {UCI} machine learning repository, 2017.
\newblock URL \url{http://archive.ics.uci.edu/ml}.

\bibitem[El~Karoui et~al.(2010)]{ElKaroui10}
El~Karoui, N. et~al.
\newblock The spectrum of kernel random matrices.
\newblock \emph{Annals of Statistics}, 38\penalty0 (1):\penalty0 1--50, 2010.

\bibitem[Geifman et~al.(2020)Geifman, Yadav, Kasten, Galun, Jacobs, and
  Basri]{Geifman20}
Geifman, A., Yadav, A., Kasten, Y., Galun, M., Jacobs, D., and Basri, R.
\newblock On the similarity between the laplace and neural tangent kernels.
\newblock \emph{arXiv preprint arXiv:2007.01580}, 2020.

\bibitem[Ghorbani et~al.(2019)Ghorbani, Mei, Misiakiewicz, and
  Montanari]{Ghorbani19}
Ghorbani, B., Mei, S., Misiakiewicz, T., and Montanari, A.
\newblock Linearized two-layers neural networks in high dimension.
\newblock \emph{arXiv preprint arXiv:1904.12191}, 2019.

\bibitem[Ghorbani et~al.(2020)Ghorbani, Mei, Misiakiewicz, and
  Montanari]{Ghorbani20}
Ghorbani, B., Mei, S., Misiakiewicz, T., and Montanari, A.
\newblock When do neural networks outperform kernel methods?
\newblock \emph{arXiv preprint arXiv:2006.13409}, 2020.

\bibitem[Hastie et~al.(2019)Hastie, Montanari, Rosset, and
  Tibshirani]{Hastie19}
Hastie, T., Montanari, A., Rosset, S., and Tibshirani, R.~J.
\newblock Surprises in high-dimensional ridgeless least squares interpolation.
\newblock \emph{arXiv preprint arXiv:1903.08560}, 2019.

\bibitem[Ikramov \& Savel'eva(2000)Ikramov and Savel'eva]{Ikramov00}
Ikramov, K. and Savel'eva, N.
\newblock Conditionally definite matrices.
\newblock \emph{Journal of Mathematical Sciences}, 98:\penalty0 1--50, 2000.

\bibitem[Jacot et~al.(2018)Jacot, Gabriel, and Hongler]{Jacot20}
Jacot, A., Gabriel, F., and Hongler, C.
\newblock Neural tangent kernel: Convergence and generalization in neural
  networks.
\newblock \emph{arXiv preprint arXiv:1806.07572}, 2018.

\bibitem[Larsson \& Fornberg(2005)Larsson and Fornberg]{Larsson05}
Larsson, E. and Fornberg, B.
\newblock Theoretical and computational aspects of multivariate interpolation
  with increasingly flat radial basis functions.
\newblock \emph{Computers \& Mathematics with Applications}, 49\penalty0
  (1):\penalty0 103--130, 2005.

\bibitem[Ledoux(2001)]{Ledoux2001}
Ledoux, M.
\newblock \emph{The Concentration of Measure Phenomenon}.
\newblock Mathematical surveys and monographs. American Mathematical Society,
  2001.
\newblock ISBN 9780821837924.

\bibitem[Lee et~al.(2018)Lee, Sohl-dickstein, Pennington, Novak, Schoenholz,
  and Bahri]{Lee18}
Lee, J., Sohl-dickstein, J., Pennington, J., Novak, R., Schoenholz, S., and
  Bahri, Y.
\newblock Deep neural networks as gaussian processes.
\newblock In \emph{Proceedings of the International Conference on Learning
  Representations (ICLR)}, 2018.

\bibitem[Lee et~al.(2014)Lee, Micchelli, and Yoon]{Lee2014}
Lee, Y.~J., Micchelli, C.~A., and Yoon, J.
\newblock On convergence of flat multivariate interpolation by translation
  kernels with finite smoothness.
\newblock \emph{Constructive Approximation}, 40\penalty0 (1):\penalty0 37--60,
  2014.

\bibitem[Li et~al.(2018)Li, Cheng, Wang, Morstatter, Trevino, Tang, and
  Liu]{Li18}
Li, J., Cheng, K., Wang, S., Morstatter, F., Trevino, R.~P., Tang, J., and Liu,
  H.
\newblock Feature selection: A data perspective.
\newblock \emph{ACM Computing Surveys (CSUR)}, 50\penalty0 (6):\penalty0 94,
  2018.

\bibitem[Liang et~al.(2020{\natexlab{a}})Liang, Rakhlin, and Zhai]{Liang20b}
Liang, T., Rakhlin, A., and Zhai, X.
\newblock On the multiple descent of minimum-norm interpolants and restricted
  lower isometry of kernels.
\newblock In \emph{Proceedings of the Conference on Learning Theory (COLT)},
  2020{\natexlab{a}}.

\bibitem[Liang et~al.(2020{\natexlab{b}})Liang, Rakhlin, et~al.]{Liang20a}
Liang, T., Rakhlin, A., et~al.
\newblock Just interpolate: Kernel “ridgeless” regression can generalize.
\newblock \emph{Annals of Statistics}, 48\penalty0 (3):\penalty0 1329--1347,
  2020{\natexlab{b}}.

\bibitem[Liu et~al.(2020)Liu, Liao, and Suykens]{Liu20}
Liu, F., Liao, Z., and Suykens, J.~A.
\newblock Kernel regression in high dimension: Refined analysis beyond double
  descent.
\newblock \emph{arXiv preprint arXiv:2010.02681}, 2020.

\bibitem[MacKay(1996)]{MacKay96}
MacKay, D. J.~C.
\newblock \emph{Bayesian Non-Linear Modeling for the Prediction Competition},
  pp.\  221--234.
\newblock Springer Netherlands, Dordrecht, 1996.
\newblock ISBN 978-94-015-8729-7.

\bibitem[Mei \& Montanari(2019)Mei and Montanari]{Mei19}
Mei, S. and Montanari, A.
\newblock The generalization error of random features regression: Precise
  asymptotics and double descent curve.
\newblock \emph{arXiv preprint arXiv:1908.05355}, 2019.

\bibitem[Mei et~al.(2021{\natexlab{a}})Mei, Misiakiewicz, and Montanari]{Mei21}
Mei, S., Misiakiewicz, T., and Montanari, A.
\newblock Generalization error of random features and kernel methods:
  hypercontractivity and kernel matrix concentration.
\newblock \emph{arXiv preprint arXiv:2101.10588}, 2021{\natexlab{a}}.

\bibitem[Mei et~al.(2021{\natexlab{b}})Mei, Misiakiewicz, and
  Montanari]{Mei21b}
Mei, S., Misiakiewicz, T., and Montanari, A.
\newblock Learning with invariances in random features and kernel models,
  2021{\natexlab{b}}.

\bibitem[Milman \& Schechtman(1986)Milman and Schechtman]{Milman86}
Milman, D.~V. and Schechtman, G.
\newblock Asymptotic theory of finite dimensional normed spaces.
\newblock 1986.

\bibitem[Muthukumar et~al.(2020)Muthukumar, Narang, Subramanian, Belkin, Hsu,
  and Sahai]{Muthukumar20}
Muthukumar, V., Narang, A., Subramanian, V., Belkin, M., Hsu, D., and Sahai, A.
\newblock Classification vs regression in overparameterized regimes: Does the
  loss function matter?
\newblock \emph{arXiv preprint arXiv:2005.08054}, 2020.

\bibitem[Neal(1996)]{Neal96}
Neal, R.~M.
\newblock \emph{Bayesian Learning for Neural Networks}.
\newblock Springer-Verlag, 1996.
\newblock ISBN 0387947248.

\bibitem[Novak et~al.(2019)Novak, Xiao, Bahri, Lee, Yang, Abolafia, Pennington,
  and Sohl-dickstein]{Novak19}
Novak, R., Xiao, L., Bahri, Y., Lee, J., Yang, G., Abolafia, D.~A., Pennington,
  J., and Sohl-dickstein, J.
\newblock Bayesian deep convolutional networks with many channels are gaussian
  processes.
\newblock In \emph{International Conference on Learning Representations}, 2019.
\newblock URL \url{https://openreview.net/forum?id=B1g30j0qF7}.

\bibitem[Schaback(2005)]{Schaback05}
Schaback, R.
\newblock Multivariate interpolation by polynomials and radial basis functions.
\newblock \emph{Constructive Approximation}, 21\penalty0 (3):\penalty0
  293--317, 2005.

\bibitem[Shankar et~al.(2020)Shankar, Fang, Guo, Fridovich-Keil, Ragan-Kelley,
  Schmidt, and Recht]{Shankar20}
Shankar, V., Fang, A., Guo, W., Fridovich-Keil, S., Ragan-Kelley, J., Schmidt,
  L., and Recht, B.
\newblock Neural kernels without tangents.
\newblock In \emph{Proceedings of the International Conference on Machine
  Learning (ICML)}, volume 119, pp.\  8614--, 2020.

\bibitem[Smola \& Sch{\"o}lkopf(1998)Smola and Sch{\"o}lkopf]{Smola98}
Smola, A.~J. and Sch{\"o}lkopf, B.
\newblock \emph{Learning with kernels}, volume~4.
\newblock Citeseer, 1998.

\bibitem[Wahba(1990)]{Wahba90}
Wahba, G.
\newblock \emph{Spline models for observational data}.
\newblock SIAM, 1990.

\bibitem[Wainwright(2019)]{Wainwright19}
Wainwright, M.
\newblock \emph{High-Dimensional Statistics: A Non-Asymptotic Viewpoint}.
\newblock Cambridge University Press, 2019.
\newblock ISBN 9781108498029.

\bibitem[Wyner et~al.(2017)Wyner, Olson, Bleich, and Mease]{Wyner17}
Wyner, A.~J., Olson, M., Bleich, J., and Mease, D.
\newblock Explaining the success of adaboost and random forests as
  interpolating classifiers.
\newblock \emph{Journal of Machine Learning Research (JMLR)}, 18\penalty0
  (1):\penalty0 1558--1590, 2017.

\end{thebibliography}
